\documentclass[a4paper,11pt]{article}
\usepackage[usenames,dvipsnames,svgnames,table,xcdraw]{xcolor}
\usepackage{amsmath,amsthm,thmtools} 
\usepackage{mfpaperstuff3}
\DeclareRobustCommand\Iff{\;\Longleftrightarrow\;}
\usepackage{mfabacus_edited}
\usepackage{multirow}
\usepackage{ifthen}
\usepackage{booktabs}

\usetikzlibrary{calc,decorations.markings}

\DeclarePairedDelimiter{\ceil}{\lceil}{\rceil} 
\DeclarePairedDelimiter{\floor}{\lfloor}{\rfloor} 
\DeclarePairedDelimiter{\abs}{\lvert}{\rvert} 
\makeatletter
\newsavebox{\@brx}
\newcommand{\llangle}[1][]{\savebox{\@brx}{\(\m@th{#1\langle}\)}%
  \mathopen{\copy\@brx\kern-0.7\wd\@brx\usebox{\@brx}}}
\newcommand{\rrangle}[1][]{\savebox{\@brx}{\(\m@th{#1\rangle}\)}%
  \mathclose{\copy\@brx\kern-0.7\wd\@brx\usebox{\@brx}}}
\makeatother

\DeclareFontFamily{U}{mathx}{}
\DeclareFontShape{U}{mathx}{m}{n}{<-> mathx10}{}
\DeclareSymbolFont{mathx}{U}{mathx}{m}{n}
\DeclareMathAccent{\widehat}{0}{mathx}{"70}
\DeclareMathAccent{\widecheck}{0}{mathx}{"71}

\definecolor{darkazure}{HTML}{0059b3}
\definecolor{azure}{HTML}{007fff}
\definecolor{paleazure}{HTML}{a6d2ff}
\definecolor{vdarkgold}{HTML}{807100}
\definecolor{darkergold}{HTML}{b39e00}
\definecolor{darkgold}{HTML}{f2d400}
\definecolor{gold}{HTML}{ffdf00}
\definecolor{palegold}{HTML}{fff7bf}
\definecolor{beaverred}{HTML}{cc0000}
\definecolor{palered}{HTML}{f68c67}
\definecolor{mygrey}{HTML}{e6e6e6}
\definecolor{darkgrey}{HTML}{bbbbbb}
\definecolor{mygreen}{HTML}{1b9e77}
\definecolor{palegreen}{HTML}{6bc0a7}
\definecolor{mintgreen}{rgb}{0.6, 1.0, 0.6}
\definecolor{goldazure}{HTML}{84b441}
\newcommand{\colourA}{azure}

\newcommand{\palecolourA}{paleazure}
\newcommand{\colourB}{gold}
\newcommand{\palecolourB}{palegold}
\newcommand{\darkcolourB}{darkergold}
\newcommand{\colourC}{beaverred}

\newcommand{\palecolourC}{palered}
\newcommand{\colourGreen}{goldazure}

\newcommand{\YcA}{\Yfillcolour{\colourA}}
\newcommand{\YcB}{\Yfillcolour{\colourB}}

\newcommand{\YpA}{\Yfillcolour{\palecolourA}}
\newcommand{\YpB}{\Yfillcolour{\palecolourB}}
\newcommand{\YpC}{\Yfillcolour{\palecolourC}}

\newcommand{\YdB}{\Yfillcolour{\darkcolourB}}

\newcommand{\YvdB}{\Yfillcolour{vdarkgold}}

\newcommand{\YcG}{\Yfillcolour{\colourGreen}}

\renewcommand\iff{if and only if\xspace}

\crefname{enumi}{}{} 

\makeatletter    
\newcommand{\prelistcommand}{\nobreak\leavevmode\@nobreaktrue\vspace*{0.4em}}
\makeatother
\SetEnumitemKey{beginthm}{before=\prelistcommand}   

\crefformat{section}{\S#2#1#3}
\crefrangeformat{section}{\S#3#1#4--#5#2#6}
\crefmultiformat{section}{\S#2#1#3}{ and~\S#2#1#3}{, \S#2#1#3}{, and~\S#2#1#3}
\crefrangemultiformat{section}{\S#2#1#3}{ and~\S#2#1#3}{, \S#2#1#3}{, and~\S#2#1#3}

\usepackage{letltxmacro}
\makeatletter
\let\oldr@@t\r@@t
\def\r@@t#1#2{%
\setbox0=\hbox{$\oldr@@t#1{#2\,}$}\dimen0=\ht0
\advance\dimen0-0.2\ht0
\setbox2=\hbox{\vrule height\ht0 depth -\dimen0}%
{\box0\lower0.4pt\box2}}
\LetLtxMacro{\oldsqrt}{\sqrt}
\renewcommand*{\sqrt}[2][\ ]{\oldsqrt[#1]{#2}}
\makeatother

\usepackage{ subcaption }

\newcommand\qt{\sqrt2}

\newcommand\Ga\Gamma
\newcommand\ze\zeta
\newcommand\eps\epsilon 
\newcommand\bep{\bar\ep}
\newcommand\beps{\bep}

\renewcommand{\phi}{\varphi}
\renewcommand{\emptyset}{\varnothing}

\newcommand\ip[2]{\left(#1:#2\right)} 

\renewcommand\set[1]{\left\{#1\right\}}
\NewDocumentCommand\setbuild{s m m}{%
    \IfBooleanTF#1%
    {\ensuremath{\left\{\, #2 \, \middle| \, #3 \,\right\}}}%
    {\ensuremath{\{\, #2 \, \mid \, #3 \,\}}}%
}

\renewcommand\rt[1]{\rotatebox{90}{$#1$}}

\newcommand{\len}[1]{l(#1)} 
\newcommand{\dbl}[1]{\mathrm{dbl}(#1)} 
\newcommand\reg{^{\mathrm{reg}}} 
\newcommand\lad[1]{\call_{#1}} 

\newcommand\bip{bipartition\xspace}
\newcommand\bips{bipartitions\xspace}

\newcommand\rh[2]{R(#1,#2)} 
\newcommand\rb[2]{S(#1,#2)} 

\newcommand\fbc{$4$-bar-core\xspace}
\newcommand\fbcs{\fbc{}s\xspace}
\newcommand{\bipf}[1]{\underline{\smash{#1}}} 
\newcommand{\bipmu}{\bipf{\mu}}
\newcommand{\bipla}{\bipf{\la}}
\newcommand{\bipnu}{\bipf{\nu}}
\newcommand{\bipkap}{\bipf{\ka}}

\newcommand{\partsm}{\setminus} 
\newcommand{\ydsm}{/} 
\newcommand{\nodesm}{/} 
\newcommand{\hooksm}{\setminus} 

\newcommand\colrem[1]{\check{#1}} 
\newcommand\ceta{\colrem\eta}
\newcommand\ctheta{\colrem\theta}
\newcommand\czeta{\colrem\ze}

\newcommand\cor[1]{\ka_{#1}}
\newcommand{\twoc}[1]{\cor{#1}}
\newcommand\fcor[1]{\bar{\ka}_{#1}}
\newcommand{\btwoc}[1]{\fcor{#1}}
\newcommand{\biptwoc}[2]{\bipkap_{#1, #2}}
\newcommand{\corandquot}[3]{[#1; (#2,#3)]}
\newcommand{\corandbipquot}[2]{[#1; #2]}

\newcommand{\lenabove}[2]{r}

\newcommand{\spe}[1]{\chi(#1)} 
\newcommand\jms[1]{\varphi(#1)} 
\newcommand\jmsm[1]{\operatorname{D}^{#1}} 
\newcommand\spem[1]{\operatorname{S}^{#1}}

\newcommand\dn[2]{[#1:\jms{#2}]} 

\newcommand\hsss{\hat{\mathfrak{S}}_}
\newcommand\spnorig[1]{\lan#1\ran} 
\newcommand\spn[1]{\llangle#1\rrangle} 

\newcommand\br[1]{\widebar{#1}} 
\newcommand\bspe[1]{\br{\spe{#1}}}
\newcommand\bspnorig[1]{\br{\spnorig{#1}}}
\newcommand\bspn[1]{\br{\spn{#1}}}

\newcommand\spr{spin-residue\xspace}
\newcommand\sprs{\spr{}s\xspace}
\newcommand\nd[1]{$#1$-spin-node\xspace}
\newcommand\nds[1]{\nd#1{}s\xspace}
\newcommand\spre{spin-removable\xspace}
\newcommand\sprm{\spre node\xspace}
\newcommand\sprms{\sprm{}s\xspace}
\newcommand\spad{spin-addable\xspace}
\newcommand\spam{\spad node\xspace}
\newcommand\spams{\spam{}s\xspace}
\newcommand\esprm[1]{$#1$-spin-removable node\xspace}
\newcommand\esprms[1]{$#1$-spin-removable nodes\xspace}

\newcommand\espams[1]{$#1$-spin-addable nodes\xspace}

\newcommand{\fsas}{\(4\)-stepped-and-semi\-congruent\xspace}
\newcommand{\fstepped}{\(4\)-stepped\xspace}
\newcommand{\fsemic}{\(4\)-semicongruent\xspace}
\newcommand\prp{proportional\xspace}

\newcommand\lx[1]{{#1}_\circ} 

\newcommand{\runnerswap}[2]{S_{#1}^{(#2)}} 
\newcommand\rsf{runner-swapping function\xspace}

\newcommand{\quotred}[2]{R_{#1}^{(#2)}} 
\newcommand\qrf{quotient-redistributing function\xspace}

\newcommand{\tlad}[1]{t(#1)} 
\newcommand{\flad}[1]{f(#1)} 
\newcommand{\lende}{\abs{\de}} 

\newcommand{\ulob}{\pi} 
\newcommand{\hatal}{\al\!\hooksm\!\ulob}

\newcommand{\addsondiag}[2]{\iota_{#1}(#2)} 
\newcommand{\netaddables}[2]{\Delta_{#1}#2} 
\newcommand{\diffondiag}[3]{#2-_{#1}#3} 
\newcommand\swp[2]{#1^{\ast#2}}

\newcommand{\spindecoration}[1]{\mathring{#1}}
\newcommand{\spinaddsondiag}[2]{\spindecoration{\iota}_{#1}(#2)} 
\newcommand{\netspinaddables}[2]{\spindecoration{\Delta}_{#1}#2} 
\newcommand{\bdiffondiag}[3]{#2 \ominus_{#1} #3} 
\newcommand\bswp[2]{#1^{\oast#2}}

\newcommand{\kom}[2]{k(#1\ydsm #2)} 

\newcommand{\addtoget}{\rightharpoonup}
\newcommand{\removetoget}{\leftharpoonup}
\newcommand{\interm}[2]{\mathcal{I}(#1,#2)} 
\newcommand{\intermzero}[2]{\mathcal{I}_{0}(#1,#2)}
\newcommand{\intermone}[2]{\mathcal{I}_{1}(#1,#2)}
\newcommand{\intermsize}[2]{A_{#1,#2}}
\newcommand\iet{\intermzero\eta\theta}
\newcommand\ciet{\intermzero\ceta\ctheta}

\newcommand{\coldiff}[1]{\Delta_{#1}}
\newcommand{\cdifseq}{\Delta}

\newcommand\isummand[3]{b^{#1}_{#2,#3}}
\newcommand\isum[2]{B_{#1,#2}}

\newcommand\wht{\Yfillcolour{white}}

\newcommand{\Ypossible}{\Ylinestyle{densely dotted}\Ynodecolour{gray}\Ylinecolour{gray}}
\newcommand{\Ydefault}{\Ylinecolour{black}\Ynodecolor{black}\Yfillcolour{white}\Ylinestyle{}}

\def\ladderwidth{1pt}
\def\ladderstep{2.6pt}
\tikzset{
ladder/.style={decorate,decoration={
      markings,
      mark=between positions {1/#1/2} and {-1/#1/2} step {1/#1} with {
        \draw[\colourC]
        (0,-\ladderwidth) -- (0,\ladderwidth)
        (\ladderstep,-\ladderwidth) -- (-\ladderstep,-\ladderwidth)
        (\ladderstep,\ladderwidth) -- (-\ladderstep,\ladderwidth);
        }
    },
    },
  ladder auto/.style={to path={
      let \p1=($(\tikztostart) - (\tikztotarget)$), \n1={veclen(\x1,\y1)} in
      \pgfextra{
        \pgfmathsetmacro{\bars}{int(\n1/\ladderstep/2)+1}
        \pgfinterruptpath
        \draw[ladder=\bars] (\tikztostart) -- (\tikztotarget);
        \endpgfinterruptpath
      }
    },
  },
}

\makeatletter
\newcommand\customsize{\@setfontsize\customsize{8.2}{9.5}}
\def\journalinfo#1{\thanks{\customsize#1}}
\makeatother

\begin{document}

\title{Spin characters of the symmetric group which are proportional to linear characters in characteristic $2$}


\author{{\Large Matthew Fayers}\\
{\normalsize Queen Mary University of London}\\
{\normalsize\texttt{\normalsize m.fayers@qmul.ac.uk}}\\[12pt]
{\Large Eoghan McDowell}\\
{\normalsize Okinawa Institute of Science and Technology}\\
{\normalsize (current: University of Bristol)}\\
\texttt{\normalsize eoghan.mcdowell@bristol.ac.uk}}

\renewcommand\auth{Matthew Fayers \& Eoghan McDowell}
\runninghead{Spin characters proportional to linear characters}

\toptitle

\journalinfo{This is the accepted manuscript for an article published in \emph{Annals of Representation Theory} 2(1) (2025), pp.~37--83, \href{https://art.centre-mersenne.org/articles/10.5802/art.21}{doi:10.5802/art.21}.}

\begin{abstract}
For a finite group, it is interesting to determine when two ordinary irreducible representations have the same $p$-modular reduction; that is, when two rows of the decomposition matrix in characteristic $p$ are equal, or equivalently when the corresponding $p$-modular Brauer characters are the same. We complete this task for the double covers of the symmetric group when $p=2$, by determining when the $2$-modular reduction of an irreducible spin representation coincides with a $2$-modular Specht module. In fact, we obtain a more general result: we determine when an irreducible spin representation has \(2\)-modular Brauer character proportional to that of a Specht module. In the course of the proof, we use induction and restriction functors to construct a function on generalised characters which has the effect of swapping runners in abacus displays for the labelling partitions.
\end{abstract}

\setcounter{tocdepth}{2} 
\tableofcontents

\section{Introduction}

When an irreducible spin representation of the symmetric group is reduced modulo \(2\), it becomes a representation of the symmetric group in the usual sense, and it is an interesting problem to say exactly which representation is obtained.
In this paper, we determine when the result is ``a multiple of a Specht module'', in the sense of their images in the Grothendieck group or equivalently of their Brauer characters.
That is, we determine when the \(2\)-modular reduction of an irreducible spin representation has Brauer character which is a multiple of that of a Specht module.

We show that the following pair of conditions on the labelling partition are necessary and sufficient for a spin representation to have a \(2\)-modular Brauer character proportional to that of a Specht module.
For a strict partition \(\al\), we say that:
\begin{itemize}
\item
\(\al\) is \emph{\fstepped} if for every part $\al_r>4$ the integer $\al_r-4$ is also a part of $\al$; and
\item
\(\al\) is \emph{\fsemic} if the odd parts of $\al$ are congruent modulo $4$.
\end{itemize}
If \(\al\) satisfies both conditions, we abbreviate this by saying that \(\al\) is \fsas.

To obtain the partition labelling the corresponding Specht module, observe that if $\al$ is \fsas then there are unique (up to reordering) $2$-core partitions $\si$ and $\tau$ such that the even parts of $\al$ comprise the partition $2(\si+\tau)$ (see \Cref{subsec:label_descs} below for details).
We define $\lx\al$ to be the partition of $|\al|$ (unique up to conjugation) whose $2$-quotient has the form $(\si,\tau)$.

We denote by \(\spe\la\) the character of the Specht module labelled by a partition \(\la\), and denote by \(\lan\al\ran\) the character of an irreducible spin representation labelled by a strict partition \(\al\) (when $\al$ has an odd number of even parts, there is actually an associate pair of characters \(\lan\al\ran_+\) and \(\lan\al\ran_-\) labelled by \(\al\), but they have the same $2$-modular reduction, so we write $\lan\al\ran$ to denote either of them).
For any ordinary character $\phi$, we write $\br\phi$ for the Brauer character of the \(2\)-modular reduction of the corresponding representation. 

Our main result is the following.

\begin{thm}\label{main}
Let \(\la\) be a partition of \(n\) and let \(\al\) be a strict partition of \(n\).
Then $\bspnorig{\al}$ is proportional to $\bspe\la$ \iff $\al$ is \fsas and $\la \in \{\lx\al, \lx\al'\}$.
In this case, $\bspnorig{\al} = 2^{\lfloor e/2\rfloor} \bspe\la$, where \(e\) is the number of even parts of \(\al\).
\end{thm}

Equality of Brauer characters occurs when \(e \in \set{0,1}\), giving the following corollary.
The \emph{double} of a strict partition \(\al\), denoted \(\dbl{\al}\), is the partition obtained by replacing every odd part \(2k+1\) with parts \(k+1\) and \(k\), and every even part \(2k\) with two parts equal to \(k\). Cores and quotients are defined in \Cref{blocksec,subsec:abacus_and_quotient}.
Given partitions \(\mu\) and \(\nu\), we write \(\mu \sqcup \nu\) for the partition obtained by combining all parts of \(\mu\) and \(\nu\) and reordering as necessary.

\begin{cory}\label{mainequal}
Let \(\la\) be a partition of \(n\) and let \(\al\) be a strict partition of \(n\).
Then $\br{\lan\al\ran} = \bspe\la$ \iff:
\begin{itemize}
    \item
\(\al = \kappa \sqcup \eta\), where \(\kappa\) is of the form \((\ldots,9,5,1)\) or \((\ldots, 11,7,3)\) and \(\eta \in \set{\emptyset, (2), (4)}\), and
    \item 
\(\la\) has \(2\)-core \(\dbl{\kappa}\) and \(2\)-quotient \((\sigma,\tau)\), where \(\sigma,\tau \in \set{\emptyset,(1)}\) and \(2(\sigma_1 + \tau_1) = \eta_1\).
\end{itemize}
\end{cory}

\cref{mainequal} implies in particular that equality of Brauer characters can occur only when the characters lie in blocks of weight \(0\), \(1\) or \(2\). In the cases of weights \(0\) and \(1\), the equality of Brauer characters in fact implies isomorphism of representations, since in this case the modular reductions remain simple.

We wish to highlight the ``runner-swapping function'' (\Cref{runnerswapping}) we introduce in the course of the proof as being of independent interest.
This function acts (up to a sign) on Specht modules by swapping a pair of runners in an abacus display for the labelling partition, or equivalently by adding all addable nodes and removing all removable nodes of a particular residue in the Young diagram (see \Cref{rsfex}, \Cref{runner-swap_application_on_charsgeneral} and \Cref{remark:runner-swapping_functor_odd_p}).
This action can be understood as the affine Weyl group acting on the set of partitions (see \cite{mfgencore}). Our approach exhibits this action as a combination of induction and restriction functors. Our expression is new; although expressing the action via induction and restriction functors has been achieved before by \cite{cr}, we believe our expression is simpler, and allows us to also compute its action also on spin representations. 

\subsection{Discussion on main problem}

Fix a prime $p$.
For an ordinary character \(\chi\) of a double cover \(\hsss n\) of the symmetric group, the Brauer character \(\br\chi\) of the \(p\)-modular reduction of the representation affording \(\chi\) can be found by restricting \(\chi\) to the \(p\)-regular conjugacy classes of \(\hsss n\).
We are interested in the question of when two ordinary irreducible characters $\chi\neq\psi$ of \(\hsss n\) can give $\br\chi=\br\psi$, or equivalently whether an ordinary irreducible character is determined by its values on $p$-regular conjugacy classes.

Our question naturally splits into three cases, two of which are (almost) already answered.
We call a representation or character of \(\hsss n\) \emph{spin} if the central element of order \(2\) acts non-trivially, and call it \emph{linear} otherwise (see \Cref{spincharsec}) (unfortunately the term \emph{linear character} can also be used to mean a character of degree \(1\); we will not use it in that sense in this paper).

\begin{itemize}
\item
The case where $\chi$ and $\psi$ are both linear characters reduces to the corresponding question for the symmetric group, which was answered by Wildon \cite[Theorem 1.1.1]{wildon2008distinctrows}: if $\la,\mu\in\scrp(n)$ are distinct, then $\bspe\la=\bspe\mu$ \iff $p=2$ and $\mu=\la'$ is the conjugate of \(\la\) (see \Cref{subsec:yd}).

\item
In the case where $\chi$ and $\psi$ are both spin characters and \(p\neq 3\), this question has been answered by the second author \cite[Theorem 1.2]{mcdowell}: $\br\chi$ and $\br\psi$ are equal \iff $\chi$ and $\psi$ are an associate pair $\lan\al\ran_+$ and $\lan\al\ran_-$, where $\al$ is a strict partition with an odd number of even parts and a part divisible by $p$.
\end{itemize}

It remains to consider the case where exactly one of $\chi$ and $\psi$ is a spin character. In this case, if $p$ is odd, it is never the case that $\br\chi=\br\psi$: the composition factors of a $p$-modular reduction of an ordinary spin representation are always spin representations, so (the modules affording) $\br\chi$ and $\br\psi$ will have no composition factors in common. Thus we restrict attention to the case $p=2$ from now on.

It will be helpful to interpret our main theorem also in terms of decomposition numbers (see \Cref{spincharsec} for details).
It is a standard fact from modular representation theory that if $\chi$ and $\psi$ are ordinary characters, then $\br\chi=\br\psi$ \iff \([\chi:\phi]=[\psi:\phi]\) for all modular irreducible characters \(\phi\).
Thus our question can be phrased as asking when two rows of the decomposition matrix are equal.

\begin{eg}
The decomposition matrix of $\hsss4$ in characteristic $2$ is given as follows, with rows corresponding to ordinary characters and columns to $2$-modular irreducibles.
\[
\begin{array}{r|cc|}\\
&\rt{(4)}
&\rt{(3,1)}
\\\hline
\spe4&1&\cdot\\
\spe{3,1}&1&1\\
\spe{2^2}&\cdot&1\\
\spe{2,1^2}&1&1\\
\spe{1^4}&1&\cdot\\
\lan4\ran_+&\cdot&1\\\
\lan4\ran_-&\cdot&1\\\
\lan3,1\ran&2&1\\\hline
\end{array}
\]
We can see that if $\chi$ and $\psi$ are distinct ordinary irreducible characters, then $\br\chi=\br\psi$ \iff $\{\chi,\psi\}=\{\spe4,\spe{1^4}\}$, $\{\chi,\psi\}=\{\spe{3,1},\spe{2,1^2}\}$, or $\{\chi,\psi\}\subset\{\spe{2^2},\lan4\ran_+,\lan4\ran_-\}$.
Of these:
\begin{enumerate}[(i)]
    \item 
the pairs of linear characters \(\{\spe4,\spe{1^4}\}\) and \(\{\spe{3,1},\spe{2,1^2}\}\) are conjugate pairs, as predicted by \cite{wildon2008distinctrows};
    \item
the pair of spin characters \(\{\lan4\ran_+, \lan4\ran_-\}\) is an associate pair labelled by a partition with a part divisible by \(2\), as predicted by \cite{mcdowell};
    \item
the equality of \(\bspe{2^2}\) with \(\br{\lan4\ran_+}\) and \(\br{\lan4\ran_-}\) is predicted by our main theorem.
\end{enumerate}
\end{eg}

In fact, it turns out to be rare to have $\br\chi=\br\psi$ when only one of $\chi$ and $\psi$ is a spin character, because the spin characters typically have much larger degrees than the characters of the Specht modules. So we address a more general question in this paper: when are $\br\chi$ and $\br\psi$ proportional to each other? In other words, when is one row of the decomposition matrix proportional to another? Given two characters $\chi$ and $\psi$ whose $2$-modular reductions are proportional, we can easily find the constant of proportionality using the regularisation theorems (see \Cref{subsec:regularisation} and \Cref{lemma:first_consequences}), and therefore extract the answer to our original question. Working with proportionality will actually make many of our calculations easier, because we will able to neglect constants occurring.

Having broadened our scope to proportionality, one may ask whether \(p\)-modular reductions of pairs of linear characters, or of pairs of spin characters, can be proportional.
The answer, except for pairs of spin characters when \(p=3\), is ``no'': by a straightforward extension of their arguments, the results cited above that classify equality between such pairs also rule out non-trivial proportionality. Thus our result is in fact a characterisation of all proportional pairs of \(p\)-modular characters of \(\hsss n\) for \(p \neq 3\). For $p=3$, it is possible for the $p$-modular reductions of irreducible spin characters to be proportional but unequal: for example, $\br{\lan6,3,2\ran}=2\br{\lan8,2,1\ran}$.

\subsection{Description of the labelling partitions}
\label{subsec:label_descs}

We elaborate on the characterisation of a \fsas strict partition \(\al\) and the definition of the corresponding \(\lx{\al}\) given at the beginning of the introduction.
We use the notions of \(2\)-core, \(2\)-quotient and \fbc defined in \Cref{blocksec,subsec:abacus_and_quotient}, as well as the notation \(\twoc{a}\) for the \(2\)-core partition with largest part \(a\), and \(\btwoc{a}\) for the \fbc strict partition with largest part \(2a-1\), satisfying \(\dbl{\btwoc{a}} = \twoc{a}\).
We also use the operations defined on partitions \(\sqcup\), \(+\), and integer multiplication on partitions, as defined in \Cref{subsec:partitions}.

Suppose $\al$ is \fsas, and write $\al=\ga\sqcup2\be$, where $\ga$ is the strict partition comprising the odd parts of $\al$.
From the \fsas property, it is clear that $\ga$ is a \fbc, and in particular is the \fbc of $\al$, so we can write \(\ga = \btwoc{a}\) for some \(a \geq 0\).
Meanwhile, using the \(4\)-stepped property again, we have
\[
\beta = (2m,2m-2,\dots,2)\sqcup(2k-1,2k-3,\dots,1)
\]
for some $m,k\gs0$. If we let
\(
r=m+k\) and \(s=|m-k+\tfrac12|-\tfrac12
\),
then this becomes
\[\be=\twoc r+\twoc s.
\]
Thus the \fsas strict partitions are precisely the partitions of the form $\btwoc{a}\sqcup2(\twoc r+\twoc s)$ for \(a \geq 0\) and $r, s \geq 0$.

\begin{defn}
Given a \fsas partition \(\al =\btwoc{a}\sqcup2(\twoc r+\twoc s)\), we define \(\lx\al\) up to conjugacy by
\[
    \lx{(\btwoc{a} \sqcup 2(\twoc{r}+\twoc{s}))} = \corandquot{\twoc{a}}{\twoc{r}}{\twoc{s}}
\] 
where  \(\corandbipquot{\kappa}{\bipla}\) denotes the partition with \(2\)-core \(\kappa\) and \(2\)-quotient \(\bipla\).
\end{defn}

Note the conjugate is \(\corandquot{\twoc{a}}{\twoc{r}}{\twoc{s}}' = \corandquot{\twoc{a}}{\twoc{s}}{\twoc{r}}\), so indeed this is well-defined up to conjugacy.

With this notation, our main theorem \Cref{main} claims that
\[
    \bspnorig{\btwoc{a} \sqcup 2(\twoc{r} + \twoc{s})}
    = 2^{\floor{\frac{\max\{r,s\}}{2}}} \bspe{\corandquot{\twoc{a}}{\twoc{r}}{\twoc{s}}}
    = 2^{\floor{\frac{\max\{r,s\}}{2}}} \bspe{\corandquot{\twoc{a}}{\twoc{s}}{\twoc{r}}},
\]
and that these are the only such instances of proportionality between linear and spin characters.

\begin{eg}
Let $\al=(12,8,7,4,3,2)$.
Then \(\al\) is a \fsas strict partition, and we can write $\al = \btwoc{4}\sqcup2(\twoc 4+\twoc 2)$.
Then the partition $\lx\al$ is either $(12,9,6,3^2,1^3)$ or its conjugate $(8,5^2,3^3,2^3,1^3)$, having \(2\)-core \(\twoc{4}\) and \(2\)-quotient \((\twoc{2},\twoc{4})\) or \((\twoc{4},\twoc{2})\) respectively, as can be verified from their \abds below.
\[
\corandquot{\twoc{4}}{\twoc{2}}{\twoc{4}}
\ = \
\abacus(lr,nb,bb,nn,bb,nn,nb,nn,nb,nn,nb)
\ \xleftrightarrow{'} \
\abacus(lr,nb,bb,nb,bb,nb,bb,nn,bb,nn,nb)
\ = \
\corandquot{\twoc{4}}{\twoc{4}}{\twoc{2}}
\]
Our main theorem claims that \(\bspnorig{12,8,7,4,3,2} = 4  \bspe{12,9,6,3^2,1^3} = 4 \bspe{8,5^2,3^3,2^3,1^3}\).
\end{eg}

\subsection{Structure of paper}

Relevant background is collected in \Cref{sec:background}.
The proofs of the ``if'' and ``only if'' directions of our main theorem are logically independent, and can be read in either order.

The proof of the ``only if'' direction (that is, that a spin character proportional to a linear one modulo \(2\) must be labelled by a \fsas strict partition) is contained in \Cref{sec:first_results,sec:proportional_implies_fsas-only_if}.
The proof is by induction, using Robinson's induction and restriction functors and the Murnaghan--Nakayama--Morris rules to constrain the partition labelling a proportional character.

The proof of the ``if'' direction (that is, that \fsas strict partitions do label spin characters proportional to linear ones modulo \(2\)) is contained in \Cref{sec:fsas_implies_proportional-if,sec:runner-swapping,sec:quotient-redistributing}.
We first show the claim for \emph{homogeneous} spin characters, then use certain combinations of Robinson's induction and restriction functors to propagate the property of being proportional.

\begin{acks}
The authors worked on this project during the second author's visit to Stacey Law at the University of Cambridge; we are grateful to the Heilbronn Institute for supporting that visit and to the Centre for Mathematical Sciences for providing facilities.
\end{acks}

\section{Background}
\label{sec:background}

In this paper, $\bbn$ denotes the set of positive integers.
If $\ep\in\{0,1\}$, we write $\bep=1-\ep$.

\subsection{Partitions and Young diagrams}

We recall some simple combinatorial concepts relating to representations of symmetric groups.

\subsubsection{Partitions}
\label{subsec:partitions}

A \emph{partition} is an infinite weakly-decreasing sequence $\la=(\la_1,\la_2,\dots)$ of integers which are eventually zero. The integers $\la_1,\la_2,\dots$ are the \emph{parts} of $\la$. Given a partition $\la$ and a natural number $a$, we write $a\in\la$ to mean that $a$ is a part of $\la$.

We write $\card\la$ for the sum $\la_1+\la_2+\dots$, and say that $\la$ is a partition of $\card\la$. We write $\scrp$ for the set of all partitions, and \(\scrp(n)\) for the set of partitions of \(n\). When writing partitions, we omit trailing zeroes and group together equal parts with a superscript. We write $\vn$ for the unique partition of $0$. For any partition $\la$, we write $\len{\la}$ for the number of non-zero parts of $\la$, which we call the \emph{length} of $\la$.

We use some natural notation for combining partitions.
Given partitions \(\la\) and \(\mu\), we write:
\begin{itemize}
\item
$a\la = (a\la_1,a\la_2,\dots)$ for \(a \in \bbn\);
\item
$\la+\mu = (\la_1+\mu_1,\la_2+\mu_2,\dots)$;
\item
$\la\sqcup\mu$ for the partition obtained by combining all the parts of $\la$ and $\mu$ and writing them in decreasing order;
\item 
\(\la\partsm\mu\) for the partition obtained by retaining only those parts of \(\la\) which are not parts of \(\mu\).
\end{itemize}

A \emph{\bip} is an ordered pair $\bipla=(\la^{(0)},\la^{(1)})$ of partitions, which we call the \emph{components} of $\bipla$. We write $\card\bipla=\card{\la^{(0)}}+\card{\la^{(1)}}$, and say that $\bipla$ is a \bip of $\card\la$.

\subsubsection{Young diagrams}
\label{subsec:yd}

The \emph{Young diagram} of a partition $\la$ is the set
\[
[\la]=\lset{(r,c)\in\bbn^2}{c\ls\la_r},
\]
drawn in the plane using the English convention. We call the elements of the Young diagram the \emph{nodes} of $\la$. In general, a node means an element of $\bbn^2$.
We may not distinguish between a partition and its Young diagram; for example, given partitions \(\la\) and \(\mu\), we write:
\begin{itemize}
    \item 
\(\la \cap \mu\) for the partition with Young diagram the intersection of the Young diagrams of \(\la\) and \(\mu\);
    \item 
\(\mu \subseteq \la\) to mean that every node of \(\mu\) is also a node of \(\la\);
    \item 
\(\la \ydsm \mu\) for the set of nodes of \(\la\) which are not nodes of \(\mu\), provided that \(\mu \subseteq \la\).
\end{itemize}

The \emph{conjugate} of a partition \(\la\), denoted \(\la'\), is the partition whose Young diagram is obtained by reflecting that of \(\la\) in the main diagonal.
Thus the parts of \(\la'\) are the lengths of the columns of \(\la\).

A node $(r,c)$ of $\la$ is \emph{removable} if it can be removed from $[\la]$ to leave the Young diagram of a smaller partition (that is, if $c=\la_r>\la_{r+1}$), while a node not in $\la$ is an addable node of $\la$ if it can be added to $[\la]$ to produce a Young diagram of a larger partition.
The \emph{rim} of a partition $\la$ is the set of nodes $(r,c)\in[\la]$ such that $(r+1,c+1)\notin[\la]$. All removable nodes lie in the rim, but not all nodes in the rim are removable.
We denote by \(\la^{-\eps}\) the partition obtained by removing all removable \(\eps\)-nodes.

If $\la,\mu\in\scrp(n)$, then we say that $\la$ \emph{dominates} $\mu$ (and write $\la\dom\mu$) if
\[
\la_1+\dots+\la_r\gs\mu_1+\dots+\mu_r
\]
for every $r\gs0$. Another way of saying this is that $[\mu]$ can be obtained from $[\la]$ by moving some nodes to lower positions.

\subsubsection{Residues, diagonals and ladders}
\label{subsubsec:residues_and_ladders}

The \emph{residue} of a node $(r,c)$ is the residue of $c-r$ modulo $2$. For $\ep\in\{0,1\}$, an \emph{$\ep$-node} means a node of residue $\ep$.

For \(d\) in \(\bbz\), the $d$th \emph{diagonal} in $\bbn^2$ is the set of all nodes $(r,c)$ for which $c-r=d$ (that is, a top-left to bottom-right line of nodes of constant residue).
We write $\cald_d$ for the \(d\)th diagonal.

For $l\in\bbn$, the $l$th \emph{ladder} in $\bbn^2$ is the set of all nodes $(r,c)$ for which $r+c=l+1$ (that is, a bottom-left to top-right line of nodes of constant residue).
We write $\lad l$ for the $l$th ladder, and say that $\lad m$ is \emph{larger} than $\lad l$ when $m>l$.

If $\la\in\scrp$, the \emph{$2$-regularisation} of $\la$ is the partition obtained by sliding the nodes of $\la$ as far up their ladders as they will go. We write this partition as $\la\reg$. It is easy to see that if $\la,\mu\in\scrp$ then $\la\reg=\mu\reg$ \iff $\card{\la\cap\lad l}=\card{\mu\cap\lad l}$ for every $l$.

\subsection{Strict partitions}
\label{background:strict_shifted_spin}

Now we outline the combinatorial notions relating to spin representations of the double covers of the symmetric group.

A \emph{strict partition} is a partition \(\al\) such that $\al_i>\al_{i+1}$ for all $i\ls\len\al$.
We write $\scrd$ for the set of all strict partitions, and $\scrd(n)$ for the set of strict partitions of $n$.

As in the introduction, the \emph{double} of a strict partition $\al$ is the partition
\[
\dbl\al=\left(\roundup{\mfrac{\al_1}2},\inp{\mfrac{\al_1}2},\roundup{\mfrac{\al_2}2},\inp{\mfrac{\al_2}2},\dots\right).
\]

The \emph{\spr} of a node $(r,c)$ is the residue of $\inp{c/2}$ modulo $2$. For $\ep\in\{0,1\}$, an \emph{\nd\ep} means a node of \spr $\ep$.

A node \((r,c)\) of \(\al\) is \emph{\spre} if it can be removed, possibly together with other nodes of the same \spr, to leave a strict partition.
A node $(r,c)$ which is not a node of $\al$ is \emph{\spad} if it can be added, possibly together with other nodes of the same \spr, to obtain a strict partition.
An \emph{\esprm\ep} means a \spre \nd\ep, and likewise with ``addable'' in place of ``removable''.
We denote by \(\al^{-\eps}\) the strict partition obtained by removing from \(\al\) all \esprms\eps.

For example, let $\al=(7,5,4,1)$. The \spr{}s of the nodes of $\al$ are illustrated in the following diagram.
\[
\young(0110011,01100,0110,0)
\]
The \esprms0 are $(2,5)$, $(3,4)$ and $(4,1)$, and \(\al^{-0} = (7,4,3)\); the \espams0 are $(1,8)$ and $(1,9)$.

\subsection{Irreducible characters of the symmetric group and its double cover}\label{spincharsec}

For $n\gs1$, let $\hsss n$ denote the group with generators $t_1,\dots,t_{n-1},z$ and defining relations
\begin{align*}
t_1^2=\dots=t_{n-1}^2&=z^2=1,
\\
t_iz&=zt_i&&\text{for }1\ls i\ls n-1,
\\
t_it_{i+1}t_i&=t_{i+1}t_it_{i+1}&&\text{for }1\ls i\ls n-2,
\\
t_it_j&=zt_jt_i&&\text{for }1\ls i\ls j-1\ls n-2.
\end{align*}
Then $\hsss n$ is a double cover of the symmetric group $\sss n$, and is a Schur cover of $\sss n$ provided $n\gs4$. The ordinary representation theory of $\hsss n$ is described in detail in the book by Hoffman and Humphreys \cite{hohum}, and the modular theory is addressed in Kleshchev's book \cite{kleshbook}. We will be concerned with representations in characteristic $0$ and characteristic $2$. (We remark that there is another double cover of $\sss n$, obtained by replacing the relations $t_i^2=1$ with $t_i^2=z$; our results apply equally to either double cover.)

The element $z$ is a central involution, so acts as a scalar $\pm1$ on any irreducible representation (over any field). The irreducible representations on which $z$ acts as $1$ are precisely the lifts of the irreducible representations of $\sss n$. Irreducible representations on which $z$ does not act as the identity are called \emph{spin representations}, and (in characteristic $0$) their characters are called \emph{spin characters}.
The irreducible characters which are not spin characters are called \emph{linear characters}
.
We write $\ip\ \ $ for the usual inner product on ordinary characters, extended bilinearly to generalised characters.

The classification of ordinary irreducible representations of $\sss n$ is very well known. For $\la\in\scrp(n)$, let $\spem\la$ denote the \emph{Specht module} defined over a field of characteristic $0$ (see \cite[Section 4]{jbook}, for example), and let $\spe\la$ denote its character. Then the set $\lset{\spe\la}{\la\in\scrp(n)}$ is a complete irredundant list of ordinary irreducible characters of $\sss n$. We will treat $\spem\la$ as a module for $\hsss n$ by letting $z$ act as the identity, and correspondingly treat $\spe\la$ as a character for $\hsss n$.

The classification of ordinary irreducible spin characters of $\hsss n$ goes back to Schur \cite{schu} (although, remarkably, a construction for the corresponding representations was not found until 1990, by Nazarov \cite{naz}). Let $\al\in\scrd(n)$, and say that $\al$ is \emph{even} or \emph{odd} as the number of positive even parts of $\al$ is even or odd. If $\al$ is even, then Schur constructs an irreducible character $\lan\al\ran$ corresponding to $\al$. If $\al$ is odd, then Schur gives two characters $\lan\al\ran_+$ and $\lan\al\ran_-$ corresponding to $\al$. These characters together give a complete irredundant list of ordinary irreducible spin characters of $\hsss n$.

For studying $2$-modular reductions, it will be helpful to adopt some notation introduced in \cite{fayers20spin2alt} for spin characters.
For any $\al\in\scrd$, we define the generalised character
\[
\spn\al=
\begin{cases}
\lan\al\ran&\text{if $\al$ is even},
\\
\frac1{\sqrt2}(\lan\al\ran_++\lan\al\ran_-)&\text{if $\al$ is odd}.
\end{cases}
\]
Using the generalised characters $\spn\al$ simplifies several results for spin representations, such as Schur's degree formula \cite{schu}, the Bessenrodt--Olsson regularisation theorem (\cref{spinregn} below), and the branching rule (\cref{spinbranch} below).
Since the \(2\)-modular reductions of $\lan\al\ran_+$ and $\lan\al\ran_-$ are equal (as they differ only on elements of even order), and since we are concerned only with proportionality, there is no harm in working with these generalised characters $\spn\al$.

We also need to recall the classification of irreducible $\hsss n$-modules in characteristic $2$. Obviously in characteristic $2$ there are no irreducible spin representations, so the irreducible representations of $\hsss n$ are precisely the lifts of the irreducible representations of $\sss n$. These are afforded by the \emph{James modules} $\jmsm\mu$ for $\mu\in\scrd$: the James module $\jmsm\mu$ can be constructed as the unique simple quotient of the Specht module $\spem\mu$. We write $\jms\mu$ for the Brauer character of $\jmsm\mu$, and for an ordinary character $\chi$, we define $\dn\chi\mu$ to be the multiplicity of $\jms\mu$ as a composition factor of a $2$-modular reduction of $\br\chi$. These multiplicities, as $\chi$ ranges over the ordinary irreducible characters, are the \emph{decomposition numbers} for $\hsss n$.

\subsection{Hooks, bars and character values}
\label{charvalsec}

In this section we recall the Murnaghan--Nakayama rule for computing character values for the symmetric group, and its spin version, proved by Morris. In fact for our purposes we can make do with very simplified versions of these rules.

\subsubsection{Hooks and the Murnaghan--Nakayama rule}

We begin by recalling the well-known concept of rim-hooks of partitions.
Take \(k\in \bbn\) and \(\la\) a partition. A \emph{\(k\)-rim-hook}, or rim-hook of length \(k\), is a connected subset of the rim of $\la$ comprising $k$ nodes which can be removed to leave the Young diagram of a smaller partition. For a rim-hook \(\de\) in \(\la\), the smaller partition obtained by removing \(\de\) is denoted \(\la \hooksm \de\).

Let $\rh\la k$ be the set of partitions that can be obtained by removing a \(k\)-rim-hook from $\la$. Given $\nu\in\scrp(n)$, let $\spe\la(\nu)$ be the value of the character $\spe\la$ on elements of $\sss n$ of cycle type $\nu$. Now a simplified version of the Murnaghan--Nakayama rule (see for example \cite[21.1]{jbook}) can be stated as follows.

\begin{thm}\label{murnak}
Suppose $\la\in\scrp(n)$ and $k\in\bbn$. For each $\mu\in\rh\la k$ there is a non-zero constant $c_\mu$ such that for any $\nu\in\scrp(n-k)$,
\[
\spe\la(\nu\sqcup(k))=\sum_{\mu\in\rh\la k}c_\mu\spe\mu(\nu).
\]
\end{thm}

\subsubsection{Odd bars and Morris's rule}

Now we consider spin characters, and explain Morris's rule (as given by Hoffman and Humphreys \cite[Theorem 10.1]{hohum}). Because we are concerned with $2$-modular Brauer characters in this paper, we need only consider $2$-regular conjugacy classes in $\hsss n$. In fact these are in one-to-one correspondence with $2$-regular conjugacy classes in $\sss n$: for each $2$-regular conjugacy class $C$ in $\sss n$, there are two conjugacy classes in $\hsss n$ of elements that map to $C$ under the quotient map $\hsss n\to\sss n$, and only one of these consists of elements of odd order. So $2$-regular conjugacy classes in $\hsss n$ can be labelled by partitions of $n$ into odd parts. Given $\al\in\scrd$ and a partition $\nu$ into odd parts, we write $\spn\al(\nu)$ for the value of $\spn\al$ on the $2$-regular conjugacy class corresponding to $\nu$.

The appropriate analogue of rim-hooks for strict partitions is \emph{bars}. Take $k\in\bbn$ odd and \(\al\in\scrd\). A \emph{$k$-bar} in $\al$ is one of the following:
\begin{itemize}
\item
a part of size equal to $k$;
\item
a pair of parts whose sizes sum to $k$;
\item
the last \(k\) nodes in a part of size $\al_r>k$ such that $\al_r-k$ is not already a part of $\al$.
\end{itemize}
Removing a \(k\)-bar from \(\al\) means removing the bar and, if necessary, reordering the remaining parts to form a strict partition.

Let $\rb\al k$ denote the set of strict partitions that can be obtained from $\al$ by removing a $k$-bar. Now a simplified version of Morris's rule can be stated as follows.

\begin{thm}\label{spinmurnak}
Suppose $\al\in\scrd(n)$ and $k\in\bbn$ is odd. For each $\be\in\rb\al k$ there is a non-zero constant $d_\be$ such that for any partition $\nu$ of \(n-k\) into odd parts,
\[
\spn\al(\nu\sqcup(k))=\sum_{\be\in\rb\al k}d_\be\spn\be(\nu).
\]
\end{thm}

\subsection{Blocks}\label{blocksec}

The $2$-block structure of $\hsss n$ was determined by Bessenrodt and Olsson \cite{bessenrodt97blocksofdoublecovers}, extending the classical result of Brauer and Robinson (the ``Nakayama Conjecture'') for $\sss n$. We summarise the essential points here.

\subsubsection{Cores and weights}

Suppose $\la\in\scrp$ and $k\in\bbn$. The \emph{$k$-core} of $\la$ is the partition obtained by successively removing \(k\)-rim-hooks until no more remain. It is well-known that this partition is independent of the choice of \(k\)-rim-hook removed at each stage. The \emph{$k$-weight} of $\la$ is the number of \(k\)-rim-hooks removed to reach the $k$-core, or equivalently the number of rim-hooks of $\la$ of size divisible by $k$. By the Murnaghan--Nakayama rule (\Cref{murnak}), the $k$-weight of $\la$ is the maximum $w$ such that the character value $\spe\la(k^w,1^{n-kw})$ is non-zero.

When \(k=2\), it is easy to see that the $2$-cores are precisely the partitions of the form $(r,r-1,\dots,1)$ for $r\gs0$. In this paper we will write $\cor r=(r,r-1,\dots,1)$.

The Brauer--Robinson Theorem (which applies equally well for $\hsss n$) says that $\spe\la$ and $\spe\mu$ lie in the same $k$-block \iff $\la$ and $\mu$ have the same $k$-core. This result also gives the distribution of the irreducible Brauer characters $\jms\mu$ into blocks, since $\jms\mu$ is a composition factor of $\bspe\mu$.

This means that we can define the \emph{\(k\)-core} and \emph{\(k\)-weight} of a block of $\hsss n$ to be the common $k$-core and $k$-weight of the partitions $\la$ labelling characters $\spe\la$ in that block.

\subsubsection{Bar-cores and bar-weights}

Now we come to spin characters. For \(\al\in\scrd\) and \(k\) odd, the \emph{$k$-bar-core} of \(\al\) is the strict partition obtained by successively removing \(k\)-bars until none remain, and the \emph{$k$-bar-weight} of $\al$ is the number of \(k\)-bars removed to reach the $k$-bar-core. By Morris's rule (\Cref{spinmurnak}), the $k$-bar-weight of $\al$ is the maximum $w$ such that the character value $\spn\al(k^w,1^{n-kw})$ is non-zero.

To work in characteristic \(2\), we require a slightly different definition:
for $\al\in\scrd$, the \emph{\fbc} of $\al$ is obtained by repeatedly applying moves of the following types:
\begin{itemize}
\item
deleting an even part;
\item
deleting two parts whose sum is a multiple of $4$;
\item
replacing an odd part $\al_r > 4$ with $\al_r-4$, if $\al_r-4$ is not already a part of $\al$.
\end{itemize}
(This permits removing parts of size equal to \(2\) in addition to ``\(4\)-bars'' defined analogously to the odd case.)
The $4$-bar-weight of \(\al\) is the total number of nodes removed in this way divided by \(2\).

The possible \fbcs are the partitions $(4t-1,4t-5,\dots,3)$ and $(4t+1,4t-3,\dots,1)$ for $t\gs0$. For $r\gs1$, we write $\btwoc r$ for the \fbc with largest part $2r-1$, and we set $\btwoc0=\vn$.
Note that the function $\tau\mapsto\dbl\tau$ gives a bijection from the set of \fbcs to the set of $2$-cores, sending $\btwoc r$ to $\cor r$ for every $r$. (The inverse of this bijection is given by $\nu\mapsto\dbl\nu'$.)

Now we can describe the distribution of spin characters into blocks \cite[Theorem 4.1]{bessenrodt97blocksofdoublecovers}: suppose $\al\in\scrd$, and let $\gamma$ be the \fbc of $\al$; then the spin character $\spn\al$ lies in the \(2\)-block of $\hsss n$ with \(2\)-core $\dbl\gamma$.

\subsubsection{RoCK blocks}

Say that the block with \(2\)-core \(\twoc{a} = (a,a-1, \ldots, 1)\) and \(2\)-weight $w$ is \emph{RoCK} if $a \gs w-1$. RoCK blocks, which are often also called \emph{Rouquier blocks}, are particularly well understood, and part of the proof of our main theorem will involve reducing to the case of RoCK blocks.

\subsubsection{Regularisation theorems}
\label{subsec:regularisation}

A very useful tool in the study of decomposition numbers for the symmetric groups is James's regularisation theorem \cite[Theorem A]{james_regn}. For the double cover $\hsss n$ in characteristic $2$, a useful analogue was proved by Bessenrodt and Olsson \cite[Theorem 5.2]{bessenrodt97blocksofdoublecovers}. 
These results are as follows, where the double of a strict partition is as defined in \Cref{background:strict_shifted_spin} and the regularisation of a partition is as defined in \Cref{subsubsec:residues_and_ladders}.

\begin{thm}\label{regn}
Suppose $\la\in\scrp$.
Then $\dn{\spe\la}{\la\reg}=1$, while $\dn{\spe\la}\mu=0$ unless $\mu\dom\la\reg$.
\end{thm}

\begin{thm}\label{spinregn}
Suppose $\al\in\scrd$, and let $e$ be the number of even parts of $\al$. Then $\dn{\spn\al}{\dbl{\al}\reg}=2^{e/2}$, while $\dn{\spn\al}\mu=0$ unless $\mu \dom \dbl{\al}\reg$. 
\end{thm}

\subsubsection{Content and spin-content}

The block classification can alternatively be expressed in terms of residues of nodes. Define the \emph{content} of a partition $\la$ to be the multiset of residues of its nodes. It was proved by Littlewood that two partitions of $n$ have the same $2$-core \iff they have the same content \cite[p.~347]{littlewood51}. So given a block of $\hsss n$, we may define its content to be the common contents of the partitions labelling characters $\spe\la$ in the block.

To extend this idea to spin characters, define the \emph{spin-content} of a strict partition $\al$ to be the multiset of \sprs of its nodes. Then $\spn\al$ lies in the block whose content coincides with the spin-content of $\al$ (as can be shown, for example, by observing that the spin-content of a partition equals the content of its double \cite[Lemma~3.9]{bessenrodt97blocksofdoublecovers} and using \Cref{spinregn} above).

\subsection{Abacus displays and 2-quotients}
\label{subsec:abacus_and_quotient}

Given $\la\in\scrp$, take an integer $r\gs\len\la$, and define the \emph{beta-numbers} $\be_i=\la_i+r-i$ for $i=1,\dots,r$. Take an abacus with two vertical runners; label these runners $0$ and $1$ from left to right. Mark positions $0,2,4,\dots$ from the top down on runner $0$, and $1,3,5,\dots$ from the top down on runner $1$. Now place $r$ beads on the abacus, at positions $\be_1,\dots,\be_r$. The resulting configuration is the \emph{$r$-bead \abd} for $\la$.

The \emph{\(2\)-quotient} of a partition \(\la\) is a bipartition constructed from an abacus display for \(\la\) by viewing each runner in isolation as a \(1\)-runner abacus and reading off the corresponding partitions. This requires a choice of convention with regard to the number of beads used in the abacus display and which runner corresponds to which component of the bipartition. In this paper we follow a slightly unusual convention: we choose the number of beads \(r\) so that runner \(1\) has at least as many beads as runner \(0\) (or equivalently so that the number of beads is congruent modulo $2$ to the length of the $2$-core of $\la$), and then define the \(\eps\)th component of the \(2\)-quotient from the \(\eps\)th runner (explicitly, the \(2\)-quotient is \((\la^{(0)}, \la^{(1)})\) where $\la^{(\ep)}_i$ is the number of empty spaces above the $i$th lowest bead on runner $\ep$). This can alternatively be described as follows: given any \(r\)-bead abacus display, let the \(1\)st component \(\la^{(1)}\) correspond to the runner with strictly more beads if such a runner exists, or to runner \(1\) if the numbers of beads on the runners are equal.

(The more usual convention for the $2$-quotient (see for example \cite[Section 2.7]{jk}) is to take $r$ always to be even and let the \(\eps\)th component \(\la^{(\eps)}\) correspond to runner \(\eps\); our choice will be more convenient when we consider RoCK partitions. The analogue for odd primes of our convention is used in \cite{fayers2016}, where it is called the \emph{ordered $p$-quotient}.)

It is well-known (and easy to see on the abacus) that a partition is defined by its $2$-core and $2$-quotient. Conversely, given a $2$-core $\ka$ and a \bip $\bipla=(\la^{(0)}, \la^{(1)})$, there is a partition with $2$-core $\ka$ and $2$-quotient $\bipla$. We write this as $\corandquot\ka{\la^{(0)}}{\la^{(1)}}$.

\subsection{Induction and restriction}\label{indressec}

Induction and restriction between symmetric groups of different sizes is one of the most important tools in the representation theory of these groups, and the same applies to the double covers. Robinson's \(i\)-induction and -restriction functors, which give more powerful results in positive characteristic, likewise extend to \(\hsss{n}\). We summarise the results we will need here, following the account in \cite{fayers18spin2}, where references and further details can be found.

Given a character $\chi$ for $\hsss n$, we write $\chi{}\downarrow_{\hsss{n-1}}$ for its restriction to $\hsss{n-1}$, and $\chi{}\uparrow^{\hsss{n+1}}$ for the induced character for $\hsss{n+1}$. Now suppose $\chi$ lies in a single block, and write the content of this block as $\{0^a,1^b\}$. Define $e_0\chi$ to be the component of $\chi{}\downarrow_{\hsss{n-1}}$ lying in the block with content $\{0^{a-1},1^b\}$ if there is such a block, and $e_0\chi=0$ otherwise. Define $f_0\chi$ to be the component of $\chi{}\uparrow^{\hsss{n+1}}$ lying in the block with content $\{0^{a+1},1^b\}$ if there is such a block, or $0$ otherwise. Define the characters $e_1\chi$ and $f_1\chi$ similarly. It follows from the branching rules for $\hsss n$ and the block classification that $\chi{}\downarrow_{\hsss{n-1}}=e_0\chi+e_1\chi$ and $\chi{}\uparrow^{\hsss{n+1}}=f_0\chi+f_1\chi$ for any $\chi$.

Fix $\ep\in\{0,1\}$ for the rest of this subsection. The functor $e_\ep$ is defined for any $n>0$, so we can define the power $e_\ep^r$ and the divided power $e_\ep^{(r)}=e_\ep^r/r!$, and similarly for $f_\ep$.
These functors can be applied either to ordinary characters or $2$-modular Brauer characters, and commute with reduction modulo $2$.

The classical branching rule for induction and restriction of ordinary irreducible characters of symmetric groups extends to the double covers in characteristic $2$, and can be phrased in terms of the functors $e_i$ as follows.

\begin{thm}\label{branching}
Suppose $\la\in\scrp(n)$ and $\mu\in\scrp(n-r)$, and $\ep\in\{0,1\}$. If $\mu$ is obtained from $\la$ by removing $r$ $\ep$-nodes, then $\ip{e_\ep^{(r)}\spe\la}{\spe\mu}=\ip{f_\ep^{(r)}\spe\mu}{\spe\la}=1$; otherwise $\ip{e_\ep^{(r)}\spe\la}{\spe\mu}=\ip{f_\ep^{(r)}\spe\mu}{\spe\la}=0$.
\end{thm}

The branching rule for spin characters is given as follows (slightly rephrased from \cite{fayers20spin2alt}).

\begin{thm}[{\cite[Proposition 4.11]{fayers20spin2alt}}]\label{spinbranch}
Suppose $\al\in\scrd(n)$ and $\be\in\scrd(n-r)$, and $\ep\in\{0,1\}$. Then $\ip{e_\ep^{(r)}\spn\al}{\spn\be}$ and $\ip{f_\ep^{(r)}\spn\be}{\spn\al}$ are non-zero \iff $\be$ is obtained by removing $r$ \nds\ep from $\al$.
If this is the case, let $b$ be the number of values $c\gs2$ such that $\al\ydsm\be$ contains a node in column $c$ but does not contain any nodes in columns $c-1$ and $c+1$. Then
\[
\ip{e_\ep^{(r)}\spn\al}{\spn\be}=\ip{f_\ep^{(r)}\spn\be}{\spn\al}=2^{b/2}.
\]
\end{thm}

We remark that the integer $b$ in \cref{spinbranch} can also be characterised as the number of even integers which are parts of $\al$ or $\be$ but not both.

\section{First results on proportionality}
\label{sec:first_results}

Our first step towards proving the ``only if'' direction of the main theorem, as well as establishing the claimed constant of proportionality, is to 
examine the consequences of the regularisation theorems \cref{regn,spinregn}.
To say that $\bspe\la$ is proportional to $\bspn\al$ is to say that there is a scalar $a$ such that $\dn{\spn\al}\mu=a\dn{\spe\la}\mu$ for every $\mu\in\scrd$.
The regularisation theorems substantially restrict the possible pairs $(\la,\al)$ for which this can happen, as well as determining the constant $a$.
The following result is immediate from \cref{regn,spinregn}.

\begin{lemma}
\label{lemma:first_consequences}
Suppose \(\la\in\scrp\) and \(\al\in\scrd\), with \(\bspn{\al} \propto \bspe{\la}\). Then:
\begin{enumerate}[(i)]
    \item\label{item:2-reg}
\(\la\reg = \dbl{\al}\reg\)
(that is, \(\la\) can be obtained from \(\dbl{\al}\) by sliding nodes along the \(2\)-ladders);
    \item\label{item:multiple}
\(\bspn{\al} =  2^{e/2} \bspe{\la}\), where \(e\) is the number of even parts of \(\al\).
\end{enumerate}
\end{lemma}

Part \Cref{item:multiple} of this lemma establishes the constant of proportionality claimed in our main theorem \Cref{main}.

Next we examine the consequences of the Murnaghan--Nakayama--Morris rules for character values given in \cref{charvalsec}. First we compare $k$-weight and $k$-bar-weight.

\begin{lemma}
\label{sameweight}
Let \(\la\in\scrp\) and \(\al\in\scrd\), with \(\bspn{\al} \propto \bspe{\la}\). Then for any odd integer \(k\), the \(k\)-weight of \(\la\) equals the $k$-bar-weight of \(\al\).
\end{lemma}

\begin{proof}
The characters \(\spn{\al}\) and \(\spe{\la}\) vanish on precisely the same \(2\)-regular classes. In particular, the \(k\)-weight of \(\la\) and the $k$-bar-weight of \(\al\) are both equal to the largest $w$ such that \(\spn{\al}\) and \(\spe{\la}\) are non-zero on elements of cycle type \((k^w, 1^{n-kw})\).
\end{proof}

Observe that any strict partition has a unique largest bar of odd length (if all parts are odd, it is the largest part; if all parts are even, it is the largest part minus one node; otherwise, it is the largest odd part together with the largest even part).

\begin{propn}\label{hatprop}
Suppose $\al\in\scrd$ has at least one odd part.
Let \(\ulob\) be the unique largest bar of odd length in \(\al\), and let $k$ be its length.
If $\la\in\scrp$ with $\bspn\al\propto\bspe\la$, then $\la$ has a unique \(k\)-rim-hook \(\delta\), and $\bspn{\hatal}\propto\bspe{\la\hooksm\delta}$.
\end{propn}

\begin{pf}
It is easy to see that $\hatal$ has no $k$-bars, since the largest bar of odd length in $\hatal$ is at most the sum of the largest odd part and largest even part of \(\hatal\) (taking these to be zero if no such parts exist), which is strictly less than the sum of those for \(\al\).
So $\al$ has $k$-bar-weight $1$, and therefore \cref{sameweight} gives that $\la$ has $k$-weight $1$.
This implies in particular that $\la$ has a unique rim-hook of length $k$, which we denote \(\delta\).
Then $\rh\la k=\{\la\hooksm\delta\}$, while $\rb\al k=\{\hatal\}$.

Let $a$ be such that $\bspn\al=a\bspe\la$. Then \cref{murnak,spinmurnak} give
\[
\bspn{\hatal}(\nu)=\frac1{d_{\hatal}}\bspn\al(\nu\sqcup(k))=\frac a{d_{\hatal}}\bspe\la(\nu\sqcup(k))=\frac{ac_{\la\hooksm\delta}}{d_{\hatal}}\bspe{\la\hooksm\delta}(\nu)
\]
for any partition $\nu$ of $n-k$ into odd parts, and therefore
\[
\bspn{\hatal}=\frac{ac_{\la\hooksm\delta}}{d_{\hatal}}\bspe{\la\hooksm\delta}.\qedhere
\]
\end{pf}

Finally we derive some consequences of the modular branching rules.
Recall, given \(\la \in \scrp\), $\al\in\scrd$ and $\ep\in\{0,1\}$, that \(\la^{-\eps}\) denotes the partition obtained by removing all the removable \(\eps\)-nodes from \(\la\), and $\al^{-\ep}$ denotes the strict partition obtained by removing all the \esprms\ep from $\al$.

\begin{propn}\label{eiprop}
Let \(\la\in\scrp\) and \(\al\in\scrd\) with \(\bspn{\al} \propto \bspe{\la}\).
Let \(\ep \in \set{0,1}\).
Then \(\bspn{\al^{-\ep}} \propto \bspe{\la^{-\eps}}\).
\end{propn}

\begin{pf}
Let $r$ be the number of \esprms\ep of $\al$. \cref{spinbranch} shows that $r$ is maximal such that $e_\ep^{(r)}\bspn\al\neq0$, and so $r$ is also maximal such that $e_\ep^{(r)}\bspe\la\neq0$.
\cref{branching} then shows that $r$ is the number of removable $\ep$-nodes of $\la$, and that $e_\ep^{(r)}\bspe\la=\bspe{\la^{-\ep}}$. At the same time, \cref{spinbranch} shows that $e_\ep^{(r)}\bspn\al$ is a non-zero multiple of $\bspn{\al^{-\ep}}$. So
\[
\bspn{\al^{-\ep}} \propto e_\ep^{(r)}\bspn\al \propto e_\ep^{(r)}\bspe\la = \bspe{\la^{-\ep}}.\qedhere
\]
\end{pf}

\section{Proportional implies 4-stepped-and-semicongruent}
\label{sec:proportional_implies_fsas-only_if}

Our aim in this section is to prove the following statement.
We say that a strict partition \(\al\) is \emph{proportional} if there exists a partition \(\la\) such that \(\bspn{\al} \propto \bspe{\la}\); the definition of \fsas was given at the start of the Introduction.

\begin{thm}
\label{onlyif}
Suppose $\al\in\scrd$. If $\al$ is proportional, then $\al$ is \fsas.
\end{thm}

This is enough to prove the ``only if'' claim of \cref{main} when given also the ``if'' direction (\Cref{thm:proportionality}).
Indeed, if \(\bspn{\al}\) is proportional to \(\bspe{\la}\), then \Cref{onlyif} says that $\al$ is \fsas, and we must furthermore deduce that $\la\in\{\lx\al, \lx\al'\}$;
having that \(\al\) is \fsas, \Cref{thm:proportionality} will tell us that \(\bspn{\al}\) is proportional to \(\bspe{\lx\al}\); since the constants of proportionality are the same (by \Cref{lemma:first_consequences}\Cref{item:multiple}), we find that \(\bspe{\lx\al} = \bspe{\la}\).
Then a theorem of Wildon \cite[Theorem 1.1.1(ii)]{wildon2008distinctrows} gives $\la\in\{\lx\al, \lx\al'\}$.

The remainder of this section comprises the proof of \Cref{onlyif}. The proof is by induction on the size of \(\al\): let \(\al\) be a proportional strict partition; suppose that all smaller proportional strict partitions are \fsas; and our goal is to show that \(\al\) is \fsas.

We achieve our goal principally by applying the inductive hypothesis to three strict partitions: \(\al^{-0}\), \(\al^{-1}\) and \(\hatal\), where \(\ulob\) denotes the unique largest bar of odd length in \(\al\).
We also make use of requirements on $\dbl\al\reg$ (\Cref{lemma:first_consequences}\Cref{item:2-reg}) and the bar-weights of~\(\al\) (\Cref{sameweight}).

The strict partitions \(\al^{-0}\) and \(\al^{-1}\) are proportional by \Cref{eiprop}, and so whenever they are smaller than \(\al\) the inductive hypothesis applies. Meanwhile \(\hatal\) is proportional whenever \(\al\) contains an odd part by \Cref{hatprop} (whilst we quickly deal with the case of \(\al\) having all parts even in the first subsection below).

Throughout, let \(\la\) be a partition such that \(\bspn\al\) is proportional to \(\bspe\la\).
We summarise below the assumptions just made.

\smallskip
\needspace{10em}
\begin{framing}{black}
\noindent\textbf{Assumptions in force for the rest of \cref{sec:proportional_implies_fsas-only_if}:}
\begin{itemize}[beginthm]
\item
$\al\in\scrd$ is \prp;
\item
\(\ulob\) is the unique largest odd bar in \(\al\);
\item
$\la\in\scrp$ is such that \(\bspn\al\) is proportional to \(\bspe\la\);
\item
any proportional strict partition $\be$ with $|\be|<|\al|$ is \fsas.
\end{itemize}
\end{framing}

\subsection{Case of all parts having same parity}

First we consider the case where the parts of $\al$ all have the same parity.

\begin{lemma}
\label{lemma:one-part_partitions}
Suppose \(\al\) consists of a single odd part.
Then \(\al\) is \fsas (that is, \(\al \in \set{(1), (3)}\)).
\end{lemma}

\begin{proof}
Write \(\al_1 = 2k+1\). Then \(\al\) has a \((2k+1)\)-bar, so by \cref{sameweight} \(\la\) has a \((2k+1)\)-rim-hook. But \(\la\) is a partition of \(2k+1\), so \(\la\) must itself be a rim-hook. Meanwhile \(\la\reg = \dbl{\al}\reg = (k+1,k)\) by \Cref{lemma:first_consequences}\Cref{item:2-reg}, so the only possibility is \(\la = (k+1,1^k)\).

If $k\gs1$ then \(\al\) has also a \((2k-1)\)-bar, so \(\la\) has a $(2k-1)$-rim-hook (by \cref{sameweight}). But the second-largest rim-hook of \(\la = (k+1,k)\) is of length \(k\), so we must have \(k \geq 2k-1\). This gives \(k \leq 1\), and so \(\al_1 = 2k+1 \leq 3\) as required.
\end{proof}

\begin{propn}
Suppose \(\al\) has all parts odd.
Then \(\al\) is \fsas (that is, it is a \fbc).
\end{propn}

\begin{proof}
The case \(\len{\al} = 1\) is \Cref{lemma:one-part_partitions}, so suppose \(\len{\al} > 1\). The unique largest odd bar \(\ulob\) is the largest part of \(\al\).
By \Cref{hatprop} and the inductive hypothesis, \(\hatal\) is \fsas, and since it has all parts odd \(\hatal\) is \fbc. We can write \(\al_1 = \al_2 + 2k\) for some positive integer \(k\); we are required to show \(k = 2\). Let \(\eps\) be the \spr of the node at the end of row $2$ of \(\al\) (and hence the \spr of the nodes at the end of every row except possibly row \(1\)).

Suppose towards a contradiction \(k \geq 3\).
Then \(\al^{-\eps} = (\al_2 + r, \al_2 -2, \al_3 - 2, \ldots)\) where \(r \in \set{2k, 2k-2}\). 
Then \(\al_2 + r > 4\) but \(\al_2 +r - 4 \not\in\al^{-\eps}\), so \(\al^{-\eps}\) is not \fstepped, a contradiction.

Now suppose towards a contradiction \(k=1\). Then the node at the end of row $1$ of \(\al\) is removable and has \spr $\bep$, giving \(\al^{-\beps} = (\al_2 +1, \al_2, \al_3, \ldots)\). If \(\al_2 > 3\), then \(\al^{-\beps}_1 > 4\) but \(\al^{-\beps} -4 \not\in \al^{-\beps}\), contradicting the assumption that \(\al^{-\beps}\) is \fstepped. If \(\al_2 \in \set{1,3}\), then \(\al \in \set{(3,1), (5,3)}\) and it is easily checked that these strict partitions are not proportional (for example by inspecting the character tables in \cite{jk} and \cite{hohum}, or by identifying differences in odd weights and removable nodes between \(\al\) and candidate partitions).
\end{proof}


\begin{propn}
Suppose \(\al\) has all parts even.
Then \(\al\) is \fsas.
\end{propn}

\begin{proof}
Trivially any strict partition with all parts even is \(4\)-semicongruent; it remains to verify the \fstepped property.

Let \(i\) be such that \(\al_i > 4\) (hence \(\al_i \geq 6\)).
Let \(\eps\) be the \spr of the node at the end of row \(i\) of \(\al\). Then \(\al_i^{-\eps} = \al_i - 1 > 4\) and \(\al^{-\eps}\) is \fsas by the inductive hypothesis. Thus \(\al_i - 5 \in \al^{-\eps}\). But since \(\al_i -5\) is odd, it cannot be a part of \(\al\), and likewise nor can \(\al_i -3\); thus a part of size \(\al_i - 5\) must have been obtained by removing a single \nd\ep from a part of size \(\al_i - 4 \in \al\).
\end{proof}

Given the above two propositions, we assume for the remainder of this section that \(\al\) has both odd and even parts.
In particular, the unique largest odd bar \(\pi\) comprises the largest even part and the largest odd part, and by \Cref{hatprop} the partition \(\la\) has a unique rim-hook \(\delta\) of length \(\abs{\pi}\), and \(\bspn\hatal\) is proportional to \(\bspe{\la\hooksm\delta}\).
Then by the inductive hypothesis $\hatal$ is \fsas.

\smallskip
\needspace{10em}
\begin{framing}{black}
\noindent\textbf{Additional assumption in force for the rest of \cref{sec:proportional_implies_fsas-only_if}:}
\begin{itemize}[beginthm]
\item
$\al$ has both odd and even parts;
\item 
\(\de\) is the unique rim-hook of length \(\abs{\pi}\) in \(\la\), satisfying \(\bspn\hatal \propto \bspe{\la\hooksm\delta}\).
\end{itemize}
\end{framing}

We deal with one additional configuration here.

\begin{lemma}
\label{lemma:1_only_odd_2_not_in_hat}
Suppose \(1\) is the only odd part of \(\al\) and that \(2 \not\in \hatal\).
Then \(\al\) is \fsas.
\end{lemma}

\begin{proof}
The hypotheses imply that \(\hatal\) does not have \(1\), \(2\) or \(3\) as a part.
Then since \(\hatal\) is \fsas, it is of the form \((4k, 4k-4, \dots, 4)\) for some nonnegative integer \(k\).
So \(\al = (\al_1, 4k, 4k-4, \ldots, 4, 1)\) for some even integer \(\al_1 > 4k\).
If \(\al_1 = 4k+4\), or if \(k=0\) and \(\al_1 = 4k+2\), then \(\al\) is \fsas, as required.

Suppose towards a contradiction either \(\al_1 \geq 4k+6\) or \(k > 0\) and \(\al_1 = 4k+2\).
In either case, \(\al\) has at least one \esprm0, and \(\al^{-0} = (\al^{-0}_1, 4k-1, 4k-5, \ldots, 3)\), where \(\al^{-0}_1\) is either equal to \(\al_1\) or to \(\al_1 - 1\); in either case \(\al^{-0}_1 > 4\) but \(\al^{-0}_1 - 4\) is not a part of \(\al^{-0}\).
Then $\al^{-0}$ is not \fstepped, a contradiction.
\end{proof}

\subsection{Initial bounds on differences between parts of \texorpdfstring{\(\al\)}{alpha}}

We begin the proof for general \(\al\) by showing that consecutive parts of $\al$ cannot be too far apart. Clearly the parts except the largest odd and even parts cannot differ from their neighbouring parts by more than \(4\), since \(\hatal\) is \fstepped.
But furthermore the largest odd and even parts cannot be too much larger than the others, or else we could remove nodes from these parts and find \(\al^{-0}\) or \(\al^{-1}\) to have significantly differing parts, and hence not be \fstepped.
This is made precise by the following lemma.

\begin{lemma}
\label{lemma:alpha1_vs_alpha2_bounds}
Let \(1 \leq i \leq \len{\al} \).
\begin{enumerate}[(i)]
    \item
    \(\al_i - \al_{i+1} \leq 5\).
    \item
    \(\al_i - \al_{i+1} \leq 4\) unless \(\al_1 \not\equiv \al_2 \ppmod{2}\) and either \(i=1\), or \(i=2\), \(\al_2\) is odd and \(\al_1 - \al_2 = 3\).
\end{enumerate}
\end{lemma}

\begin{proof}
Let \(\eps\) be the \spr of the node $(i,\al_i)$ (the node at the end of row \(i\)).
\begin{enumerate}[(i)]
\item 
Suppose towards a contradiction that \(\al_i - \al_{i+1} > 5\).

First consider the case where \(\al_i\) is even.
Then the \nd\ep at the end of row $i$ of $\al$ is \spre, and thus \(\al^{-\epsilon}_i = \al_i - 1\). But then \(\al^{-\epsilon}_i - 4 = \al_i - 5 > \al_{i+1} \geq \al^{-\epsilon}_{i+1}\), contradicting the assumption that \(\al^{-\epsilon}\) is \fstepped.

Now consider the case where \(\al_i\) is odd. If \(\al_i - \al_{i+1} > 6\), then the two \nd\ep{}s at the end of row $i$ of $\al$ are \spre, and thus \(\al^{-\epsilon}_i = \al_i -2\).
But then \(\al^{-\epsilon}_i - 4 = \al_i - 6 > \al_{i+1} \geq \al^{-\epsilon}_{i+1}\), contradicting the assumption that \(\al^{-\epsilon}\) is \fstepped.

If \(\al_i - \al_{i+1} = 6\), then the node at the end of row \(i+1\) has residue \(\beps\) distinct from that of row \(i\), and \(\al_i^{-\beps} - \al_{i+1}^{-\beps} \geq \al_i - \al_{i+1} > 4\). Thus \(\al^{-\beps}\) is not \fstepped, so to avoid a contradiction with the inductive hypothesis we have \(\al^{-\beps} = \al\) (that is, \(\al\) has no \esprms\beps). In order for row \(i+1\) to have no \esprms\beps, we must have \(\al_{i+1} \geq 3\) and \(\al_{i+1} - 1, \al_{i+1}-2 \in \al\). But then \(\al_{i+1}, \al_{i+1} -2 \in \hatal\) (with \(\al_{i+1}\) odd), contradicting the assumption that \(\hatal\) is \fsemic.

\item
Since \(\hatal\) is \fstepped, we have \(\al_i - \al_{i+1} \leq  4\) for all \(i\) except possibly when \(\al_i\) is the largest part of its parity.
If \(\al_i\) is largest of its parity but \(i \geq 3\), then \(\al_{i-1}\) is not largest of its parity, 
and so \(\al_{i-1} \in \hatal\) but \(\al_i \not\in \hatal\).
Then \(\al_{i-1} - \al_{i+1} \leq  4\) and hence \(\al_i - \al_{i+1} \leq 4\).
So we may assume \(i \in \set{1,2}\).

If \(\al_1 \equiv \al_2 \ppmod{2}\), then by part (i) we have \(\al_1 - \al_2 \leq 4\) (resolving the case \(i=1\)), while \(\al_2\) is not the largest part of its parity (resolving the case \(i=2\)).
So we may assume \(\al_1 \not\equiv \al_2 \ppmod{2}\).

Consider the case \(i=2\), \(\al_1\) odd and \(\al_2\) even.
We have \(\al_2 - \al_3 \leq 5\) by part (i).
Suppose towards a contradiction that \(\al_3 = \al_2 - 5\).
We argue as in the final paragraph of part (i): the node at the end of row \(3\) has residue \(\beps\), and thus \(\al_2^{-\beps} - \al_{3}^{-\beps} \geq \al_2 - \al_{3} > 4\);
therefore \(\al^{-\beps}\) is not \fstepped, and so \(\al^{-\beps} = \al\); in order for row \(3\) to have no \esprms\beps, we have \(\al_{3} \geq 3\) and \(\al_{3} - 1, \al_{3}-2 \in \al\); but then \(\al_3, \al_3-2 \in \hatal\), contradicting the assumption that \(\hatal\) is \fsemic.
So \(\al_2 - \al_3 \leq 4\) in this case.

The only remaining case to consider is when \(i=2\), \(\al_1\) even and \(\al_2\) odd, where we must show that if \(\al_2 - \al_3 = 5\) then \(\al_1 - \al_2 = 3\).
Indeed, if \(\al_2 - \al_3 = 5\) then the nodes at the ends of rows \(2\) and \(3\) both have \spr \(\eps\), and so \(\al_2^{-\beps} - \al_{3}^{-\beps} = \al_2 - \al_{3} > 4\).
Thus \(\al^{-\beps}\) is not \fstepped, so to avoid a contradiction with the inductive hypothesis we have \(\al^{-\beps} = \al\) (that is, \(\al\) has no \esprms\beps).
In particular, we cannot have \(\al_1 = \al_2 + 5\) (for then row \(1\) would have a \esprm\beps).
So \(\al_1 \in \set{\al_2 + 1, \al_2 + 3}\).
But if \(\al_1 = \al_2 + 1\), then \(\al^{-\eps}_1 = \al_1 = \al_2 +1 > 4\), but \(\al^{-\eps}_2 = \al_2 - 2\) and \(\al^{-\eps}_3 \leq \al_3 = \al_2 -5\), contradicting the assumption that \(\al^{-\eps}\) is \fstepped.
So \(\al_1 - \al_2 = 3\) as required.
\qedhere
\end{enumerate}
\end{proof}

\subsection{Improving bounds by considering the last fully-occupied ladder}

We now tighten the bounds from the previous section by considering restrictions on the \emph{fully-occupied ladders} in \(\dbl\al\) imposed by the regularisation requirement (\Cref{lemma:first_consequences}\Cref{item:2-reg}).

Recall from \Cref{subsubsec:residues_and_ladders} that $\lad l$ denotes the $l$th ladder in $\bbn^2$.
For a partition \(\mu\), we say a ladder \(\lad{l}\) is \emph{fully-occupied} in  \(\mu\) if \(\lad{l} \subseteq \mu\).
For a strict partition \(\be\), let \(\flad{\be}\) be maximal such that \(\lad{\flad{\be}}\) is fully-occupied in \(\dbl{\be}\), and set \(\flad\vn=0\).
Trivially we find \(\flad{\be} \leq \dbl{\be}_1 = \ceil{\be_1/2}\).
We will show that for \(\al\) and \(\hatal\), this bound is in fact attained.

\begin{lemma}
\label{lemma:first_ladder_still_hit_going_down}
Suppose \(\be\in\scrd\), and let \(1 \leq i \leq \len{\be}-1\).
\begin{enumerate}[(i)]
    \item
    \label{item:final_node_ladder_still_hit_going_down}
Suppose either \(\beta_i\) is odd and \(\beta_i - \beta_{i+1} \leq 4\), or \(\beta_i\) is even and \(\beta_i - \beta_{i+1} \leq 5\).
Then the ladder that contains the node \((2i-1, \ceil{\beta_i/2})\) (that is, the final node in row $2i-1$ of $\dbl\be$) intersects \(\dbl\beta\) also in rows \(2i\), \(2i+1\) and \(2i+2\) (ignoring the last if \(\beta_{i+1} = 1\)).
    \item
    \label{item:penultimate_node_ladder_still_hit_going_down}
Suppose \(\beta_i\) is odd and \(\be_i - \be_{i+1} \leq 6\).
Then the ladder that contains the node \((2i-1, \floor{\beta_i/2})\) (that is, the penultimate node in row $2i-1$ of $\dbl\be$) intersects \(\dbl\beta\) also in rows \(2i\), \(2i+1\) and \(2i+2\) (ignoring the last if \(\beta_{i+1} = 1\)).
\end{enumerate}
\end{lemma}

\begin{proof}
\begin{enumerate}[(i),beginthm]
    \item 
The ladder that contains \((2i-1, \ceil{\beta_i/2})\) is \(\lad{\ceil{\beta_i/2}+2i-2}\), so \(\dbl{\be}\) meets this ladder in row \(j > 2i-1\) if and only if \(\dbl{\beta}_j \geq \ceil{\beta_i/2} +2i-1 -j\).
Clearly
\begin{align*}
\dbl{\beta}_{2i} &= \floor{\beta_i/2} \geq \ceil{\beta_i/2} - 1, \\
\intertext{while as illustrated in \Cref{subfig:final_node_even_ladders,subfig:final_node_odd_ladders} the assumptions imply that}
\dbl{\beta}_{2i+1} &= \ceil{\beta_{i+1}/2} \geq \ceil{\beta_i/2} - 2, \\
\dbl{\beta}_{2i+2} &= \floor{\beta_{i+1}/2} \geq \ceil{\beta_i/2} - 3.
\end{align*}

    \item 
The ladder that contains \((2i-1, \floor{\beta_i/2})\) is \(\lad{\floor{\beta_i/2}+2i-2}\), so \(\dbl{\be}\) meets this ladder in row \(j > 2i-1\) if and only if \(\dbl{\beta}_j \geq \floor{\beta_i/2} +2i-1 -j\).
Clearly
\begin{align*}
\dbl{\beta}_{2i} &= \floor{\beta_i/2} \geq \floor{\beta_i/2} - 1, \\
\intertext{while as illustrated in \Cref{subfig:penultimate_node_odd_ladders} the assumptions imply that}
\dbl{\beta}_{2i+1} &= \ceil{\beta_{i+1}/2} \geq \floor{\beta_i/2} - 2, \\
\dbl{\beta}_{2i+2} &= \floor{\beta_{i+1}/2} \geq \floor{\beta_i/2} - 3. \qedhere
\end{align*}
\end{enumerate}
\end{proof}

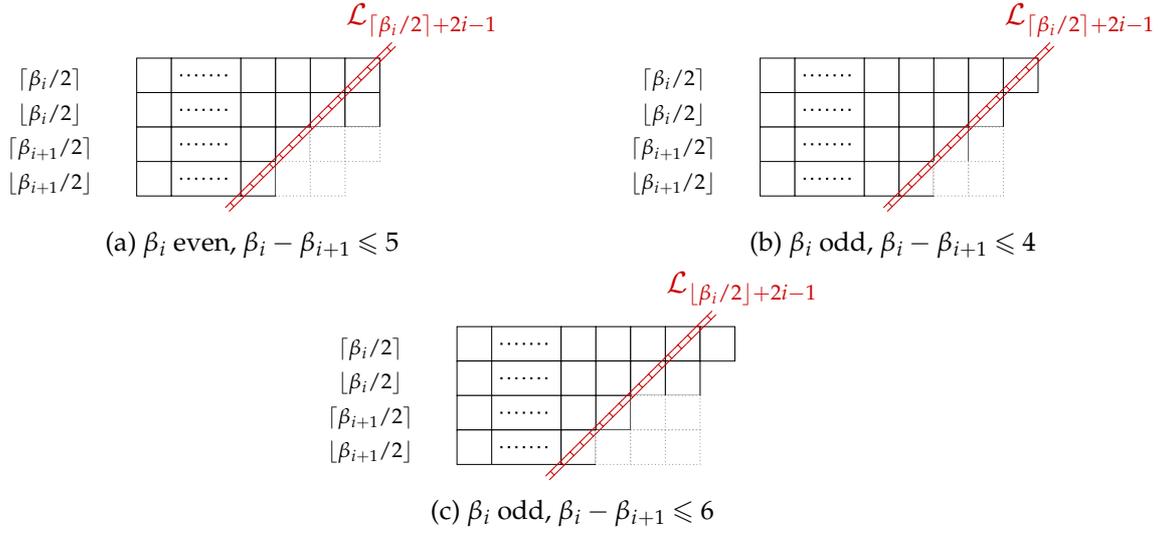
\begin{figure}
\centering
\begin{subfigure}[t]{0.49\textwidth}
\centering
\begin{tikzpicture}[x=13pt,y=13pt]
\tgyoung(0pt,0pt,%
    :<\scriptstyle\ceil{\be_i/2}>
    ::
    ;_2\hdts;;;;,
    :<\scriptstyle\floor{\be_i/2}>
    ::
    ;_2\hdts;;;;,
    :<\scriptstyle\ceil{\be_{i+1}/2}>
    ::
    ;_2\hdts;;!\Ypossible;;,!\Ydefault
    :<\scriptstyle\floor{\be_{i+1}/2}>
    ::
    ;_2\hdts;!\Ypossible;;
)
\draw[ladder auto, \colourC]
      (10.4,1.4) to (5.6,-3.4) node[color=\colourC,shift={(0.8,0.7)}]{\(\lad{\ceil{\be_i/2}+2i-1}\)};
\end{tikzpicture}
    \caption{\(\beta_i\) even, \(\beta_i - \beta_{i+1} \leq 5\)}
    \label{subfig:final_node_even_ladders}
\end{subfigure}
\begin{subfigure}[t]{0.49\textwidth}
\centering
\begin{tikzpicture}[x=13pt,y=13pt]
\tgyoung(0pt,0pt,%
    :<\scriptstyle\ceil{\be_i/2}>
    ::
    ;_2\hdts;;;;;,
    :<\scriptstyle\floor{\be_i/2}>
    ::
    ;_2\hdts;;;;,
    :<\scriptstyle\ceil{\be_{i+1}/2}>
    ::
    ;_2\hdts;;;!\Ypossible;,!\Ydefault
    :<\scriptstyle\floor{\be_{i+1}/2}>
    ::
    ;_2\hdts;;!\Ypossible;;
)
\draw[ladder auto, \colourC]
     (11.4,1.4) to (6.6,-3.4) node[color=\colourC, shift={(0.8,0.7)}]{\(\lad{\ceil{\be_i/2}+2i-1}\)};
\end{tikzpicture}
    \caption{\(\beta_i\) odd, \(\beta_i - \beta_{i+1} \leq 4\)}
    \label{subfig:final_node_odd_ladders}
\end{subfigure}
\begin{subfigure}[t]{0.5\textwidth}
\centering
\begin{tikzpicture}[x=13pt,y=13pt]
\tgyoung(0pt,0pt,%
    :<\scriptstyle\ceil{\be_i/2}>
    ::
    ;_2\hdts;;;;;,
    :<\scriptstyle\floor{\be_i/2}>
    ::
    ;_2\hdts;;;;,
    :<\scriptstyle\ceil{\be_{i+1}/2}>
    ::
    ;_2\hdts;;!\Ypossible;;,!\Ydefault
    :<\scriptstyle\floor{\be_{i+1}/2}>
    ::
    ;_2\hdts;!\Ypossible;;;
)
\draw[ladder auto, \colourC]
     (10.4,1.4) to (5.6,-3.4) node[color=\colourC, shift={(0.8,0.7)}]{\(\lad{\floor{\be_i/2}+2i-1}\)}; 
\end{tikzpicture}
    \caption{\(\beta_i\) odd, \(\beta_i - \beta_{i+1} \leq 6\)}
    \label{subfig:penultimate_node_odd_ladders}
\end{subfigure}
    \caption{An illustration of the ladder considered in \Cref{lemma:first_ladder_still_hit_going_down} hitting the following three rows. The dashed boxes indicate nodes that may or may not be present in the Young diagram.}
    \label{fig:full_ladders}
\end{figure}

The next \lcnamecref{lemma:flad_values} gives some basic results about $\flad\al$ and $\flad{\hatal}$.

\begin{lemma}
\label{lemma:flad_values}
\begin{enumerate}[(i), beginthm]
    \item
    \label{item:flad_for_X}
    \(\flad{\hatal} = \dbl{\hatal}_1 = \ceil*{\frac12 (\hatal)_1}\).
    \item
    \label{item:flad_for_alpha}
    \(\flad{\al} \in \set{ \dbl{\al}_1,\, \dbl{\al}_1 - 1}\).
    \item 
    \label{item:flad_for_alpha_alpha1_even}
    If \(\al_1\) is even, then \(\flad{\al} = \dbl{\al}_1 = \frac12 \al_1\).
\end{enumerate}
\end{lemma}

\begin{proof}
The strict partition \(\hatal\) is \fstepped and hence satisfies the conditions of \Cref{lemma:first_ladder_still_hit_going_down}\Cref{item:final_node_ladder_still_hit_going_down} for every \(i\).
Thus the ladder \(\lad{\dbl{\hatal}_1}\) intersects \(\dbl{\hatal}\) in every row, and so \(\flad{\hatal} = \dbl{\hatal}_1\) as required.

By \Cref{lemma:alpha1_vs_alpha2_bounds}, the strict partition \(\al\) also satisfies the conditions of \Cref{lemma:first_ladder_still_hit_going_down}\Cref{item:final_node_ladder_still_hit_going_down}, except possibly for a single value of \(i \in \set{1,2}\) with \(\al_i\) odd, \(\al_1 \not\equiv \al_2 \ppmod{2}\), and \(\al_1 - \al_2 = 3\) if \(i =2\).

If the condition fails for \(i=1\), then \(\al_1\) is odd and \(\al_1 - \al_2 = 5\), and so by \Cref{lemma:first_ladder_still_hit_going_down}\Cref{item:penultimate_node_ladder_still_hit_going_down} the ladder \(\lad{\dbl{\al}_1-1}\) intersects rows \(2\), \(3\) and \(4\) of \(\dbl{\al}\), and hence all rows of \(\dbl{\al}\) by \Cref{lemma:first_ladder_still_hit_going_down}\Cref{item:final_node_ladder_still_hit_going_down}.

If the condition fails for \(i=2\), then \(\al_1\) is even and \(\al_1 - \al_2 = 3\).
So by \Cref{lemma:first_ladder_still_hit_going_down}\Cref{item:final_node_ladder_still_hit_going_down}, the ladder \(\lad{\dbl{\al}}\) intersects rows \(2\), \(3\) and \(4\) of \(\dbl{\al}\), and moreover since \(\al_1 - \al_2 = 3\) it is easily seen that it intersects nodes in these rows strictly before the last.
Then by \Cref{lemma:first_ladder_still_hit_going_down}\Cref{item:penultimate_node_ladder_still_hit_going_down}, this ladder intersects also rows \(5\) and \(6\) of \(\dbl{\al}\), and hence all rows of \(\dbl{\al}\) by \Cref{lemma:first_ladder_still_hit_going_down}\Cref{item:final_node_ladder_still_hit_going_down}.

Thus in either case \(\flad{\al} \geq \dbl{\al}_1 - 1\), and if \(\al_1\) is even then \(\flad{\al} = \dbl{\al}_1\).
\end{proof}

The following is the key observation that, combined with the values of \(\flad{\al}\) and \(\flad{\hatal}\) identified above, allows us to improve the bounds on differences between parts of \(\al\).

\begin{lemma}
\label{lemma:flad_alpha_vs_alpha_hat}
\(\flad{\al} \leq \flad{\hatal} + 2\).
\end{lemma}

\begin{proof}
By \Cref{lemma:first_consequences}\Cref{item:2-reg}, the partitions $\la$ and $\dbl\al$ have the same $2$-regularisation, so $\flad\al$ is maximal such that $\lad{\flad{\al}} \subseteq\la$.
Similarly $\flad{\hatal}$ is maximal such that $\lad{\flad{\hatal}} \subseteq \la\hooksm\delta$.
Thus we can find a node $(r,c)\in\lad{\flad{\hatal}+1}$ which is not contained in $\la\hooksm\delta$.
Then the node $(r+1,c+1)$ cannot be contained in $\la$, because $\de$ contains at most one node on any diagonal, and hence cannot contain both $(r,c)$ and $(r+1,c+1)$. The node $(r+1,c+1)$ lies in $\lad{\flad{\hatal}+3}$, and so we deduce $\lad{\flad{\hatal}+3}\nsubseteq\la$, and therefore $\flad\al\ls\flad{\hatal}+2$.
\end{proof}

We can now deduce tighter bounds on the differences between parts of $\al$.

\begin{lemma}
\label{lemma:alpha1-alpha3_bounds}
\begin{enumerate}[(i),beginthm]
    \item
    \label{item:alpha1_even_alpha2_odd}
If \(\al_1\) is even and \(\al_2\) is odd, then \(\al_1 - \al_3 \leq 5\).
    \item
    \label{item:alpha1_odd_alpha2_even}
If \(\al_1\) is odd and \(\al_2\) is even, then \(\al_1 - \al_3 \leq 6\).
\end{enumerate}
\end{lemma}

\begin{proof}
In both cases, \(\al_1 \not\equiv \al_2 \ppmod{2}\) and so \((\hatal)_1 = \al_3\).
Then combining \cref{lemma:flad_values}\ref{item:flad_for_X} and \Cref{lemma:flad_alpha_vs_alpha_hat} gives \(\flad{\al} \leq \frac12(\al_3+1) + 2\).
\begin{enumerate}[(i)]
\item
Suppose \(\al_1\) is even and \(\al_2\) is odd.
Then \Cref{lemma:flad_values}\ref{item:flad_for_alpha_alpha1_even} gives \(\flad{\al} = \frac12\al_1\).
Thus \(\frac12 \al_1 \leq \frac12(\al_3+1) + 2\), which gives the claim.

\item
Suppose \(\al_1\) is odd and \(\al_2\) is even.
Then  \Cref{lemma:flad_values}\ref{item:flad_for_alpha} gives \(\flad{\al} \geq \dbl{\al}_1 - 1 = \frac12(\al_1-1)\).
Thus \(\frac12(\al_1-1) \leq \frac12(\al_3 + 1) + 2\), which gives the claim.
\qedhere
\end{enumerate}
\end{proof}

\begin{propn}
\label{prop:final_bounds}
\begin{enumerate}[(i),beginthm]
    \item
    \label{item:best_alpha1_alpha2_bounds}
Let \(1 \leq i \leq \len{\al}\).
Then \(\al_i - \al_{i+1} \leq 4\).
    \item
    \label{item:improved_flad_alpha}
\(\flad{\al} = \dbl{\al}_1 = \ceil*{\frac12\al_1}\).
    \item
    \label{item:improved_alpha1-alpha3_bound_when_alpha1_odd_alpha2_even}
If \(\al_1\) is odd and \(\al_2\) is even, then \(\al_1 - \al_3 \leq 4\).
\end{enumerate}
\end{propn}

\begin{proof}
\begin{enumerate}[(i),beginthm]
    \item
The only cases not covered by \Cref{lemma:alpha1_vs_alpha2_bounds} are as follows.
\begin{description}
\item[\(\al_1\) is even, \(\al_2\) is odd, \(i\in\set{1,2}\):]
By \Cref{lemma:alpha1-alpha3_bounds}\Cref{item:alpha1_even_alpha2_odd} we have \(\al_1 - \al_3 \leq 5\).
Then \(\al_1 - \al_2 \leq 4\) and \(\al_2 - \al_3 \leq 4\) as required.

\item[\(\al_1\) is odd, \(\al_2\) is even, \(i=1\):]
By \Cref{lemma:alpha1-alpha3_bounds}\Cref{item:alpha1_odd_alpha2_even} we have \(\al_1 - \al_3 \leq 6\).
Suppose towards a contradiction \(\al_1 - \al_{2} = 5\); then \(\al_3 = \al_1 - 6\).
The nodes at the end of rows \(1\) and \(2\) have the same \spr, say \(\eps\), and the node at the end of row \(3\) has distinct \spr \(\beps\).
If the node at the end of row \(3\) is \spre, then \(\al^{-\beps}\) is not \fstepped, contradicting the inductive hypothesis.
So we must have \(\al_3 \geq 3\) and \(\al_3 - 1, \al_3 -2 \in \al\).
But then \(\al_3, \al_3 -2 \in \hatal\), contradicting the assumption that \(\hatal\) is \fsemic.
\end{description}

\item 
By part \Cref{item:best_alpha1_alpha2_bounds}, we now have that the conditions of \Cref{lemma:first_ladder_still_hit_going_down}\Cref{item:final_node_ladder_still_hit_going_down} are satisfied for \(\al\) for every \(i\), and so as in the proof of \Cref{lemma:flad_values} the ladder \(\lad{\dbl{\al}_1}\) intersects every row of \(\dbl{\al}\).

\item
Suppose \(\al_1\) is odd and \(\al_2\) is even.
As in the proof of \Cref{lemma:alpha1-alpha3_bounds}, combining \cref{lemma:flad_values}\ref{item:flad_for_X} and \Cref{lemma:flad_alpha_vs_alpha_hat} gives \(\flad{\al} \leq \frac12(\al_3+1) + 2\).
Part \Cref{item:improved_flad_alpha} gives \(\flad{\al} = \dbl{\al}_1 = \frac12(\al_1+1)\).
Thus \(\frac12(\al_1+1) \leq \frac12(\al_3 + 1) + 2\), which gives the claim.
\qedhere
\end{enumerate}
\end{proof}

\subsection{Deducing 4-semicongruence by considering the last partially-occupied ladder}

We also gain information on \(\al\) by considering the ladders which non-trivially intersect \(\dbl\al\).

For a strict partition \(\be\), let \(\tlad{\be}\) be maximal such that \(\dbl{\be}\cap\lad{\tlad\be}\) is nonempty, and set \(\tlad\vn=0\).
Trivially
\[
\tlad{\be} \geq \len{\dbl{\be}} = \begin{cases}
    2\len{\be} & \text{ if \(1 \not\in \be\),} \\
    2\len{\be} - 1 & \text{ if \(1 \in \be\).}
\end{cases}
\]
We will show that for \(\al\) and \(\hatal\) this bound is attained or nearly attained.

\begin{lemma}
\label{lemma:more_ladders_hit_going_down}
Let \(\be\) be any strict partition, and let \(1 \leq i \leq \len{\be}-1\). Suppose either \(\be_i\) is even and \(\be_i - \be_{i+1} \leq  4\), or \(\be_i\) is odd and \(\be_i - \be_{i+1} \leq  5\). Then the largest ladder that meets row $2i-1$ or $2i$ of $\dbl\be$ is smaller than or equal to the largest ladder that meets row $2i+1$ or $2i+2$ of $\dbl\be$.
\end{lemma}

\begin{proof}
The largest ladder that meets row $2i-1$ or $2i$ is ladder $\lad{\lfloor\be_i/2\rfloor+2i-1}$, while the largest ladder that meets rows $2i+1$ or $2i+2$ is ladder $\lad{\lfloor\be_{i+1}/2\rfloor+2i+1}$.
The assumptions on $\be_i$ and $\be_{i+1}$ imply that $\lfloor\be_i/2\rfloor\ls\lfloor\be_{i+1}/2\rfloor+2$, which is what we need, as illustrated in \Cref{fig:hit_ladders}.
\end{proof}

\begin{figure}[ht]
\centering
\begin{subfigure}[t]{0.49\textwidth}
\centering
\begin{tikzpicture}[x=13pt,y=13pt]
\tgyoung(0pt,0pt,%
    :<\scriptstyle\ceil{\be_i/2}>
    ::
    ;_2\hdts;;;,
    :<\scriptstyle\floor{\be_i/2}>
    ::
    ;_2\hdts;;;,
    :<\scriptstyle\ceil{\be_{i+1}/2}>
    ::
    ;_2\hdts;!\Ypossible;;,!\Ydefault
    :<\scriptstyle\floor{\be_{i+1}/2}>
    ::
    ;_2\hdts;!\Ypossible;
)
\draw[ladder auto, \colourC]
      (10.4,1.4) to (5.6,-3.4) node[color=\colourC,shift={(0.8,0.7)}]{\(\lad{\floor{\be_i/2}+2i-1}\)};
\end{tikzpicture}
    \caption{\(\beta_i\) even, \(\beta_i - \beta_{i+1} \leq 4\)}
    \label{subfig:longest_ladder_even_parts}
\end{subfigure}
\begin{subfigure}[t]{0.49\textwidth}
\centering
\begin{tikzpicture}[x=13pt,y=13pt]
\tgyoung(0pt,0pt,%
    :<\scriptstyle\ceil{\be_i/2}>
    ::
    ;_2\hdts;;;;,
    :<\scriptstyle\floor{\be_i/2}>
    ::
    ;_2\hdts;;;,
    :<\scriptstyle\ceil{\be_{i+1}/2}>
    ::
    ;_2\hdts;!\Ypossible;;,!\Ydefault
    :<\scriptstyle\floor{\be_{i+1}/2}>
    ::
    ;_2\hdts;!\Ypossible;;
)
\draw[ladder auto, \colourC]
     (10.4,1.4) to (5.6,-3.4) node[color=\colourC, shift={(0.8,0.7)}]{\(\lad{\floor{\be_i/2}+2i-1}\)}; 
\end{tikzpicture}
    \caption{\(\beta_i\) odd, \(\beta_i - \beta_{i+1} \leq 5\)}
    \label{subfig:longest_ladder_odd_parts}
\end{subfigure}
    \caption{An illustration of the largest ladder meeting row \(2i-1\) or \(2i\), considered in \Cref{lemma:more_ladders_hit_going_down}, being met and possibly exceeded by row \(2i+1\) or \(2i+2\).
    The dashed boxes indicate nodes that may or may not be present in the Young diagram.
    }
    \label{fig:hit_ladders}
\end{figure}
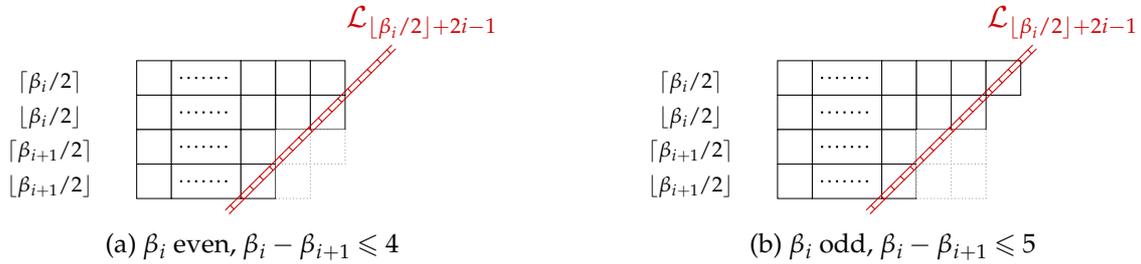

\begin{propn}
\label{prop:tlad_values}
Let \(\beta \in \set{\al, \hatal}\).
Then
\[
    \tlad{\beta} = \begin{cases}
        \len{\dbl{\beta}} & \text{ if \(1\), \(2\) or \(3\) is a part of \(\beta\);} \\
        \len{\dbl{\beta}}+1  & \text{ otherwise.} 
    \end{cases}
\]
In particular, if \(\al\) has at least two odd parts, then \(\tlad{\al} = \len{\dbl{\al}}\). 
\end{propn}

\begin{proof}
Both \(\hatal\) and \(\al\) satisfy the conditions of \cref{lemma:more_ladders_hit_going_down} for every $i$: \(\hatal\) does so since it is \fstepped, and \(\al\) does so by \Cref{prop:final_bounds}\Cref{item:best_alpha1_alpha2_bounds}. Thus the largest ladder that meets $\beta$ meets it in its final row. If the final part of \(\beta\) is \(1\), \(2\) or \(3\) then this is the \(\len{\dbl{\beta}}\) ladder, while if the final part of \(\beta\) is \(4\) then this is the \(\len{\dbl{\beta}}+1\) ladder. The \fstepped property of \(\hatal\) and the bounds on \(\al\) imply that the final part of $\beta$ is at most $4$, so these are all possible cases.

The ``in particular'' follows since if \(\al\) has at least two odd parts then \(\hatal\) has at least one odd part, and since it is \fstepped therefore has \(1\) or \(3\) as a part, and hence so does \(\al\).
\end{proof}

\begin{lemma}
\label{lemma:tlad_alpha_vs_alpha_hat}
\(\tlad{\al} \geq \tlad{\hatal} + 3\).
\end{lemma}

\begin{proof}
It suffices to show \(\len{\dbl{\al}} \geq \tlad{\hatal}+3\).
If \(\al\) has an odd part greater than \(1\), then \(\len{\dbl{\al}} = {\len{\dbl{\hatal}} + 4}\), and 
\Cref{prop:tlad_values} gives \(\len{\dbl{\hatal}} +1 \geq \tlad{\hatal}\), yielding the result.
Otherwise, the only odd part of \(\al\) is \(1\), and \(\len{\dbl{\al}} = \len{\dbl{\hatal}} + 3\).
By \Cref{lemma:1_only_odd_2_not_in_hat}, we may assume \(2 \in \hatal\) in this case.
Then \(\len{\dbl{\hatal}} = \tlad{\hatal}\) by \Cref{prop:tlad_values}, and the result follows.
\end{proof}

Now recall the rim-hook $\de$. Because $|\de|$ is odd, we can write $|\de|=2a+1$. The residues of the nodes in a rim-hook alternate between $0$ and $1$ from one end of the rim-hook to the other, so the multiset of residues of the nodes of $\de$ is either $\{0^{a+1},1^a\}$ or $\{0^a,1^{a+1}\}$.

\begin{lemma}
\label{lemma:delta_meets_tlad}
The rim-hook \(\delta\) meets \(\lad{\tlad{\al}}\) in the first row or column.
In particular, the multiset of residues of \(\delta\) is \(\{0^{a+1},1^a\}\)  if \(\tlad{\al}\) is odd, or \(\{0^a,1^{a+1}\}\) if \(\tlad{\al}\) is even.
\end{lemma}

\begin{proof}
That \(\delta\) meets the ladder \(\lad{\tlad{\al}}\) is clear from the fact that \(\tlad{\al} > \tlad{\hatal}\). Since furthermore \(\tlad{\al} \geq \tlad{\hatal} + 3\)  by \Cref{lemma:tlad_alpha_vs_alpha_hat}, we deduce that \(\la\hooksm\delta\) does not contain any nodes in the ladder \(\lad{\tlad{\al}-2}\), and hence the only possible nodes of $\lad{\tlad\al}$ that can belong to \(\delta\) are $(1,\tlad\al)$ and $(\tlad\al,1)$. So $\de$ meets \(\lad{\tlad{\al}}\) in the first row or column.
\end{proof}

\begin{propn}
\label{prop:odd_congruence_of_alpha}
The odd parts of \(\al\) are congruent modulo \(4\) (that is, \(\al\) is \fsemic).
\end{propn}

\begin{proof}
If \(\al\) has only one odd part, this is trivially true. So suppose \(\al\) has at least two odd parts. Since $\hatal$ is \fsemic, all the odd parts of $\al$ except possibly the largest are congruent modulo $4$, so it suffices to show that the largest odd part of $\al$ is congruent to the smallest odd part of $\al$.
Write \(2k+1\) for the largest odd part of \(\al\). Note that the smallest odd part of \(\al\) is the smallest odd part of \(\hatal\), and so (since \(\hatal\) is \fstepped) is equal to \(1\) or \(3\).

Recall that (spin-)contents (i.e.\ multisets of (spin-)residues) of partitions characterise blocks (\cref{blocksec}). Since \(\bspn{\al}\) and \(\bspe{\la}\) are proportional, they lie in the same block, so the spin-content of \(\al\) is equal to the content of \(\la\); likewise the spin-content of \(\hatal\) is equal to the content of \(\la\hooksm\de\). It follows that the multiset of \sprs of \(\pi\) is equal to the multiset of residues of \(\de\). Write \(2a+1\) for the size of \(\de\).

Observe that an even part of \(\al\) has equal numbers of nodes of \spr \(0\) and \(1\), a part congruent to \(1\ppmod{4}\) has one more node of \spr \(0\) than \(1\), and a part congruent to \(3\ppmod{4}\) has one more node of \spr \(1\) than \(0\).
Thus we have
\begin{align*}
2k+1 \equiv 1 \ppmod{4}
    &\Iff \text{\(\ulob\) has multiset of \spr{}s \(\{0^{a+1}, 1^a\}\)} && \text{(by the above observation)}\\
    &\Iff \text{\(\de\) has multiset of residues \(\{0^{a+1}, 1^a\}\)} && \text{(as contents characterise blocks)} \\
    &\Iff \text{\(\tlad{\al}\) is odd}                                && \text{(by \Cref{lemma:delta_meets_tlad})} \\
    &\Iff \text{\(\len{\dbl{\al}}\) is odd}                           && \text{(by \Cref{prop:tlad_values})} \\
    &\Iff \text{\(1\) is the smallest odd part of \(\al\)}.
\end{align*}
Therefore the largest odd part of \(\al\) is congruent modulo \(4\) to the smallest, as required.
\end{proof}

The congruence of the odd parts of \(\al\) also allows us to make the following useful deductions.

\begin{cory}
\label{cor:semicongruence_consequences}
\begin{enumerate}[(i),beginthm]
    \item\label{item:removable_nodes_in_odd_parts}
Suppose $1\ls i\ls\len\al$. If $\al_i$ is odd, then the node $(i,\al_i)$ at the end of row \(i\) is \spre.
    \item\label{item:even_parts_reflected_by_restriction}
Let \(\eps \in \set{0,1}\) and let \(m \geq 2\) be even.
If \(m \in \al^{-\eps}\), then \(m \in \al\).
\end{enumerate}
\end{cory}

\begin{proof}
\begin{enumerate}[(i),beginthm]
    \item 
The only way the node $(i,\al_i)$ can fail to be \spre is if $\al_i\gs3$ and both \(\al_i-1\) and $\al_i-2$ are parts of $\al$. But this contradicts \cref{prop:odd_congruence_of_alpha}.

    \item
Let \(j\) be such that \(\al^{-\eps}_j=m\).
If the node $(j,\al_j)$ at the end of row \(j\) in \(\al\) has \spr \(\beps\), then \(\al_j = m\) also.
Otherwise, \((j,\al_j)\) has \spr \(\eps\), and either \(\al_j = m\), in which case the claim holds, or \(\al_j = m+1\).
In this remaining case, the penultimate node in row \(j\), the node \((j,m)\), is not \spre, and so either \(\al_{j+1} = m\) or \(\al_{j+1} = m-1\).
But if \(m\) is even then \(m+1\) and \(m-1\) cannot both be parts of \(\al\) by \Cref{prop:odd_congruence_of_alpha}, so \(\al_{j+1} = m\).
\qedhere
\end{enumerate}
\end{proof}

\subsection{Ruling out configurations using hook lengths}

Before we can complete the proof of \Cref{onlyif} by showing that \(\al\) is \fstepped, we first rule out several tricky configurations by considering restrictions on hook lengths in \(\la\).

Recall that by assumption $\la$ has a unique rim-hook $\de$ of length $|\pi|$.

\begin{lemma}
\label{lemma:long_hook_fills_two_ladders}
If \(\lende \geq \tlad{\al} + \flad{\hatal} + 1\),
then \(\flad{\al} = \flad{\hatal} + 2\).
\end{lemma}

\begin{proof}
Observe that if \(\delta\) meets all of the diagonals intersecting the ladders \(\lad{\flad{\hatal} + 1}\) and \(\lad{\flad{\hatal} + 2}\), then these ladders are both fully-occupied in $\la$ (using the fact that $\lad{\flad{\hatal}}$ is already fully-occupied in \(\la\hooksm\delta\)). Thus if \(\delta\) meets these diagonals, then \(\flad{\al} \leq \flad{\hatal} + 2\), which must then hold with equality by \Cref{lemma:flad_alpha_vs_alpha_hat}.

So suppose \(\lende \geq \tlad{\al} + \flad{\hatal} + 1\), and it suffices to show that \(\de\) meets the diagonals intersecting the ladders \(\lad{\flad{\hatal} + 1}\) and \(\lad{\flad{\hatal} + 2}\), which are the diagonals
\[
    -(\flad{\hatal}+1), \ldots,-1, 0,1, \ldots, \flad{\hatal}+1.
\]

We know from \Cref{lemma:delta_meets_tlad} that \(\delta\) meets the ladder \(\lad{\tlad{\al}}\) in either the first column or first row. Without loss of generality, suppose it is the first column; then the bottom-most node (the ``foot node'') of \(\delta\) is $(\tlad\al,1)$, which lies on the \(1-\tlad{\al}\) diagonal.
Furthermore, since a rim-hook meets each diagonal at most once, \(\delta\) meets all diagonals from $1-\tlad\al$ to $\lende-\tlad\al$ inclusive. 
So provided \(\lende \geq \tlad{\al} + \flad{\hatal} + 1\) as assumed (and also \(\tlad{\al} -1 \geq \flad{\hatal}+1\), which holds since \(\tlad{\al} \geq \tlad{\hatal} + 3 \geq \flad{\hatal}+3\)), the desired diagonals are met.
\end{proof}

\begin{propn}
\label{prop:difficult_even_cases}
\(\al\) is not of the form:
\begin{enumerate}[(i)]
    \item 
    \((4k+2) \sqcup \gamma \sqcup (4k, 4k-1, 4k-4, 4k-5, \ldots, 4,3)\) where \(k \geq 1\) and \(\gamma\) is a strict partition with all parts congruent to \(2\ppmod{4}\) and strictly less than \(4k-2\); or
    \item 
    \((4k+4) \sqcup \gamma \sqcup (4k+2, 4k+1, 4k-2, 4k-3, \ldots, 2,1)\) where \(k \geq 1\) and \(\gamma\) is a strict partition with all parts congruent to \(0\ppmod{4}\) and strictly less than \(4k\).
\end{enumerate}
\end{propn}

\begin{proof}
In both cases we suppose that \(\al\) is of the given form and find a contradiction with \Cref{lemma:long_hook_fills_two_ladders}.
\begin{enumerate}[(i)]

    \item
We compute (using \Cref{lemma:flad_values,prop:tlad_values} if desired)
\begin{align*}
\tlad{\al} &= \len{\dbl{\al}} = 2(2k+1+\len{\gamma}) \leq 6k, \\
\flad{\al} &= \dbl{\al}_1 = 2k+1, \\
\flad{\hatal} &= \dbl{\hatal}_1 = 2k, \\
\lende &= \abs{\pi} = 8k + 1. \\
\intertext{
But then \(\tlad{\al} + \flad{\hatal} + 1 \leq 8k + 1 = \lende\) but \(\flad{\al} \neq \flad{\hatal} + 2\), contradicting \Cref{lemma:long_hook_fills_two_ladders}.
    \item 
We compute (using \Cref{lemma:flad_values,prop:tlad_values} if desired)}
\tlad{\al} &= \len{\dbl{\al}}-1 = 2(2k+3+\len{\gamma})-1 \leq 6k+3, \\
\flad{\al} &= \dbl{\al}_1 = 2k+2, \\
\flad{\hatal} &= \dbl{\hatal}_1 = 2k+1, \\
\lende &= \abs{\pi} = 8k +5.
\end{align*}
But then \(\tlad{\al} + \flad{\hatal} +1 \leq 8k+5 = \lende\) but \(\flad{\al} \neq \flad{\hatal} + 2\), contradicting \Cref{lemma:long_hook_fills_two_ladders}. \qedhere
\end{enumerate}
\end{proof}

The following is in fact a statement about the longest odd hook length in any partition \(\la\) with given \(2\)-regularisation.

\begin{lemma}
\label{lemma:delta_doesnt_meet_too_many_diagonals}
In all cases, \(\lende \leq 2\tlad{\al} - 1\).
Moreover:
\begin{enumerate}[(i)]
    \item
    \label{item:last_ladder_one_node}
if \(\dbl{\al}\) contains only a single node in \(\lad{\tlad{\al}}\), then \(\lende < 2\tlad{\al} -1\);
    \item 
    \label{item:ladders_with_two_nodes}
if \(\dbl{\al}\) contains only a single node in \(\lad{\tlad{\al}}\) and exactly two nodes in each of \(\lad{\tlad{\al}-1}\) and \(\lad{\tlad{\al}-2}\), then \(\lende < 2\tlad{\al}-3\).
\end{enumerate}
\end{lemma}

\begin{proof}
The diagonals that intersect the first \(\tlad{\al}\) ladders are the diagonals \[-(\tlad{\al}-1), \ldots, -1, 0, 1, \ldots, \tlad{\al} - 1.\]
Thus \(\la\) meets at most \(2\tlad{\al} - 1\) diagonals. Since $\de$ meets each diagonal at most once, we have \(\lende \leq 2\tlad{\al} -1\). In order for equality to be achieved, $\de$ must meet both the \(-(\tlad{\al}-1)\) and \(\tlad{\al}-1\) diagonals, but their intersection with the first \(\tlad{\al}\) ladders is with only \(\lad{\tlad{\al}}\); thus \(\lad{\tlad{\al}}\) must contains at least two nodes of $\la$, proving (i).

Now suppose \(\dbl\al\) is as described in (ii).
By part (i), we have \(\lende < 2\tlad{\al} - 1\); suppose towards a contradiction \(\la\) contains a rim-hook \(\rho\) of length \(2\tlad{\al} - 3\) (because \(\lende\) is necessarily odd, this will suffice).
Then \(\rho\) must meet both the diagonals \(\tlad{\al}-3\) and \(-(\tlad{\al}-3)\), and hence \(\la\) must contain the nodes at start of each of these diagonals: \((1,\tlad{\al}-2)\) and \((\tlad{\al}-2,1)\).
These nodes lie in the ladder \(\lad{\tlad{\al}-2}\), which by assumption has exactly two nodes in \(\dbl{\al}\); thus the nodes \((1,\tlad{\al}-2)\) and \((\tlad{\al}-2,1)\) are the only nodes in \(\la \cap \lad{\tlad{\al}-1}\) and furthermore are contained in \(\rho\), and thus \(\rho\) is the largest rim-hook in \(\la\).
It follows that the nodes in \(\la\) in the ladders \(\lad{\tlad{\al}-1}\) and \(\lad{\tlad{\al}}\) must also be at the start of their diagonals (that is, in the first row or the first column), and \(\rho\) has length \(2\tlad{\al}-2\), a contradiction.
\end{proof}

\begin{propn}
\label{prop:difficult_five_cases}
\(\al\) is not of the form:
\begin{enumerate}[(i)]
    \item 
    \((4k, 4k-4, \ldots, 4) \sqcup (5)\) where \(k \geq 1\); or
        \item 
    \((4k, 4k-4, \ldots, 4) \sqcup (5,2)\) where \(k \geq 1\).
\end{enumerate}
\end{propn}

\begin{proof}
In both cases we suppose that \(\al\) is of the given form and find a contradiction with \Cref{lemma:delta_doesnt_meet_too_many_diagonals}.
In both cases \(\lende = \abs{\pi} = 4k+5\).
\begin{enumerate}[(i)]
    \item
Observe that \(\tlad{\al} = \len{\dbl{\al}} + 1 = 2k+3\) and that \(\dbl{\al}\) contains only a single node of \(\lad{\tlad{\al}}\). But also \(\lende = 2\tlad{\al} - 1\),  contradicting \Cref{lemma:delta_doesnt_meet_too_many_diagonals}\Cref{item:last_ladder_one_node}.

    \item 
Observe that \(\tlad{\al} = \len{\dbl{\al}} = 2k+4\) and that \(\dbl{\al}\) contains only a single node of \(\lad{\tlad{\al}}\). Furthermore \(\lad{\tlad{\al}-1}\) and \(\lad{\tlad{\al}-2}\) contain exactly two nodes of \(\dbl{\al}\) each. But also \(\lende = 2\tlad{\al}-3\), contradicting \Cref{lemma:delta_doesnt_meet_too_many_diagonals}\Cref{item:ladders_with_two_nodes}. \qedhere
\end{enumerate}
\end{proof}

\subsection{4-stepped property}

To complete our proof of \Cref{onlyif}, we show that $\al_i-4\in\al$ whenever $\al_i>4$.

\begin{lemma}
\label{lemma:helping_lemma_for_closure_when_alphai_odd}
Let \(\al_i\) be the largest odd part of \(\al\), and suppose \(\al_i > 4\) and \(\al_i - 4 \not\in \al\).
Then \(\al_{i+1} = \al_i-1\), and if \(\al_i -3 \in \al\) then \(\al_i + 1 \not\in\al\).
\end{lemma}

\begin{proof}
Our bound on consecutive parts \(\al_i - \al_{i+1} \leq 4\) from \Cref{prop:final_bounds}\Cref{item:best_alpha1_alpha2_bounds}, together with the assumption \(\al_i - 4 \not\in\al\) and that \(\al_i - 2 \not\in \al\) by \Cref{prop:odd_congruence_of_alpha}, gives \(\al_{i+1} \in \set{\al_i - 1, \al_i - 3}\).

First suppose \(i=1\).
Clearly \(\al_1 +1 \not\in\al\).
Meanwhile if \(\al_1 - 1 \not\in\al\) we have \(\al_2 = \al_1 - 3\);
then by \Cref{prop:final_bounds}\Cref{item:improved_alpha1-alpha3_bound_when_alpha1_odd_alpha2_even} we have \(\al_3 = \al_1 -4\), contradicting the assumption on \(\al\).

So suppose \(i>1\), and let \(\eps\) be the \spr of the node at the end of row \(i\) in \(\al\).
If \(\al_i - 3 \not\in \al\), then \(\al_{i+1} = \al_i-1\); it remains to suppose \(\al_i -3  \in \al\) and show that \(\al_i - 1 \in\al\) and \(\al_i+1 \not\in\al\).
Observe that, since \(\al_i - 4 \not\in\al\), there is a \esprm\beps in column \(\al_i - 3\) (and so the inductive hypothesis applies to \(\al^{-\beps}\)).
Since also \(\al_i - 2 \not\in \al\),
we deduce that \(\al_i - 3 \not\in \al^{-\beps}\).
Now, \(\al_{i-1}\) is even and \(\al_{i-1} - \al_i \leq 4\),
so \(\al_{i-1} \in \set{ \al_i + 1, \al_i+3}\).
In either case, the node at the end of row \(i-1\) is not a \esprm\beps, and thus \(\al_{i-1}^{-\beps} = \al_{i-1}\).
So if \(\al_{i-1} = \al_i + 1\), then \(\al_{i-1}^{-\beps} =\al_i+1\) but \(\al_i - 3 \not\in \al^{-\beps}\), contradicting the assumption that \(\al^{-\beps}\) is \fstepped; thus \(\al_i+1 \not\in\al\), and \(\al_{i-1} = \al_i + 3\).
Then \(\al^{-\beps}_{i-1} = \al_i + 3\) and hence \(\al_i - 1 \in \al^{-\beps}\), and so \(\al_i - 1 \in \al\) by \Cref{cor:semicongruence_consequences}\Cref{item:even_parts_reflected_by_restriction}.
\end{proof}

\begin{propn}
\label{prop:closure_when_alphai_odd}
If \(\al_i\) is odd and \(\al_i > 4\), then \(\al_i - 4 \in \al\).
\end{propn}

\begin{proof}
Since \(\hatal\) is \fstepped, the \lcnamecref{prop:closure_when_alphai_odd} holds for all odd parts except possibly the largest. So suppose \(\al_i\) is the largest odd part of \(\al\).
Suppose towards a contradiction \(\al_i - 4 \not \in\al\).
By \Cref{lemma:helping_lemma_for_closure_when_alphai_odd} we have \(\al_{i+1} = \al_i - 1\).
Let \(\eps\) be the \spr of the nodes \((i,\al_i)\) and \((i+1, \al_i-1)\) at the ends of rows \(i\) and \(i+1\); note that these nodes are \spre, so \(\al^{-\eps}\) is \fsas by the inductive hypothesis.

Consider first the case \(\al_i \geq 7\).
We have \(\al^{-\eps}_i = \al_i -1\) and \(\al^{-\eps}_{i+1} = \al_i - 2\).
Then since \(\al^{-\eps}\) is \fstepped and \(\al_i \geq 7\), we have \(\al_i -5, \al_i - 6 \in \al^{-\eps}\).
So $\al$ contains at least two of the integers \(\al_i-4\), \(\al_i-5\) and \(\al_i-6\). 
But \(\al_i - 6 \not\in \al\) by \cref{prop:odd_congruence_of_alpha}, so \(\al_i - 4 \in \al\) as required.

Now consider the case \(\al_i = 5\) (and hence \(\eps = 0\)).
Since any part congruent to \(2\) or \(3\) modulo \(4\) has final node of \spr \(1\), the parts of \(\al\) congruent to \(2\) modulo \(4\) are precisely the same as the parts of $\al^{-0}$ congruent to $2$ modulo $4$.
Since $\al^{-0}$ is \fstepped, these parts must therefore be $2,6,\dots,4r-2$ for some $r \geq 0$.
But by \Cref{lemma:helping_lemma_for_closure_when_alphai_odd} we have \(r \leq 1\).
Thus \(\al\) is of the form \(\al = (4k, 4k-4, \ldots, 8, 5, 4)\) or \(\al = (4k, 4k-4, \ldots, 8, 5, 4, 2)\) for some \(k \geq 1\), contradicting \Cref{prop:difficult_five_cases}.
\end{proof}

\begin{propn}\label{evensub4}
If \(\al_i\) is even and \(\al_i > 4\), then \(\al_i - 4 \in \al\).
\end{propn}

\begin{proof}
Since \(\hatal\) is \fstepped, the claim holds for all even parts except possibly the largest. So assume $\al_i$ is the largest even part of \(\al\). Let \(\eps\) be the \spr of the node $(i,\al_i)$.

By \Cref{cor:semicongruence_consequences}\Cref{item:even_parts_reflected_by_restriction}, it suffices to show that either \(\al^{-\eps}\) or $\al^{-\beps}$ has a part equal to $\al_i-4$. 
Since \(\al_i \in \al^{-\beps}\), it also suffices to show \(\al\) has a \esprm\beps, for then the inductive hypothesis implies that \(\al^{-\beps}\) is \fstepped and hence \(\al_i - 4 \in \al^{-\beps}\).

If \(\al\) has any part congruent to \(\al_i-1 \ppmod{4}\), then \(\al\) has a \esprm\beps by \Cref{cor:semicongruence_consequences}\Cref{item:removable_nodes_in_odd_parts}, and we are done.
Meanwhile, if \(\al\) has any part \(\al_j\equiv\al_i - 2 \ppmod{4}\) with \(\al_j - 1 \not\in \al\), then \(\al\) has a \esprm\beps, so again we are done.

Consider the case \(i > 1\).
Then \(\al_{i-1}\) is odd and \(\al_{i-1} - \al_i \leq 4\) (by \Cref{prop:final_bounds}\Cref{item:best_alpha1_alpha2_bounds}), so \(\al_{i-1} \in \set{\al_i + 1, \al_i + 3}\). If \(\al_{i-1} = \al_i + 3\) then we can use the argument from the last paragraph, so assume \(\al_{i-1} = \al_i + 1\). Then by \Cref{cor:semicongruence_consequences}\Cref{item:removable_nodes_in_odd_parts} \(\al\) has a \esprm\ep in row $i-1$, so the inductive hypothesis implies that $\al^{-\eps}$ is \fstepped.
Since both \(\al_i, \al_i + 1 \in \al\) we have \(\al_i \in \al^{-\eps}\), so \(\al_i - 4 \in \al^{-\eps}\) and we are done.

So we may assume \(i=1\).
\Cref{prop:final_bounds}\Cref{item:best_alpha1_alpha2_bounds} gives \(\al_1 - \al_{2} \leq 4\).
Clearly if \(\al_{2} = \al_1 - 4\) the lemma holds, while if \(\al_{2} = \al_1 -1\) we can use the argument from earlier in the proof. So assume \(\al_{2} \in \set{\al_1 - 2, \al_1 - 3}\).

Suppose \(\al_{2} = \al_1 -3\).
Then \(\al_2\) is odd and, using \Cref{lemma:alpha1-alpha3_bounds}\ref{item:alpha1_even_alpha2_odd}, we have \(\al_3 \in \set{\al_1 -4, \al_1 -5}\).
But we cannot have \(\al_1 - 3, \al_1 -5 \in \al\) by \Cref{prop:odd_congruence_of_alpha}, so \(\al_3 = \al_1 - 4\) and we are done.

We are left with the case \(\al_2 = \al_1 - 2\).
Then \(\al_1 - 2 \in \hatal\), and hence \(\al_1 -2, \al_1 - 6, \ldots \in \al\). From the third paragraph of the proof, we can also assume \(\al_1 - 3, \al_1 - 7, \ldots \in \al\).
Thus if the \lcnamecref{evensub4} is not true, then \(\al\) is of the form 
\begin{align*}
\al &= (4k+2, 4k, 4k-1, 4k-4, 4k-5, \ldots, 4, 3) \sqcup \ga
\\
\intertext{where \(k \geq 1\), and \(\ga\) is a strict partition with all parts congruent to \(2 \ppmod{4}\) and strictly less than \(4k-2\), or}
\al &= (4k+4, 4k+2, 4k+1, 4k-2, 4k-3, \ldots, 2,1) \sqcup \ga
\end{align*}
where \(k \geq 1\), and \(\ga\) is a strict partition with all parts congruent to \(0 \ppmod{4}\) and strictly less than \(4k\).
But these are the cases ruled out by \Cref{prop:difficult_even_cases}.
\end{proof}

This completes the proof of \Cref{onlyif}.

\section{4-stepped-and-semicongruent implies proportional}
\label{sec:fsas_implies_proportional-if}

We now turn towards proving the ``if'' direction of our main theorem, stated below.
We remind the reader that the construction of the partition \(\lx\al\) was given in \Cref{subsec:label_descs}.

\begin{restatable}[(``If'' direction of \Cref{main})]{thm}
{proportionality}
\label{thm:proportionality}
Let \(\al\) be a \fsas strict partition, and let \(\la \in \{\lx{\al}, \lx{\al}'\}\).
Then \(\bspn{\al}\) is proportional to \(\bspe{\la}\).
\end{restatable}

Our strategy is to first show the result for spin characters which are \emph{homogeneous}, meaning that their composition factors in characteristic $2$ are all isomorphic.
We then use (linear combinations of) induction and restriction functors to propagate the property of being proportional.
For this second step, in this section we only state the actions of our functions on relevant characters and show how they allow us to deduce \Cref{thm:proportionality}, deferring the proofs of these actions to the following sections.

\subsection{Homogeneous case}

Our proof starts with homogeneous characters.
We recall some notation from \Cref{sec:background}: 
\(\twoc{a}\) is the \(2\)-core partition \((a,a-1, \ldots, 1)\); \(\btwoc{a}\) is the \fbc whose double is \(\twoc{a}\); and given two partitions $\si$ and $\tau$, we write $\corandquot{\twoc a}\si\tau$ for the partition with $2$-core $\twoc a$ and $2$-quotient $(\si,\tau)$.

\begin{propn}
\label{prop:homogeneous_proportional}
Let \(\al = \btwoc{a} \sqcup 2\twoc{r}\) for some integers \(a, r \geq 0\) with \(a \geq r-1\).
Then:
\begin{enumerate}[(i)]
    \item \(\bspn{\al}\) is homogeneous;
    \item \(\bspe{\lx\al}\) is irreducible;
    \item \(\lx\al = \dbl{\al}\reg\);
    \item \(\bspn{\al} \propto \bspe{\lx\al}\).
\end{enumerate}
\end{propn}

\begin{proof}
\begin{enumerate}[(i),beginthm]
    \item \label{item:homogeneous_spin}
The homogeneous spin characters labelled by \emph{separated} partitions were classified in \cite[Theorem 5.11]{fayers20spin2alt}; in fact since the even parts of \(\al\) are twice a \(2\)-core, we need only use \cite[Proposition 5.8]{fayers20spin2alt} to deduce that \(\bspn{\al}\) is homogeneous.

    \item \label{item:irred_Specht}
The irreducible Specht modules were classified in \cite{jamesmathas1999}; in fact since \(\lx\al = \corandquot{\twoc{a}}{\emptyset}{\twoc{r}} = \twoc{a} + \twoc{r}\) is \(2\)-regular, we need only use \cite[Theorem 2.13]{james_carter} to deduce that \(\bspe{\lx\al}\) is irreducible.

    \item \label{item:dblreg_evaluation}
It is easily seen that \(\dbl{\al} = \twoc{a} \sqcup \twoc{r} \sqcup \twoc{r}\).
Observe that this partition is \(2\)-restricted (i.e.\ consecutive parts differ by at most $1$), and hence its regularisation is equal to its conjugate.
Note also for any partitions \(\la,\mu\) we have \((\la \sqcup \mu)' = \la' + \mu'\).
Thus
\[
\dbl\al\reg
    = (\twoc{a} \sqcup \twoc{r} \sqcup \twoc{r})'
    = \twoc{a}' + \twoc{r}' + \twoc{r}'
    = \twoc{a} + 2\twoc{r}
    = \lx\al.
\]

    \item 
By \Cref{spinregn} and part \Cref{item:dblreg_evaluation} of the present \lcnamecref{prop:homogeneous_proportional}, \(\bspn{\al}\) has least dominant composition factor given by \(\jms{\dbl{\al}\reg} = \jms{\lx\al}\).
Then since \(\bspn{\al}\) is homogeneous by \Cref{item:homogeneous_spin}, all its composition factors are isomorphic to \(\jms{\lx\al}\).
Finally by \Cref{item:irred_Specht} we have \(\jms{\lx\al} = \bspe{\lx\al}\), giving the result. \qedhere
\end{enumerate}
\end{proof}

\begin{rmk}
In fact, the spin characters $\spn\al$ appearing in the above \lcnamecref{prop:homogeneous_proportional} are the only homogeneous spin characters proportional to a linear character in characteristic \(2\), with the exception of \(\al = (4)\) (which satisfies \(\bspn{(4)} \propto \bspe{2,2}=\jms{3,1}\)).
Indeed, suppose \(\bspn{\al}\) is homogeneous and proportional to \(\bspe{\la}\).
Then \(\bspe{\la}\) is also homogeneous, and hence \(\bspe{\la} = \jms{\la\reg}\) is irreducible (since we know \(\jms{\la\reg}\) occurs with multiplicity \(1\)).
James and Mathas's classification of irreducible linear characters \cite{jamesmathas1999} says that \(\bspe{\la}\) is irreducible if and only if \(\la\) or \(\la'\) is \emph{\(2\)-Carter} (a property which implies \(2\)-regular) or \(\la = (2,2)\).

On the other hand, our main theorem gives \(\la \in \set{\lx\al, \lx\al'}\) and hence both components of the \(2\)-quotient of \(\la\) are \(2\)-cores.
If both of these \(2\)-cores are nonempty, then in an abacus display for \(\la\) or \(\la'\) there are beads in adjacent positions, and so \(\la\) and \(\la'\) are both \(2\)-singular (that is, having repeated parts) and hence not \(2\)-Carter.
Furthermore if the nonempty component of the \(2\)-quotient has length greater than the length of the \(2\)-core by \(2\) or more, then again an abacus display for \(\la\) or \(\la'\) has beads in adjacent positions.
Thus the only possibility (up to conjugacy) is
\(\la = \corandquot{\twoc{a}}{\emptyset}{\twoc{r}}\) for some \(a \geq r-1\) (which is precisely the situation of the lemma) 
or \(\la = (2,2)\) (in which case \(\al = (4)\)).
\end{rmk}

\subsection{Application of runner-swapping and quotient-redistributing functions}
\label{subsec:if_application_of_functors}

In order to prove \cref{thm:proportionality} from \cref{prop:homogeneous_proportional}, we apply certain combinations of the induction and restriction functors $e_i,f_i$.
They are both functions on the Grothendieck group, or equivalently on generalised characters.

\begin{restatable}{defn}{runnerswapping}
\label{runnerswapping}
For \(c\in\bbz\) and \(\eps \in \set{0,1}\), define the \emph{\rsf} to be 
\[
\runnerswap{\eps}{c} =
    \sum_{a\geq \max\{0,-c\}} (-1)^{a+c} f_\eps^{(a+c)} e_{\eps}^{(a)}.
\]
\end{restatable}

\begin{restatable}{defn}{quotientredistributing}
\label{quotientredistributing}
For \(d\in\bbz\) and \(\eps \in \{0,1\}\), define the \emph{\qrf} to be
\[
\quotred{\ep}{d} =
    \sum_{a\geq \max\{0,-d\}} (-1)^{a+d} f^{(a+d)}_{\eps} f^{(a+d)}_{\beps} e^{(a)}_{\beps} e^{(a)}_{\eps}.
\]
\end{restatable}

The actions of these functions on certain partitions of interest are given by the next four propositions, whose proofs are deferred to subsequent sections.

We begin with the action of the \rsf. The following proposition is proved in a much more general form in \Cref{subsec:runner-swapping_functor_on_chars}.

\begin{restatable}[(Action of \rsf on certain characters)]{propn}{runnerswapapplicationchars}
\label{runner-swap_application_on_chars}
Let \(r,s \geq 0\) and let \(a \geq 1\). Let \(\eps\) be the residue of \(a+1\) modulo $2$.
Then
\[
\runnerswap{\eps}{-a} \spe{ \corandquot{\twoc{a}}{\twoc{r}}{\twoc{s}} }
    = \pm \spe{ \corandquot{\twoc{a-1}}{\twoc{r}}{\twoc{s}} }.
\]
\end{restatable}

\begin{eg}\label{rsfex}
Take $r=1$, $s=2$, \(a=2\) and hence \(\eps=1\).
Then $\corandquot{\twoc2}{\twoc1}{\twoc2}$ is the partition $\la=(6,3,1^2)$.
The maximum $d$ for which $e_1^{(d)}\spe\la\neq0$ is $3$, so that
\begin{align*}
\runnerswap1{-2}\spe\la&=e_1^{(2)}\spe\la-f_1e_1^{(3)}\spe\la
\\
&=\spe{6,2,1}+\spe{5,3,1}+\spe{5,2,1^2}-f_1\spe{5,2,1}
\\
&=-\spe{5,2^2}
\end{align*}
and indeed $(5,2^2)=\corandquot{\twoc1}{\twoc1}{\twoc2}$.
This is illustrated with Young diagrams and abacus displays below; the removed and added nodes are coloured.
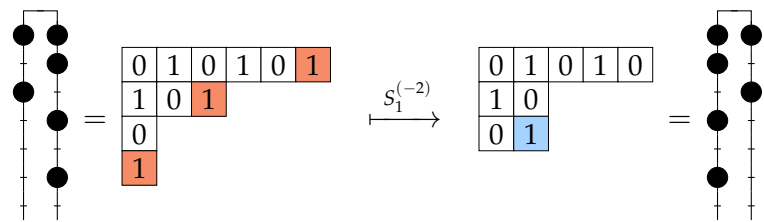
\[
\abacus(lr,bb,nb,bn,nb,nn,nb,nn)
\;=\;
\gyoung(01010!\YpC1!\wht,10!\YpC1!\wht,0,!\YpC1!\wht)
\quad
\xmapsto{\phantom{s}S_1^{(-2)}}
\quad
\gyoung(01010,10,0!\YpA1!\wht,:)
\;=\;
\abacus(lr,bb,bn,nb,bn,nn,bn,nn)
\]
\end{eg}

We have a corresponding statement for spin characters, proved in \Cref{subsec:runner-swapping_functor_on_spin_chars}.

\begin{restatable}[(Action of \rsf on certain spin characters)]{propn}{runnerswapapplicationspinchars}
\label{runner-swap_application_on_spin_chars}
Let \(r,s \geq 0\) and let \(a \geq 1\). Let \(\eps\) be the residue of \(a+1\) modulo $2$.
Then
\[
\runnerswap{\eps}{-a} \spn{ \btwoc{a} \sqcup 2(\twoc{r} + \twoc{s})}
    = \pm \spn{ \btwoc{a-1} \sqcup 2(\twoc{r} + \twoc{s})}.
\]
\end{restatable}

\begin{eg}
As in \cref{rsfex}, we take $r=1$, $s=2$, \(a=2\) and hence \(\eps=1\).
Now $\btwoc2=(3)$, so that $\btwoc{a} \sqcup 2(\twoc{r} + \twoc{s})$ is the partition $\al=(6,3,2)$.
Now we can calculate
\begin{align*}
\runnerswap1{-2}\spn\al&=e_1^{(2)}\spe\al-f_1e_1^{(3)}\spe\al
\\
&=2\spn{5,3,1}+\spn{6,2,1}-\sqrt2f_1\spn{5,2,1}
\\
&=2\spn{5,3,1}+\spn{6,2,1}-(2\spn{6,2,1}+2\spn{5,3,1})
\\
&=-\spn{6,2,1}
\end{align*}
and indeed $(6,2,1)=\btwoc1 \sqcup 2(\twoc1+\twoc2)$.
This is illustrated with Young diagrams below; the removed nodes are coloured.
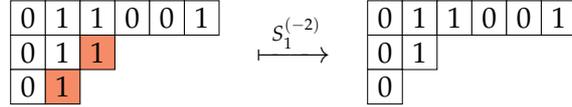
\[
\gyoung(011001,01!\YpC1!\wht,0!\YpC1!\wht)
\quad
\xmapsto{\phantom{s}S_1^{(-2)}}
\quad
\gyoung(011001,01,0)
\]
\end{eg}

For the action of the \qrf, we restrict attention to (spin) characters in RoCK blocks.
Recall that the block with $2$-core $\twoc a = (a,a-1,\ldots,1)$ and weight $w$ is \emph{RoCK} if $a\gs w-1$.
The following is proved in \Cref{subsec:quotient-redistributing_functor_on_chars}.

\begin{restatable}[(Action of \qrf on certain RoCK characters)]{propn}{quotredapplicationchars}
\label{application_of_quotient-redistributing_functor_on_chars}
Let \(s > r \geq 0\) and let \(a \geq \frac12 (r(r+1)+s(s+1))-1\). 
Let \(\eps\) be the residue of \(a+1\) modulo $2$.
Then
\begin{align*}
\quotred{\eps}{-(s-r-1)} \spe{ \corandquot{\twoc{a}}{\twoc{r}}{\twoc{s}} }
    &= \pm\spe{ \corandquot{\twoc{a}}{\twoc{r+1}}{\twoc{s-1}} }.
\end{align*}
\end{restatable}

\begin{eg}
\label{eg:qr}
Take \(r=0\), \(s=2\) and \(a=2\), and hence \(\eps = 1\).
Then \(\corandquot{\twoc{2}}{\twoc{0}}{\twoc{2}}\) is the partition \((6,3)\).
We calculate
\begin{align*}
\quotred{\eps}{-1} \spe{6,3}
    &= e_0^{(1)} e_1^{(1)} \spe{6,3} - f_1^{(1)}f_0^{(1)} e_0^{(2)} e_1^{(2)} \spe{6,3} \\
    &= \spe{6,1} + \spe{4,3} - f_1^{(1)} f_0^{(1)} \spe{4,1} \\
    &= - \spe{4,1^3}
\end{align*}
and indeed \((4,1^3) = \corandquot{\twoc{2}}{\twoc{1}}{\twoc{1}}\).
This is illustrated with Young diagrams and abacus displays below; the removed and added nodes are coloured.
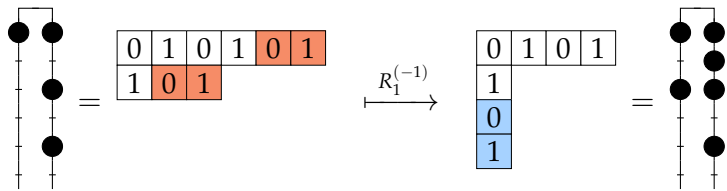
\[
\abacus(lr,bb,nn,nb,nn,nb,nn)
\;=\;
\gyoung(0101!\YpC01!\wht,1!\YpC01!\wht,:,:)
\quad
\xmapsto{\phantom{s}R_1^{(-1)}}
\quad
\gyoung(0101,1,!\YpA0,1)
\;=\;
\abacus(lr,bb,nb,bb,nn,nb,nn)
\]
\end{eg}

The corresponding statement for spin characters below is proved in \Cref{subsec:quotient-redistributing_functor_on_spin_chars}.

\begin{restatable}[(Action of \qrf on certain RoCK spin characters)]{propn}{quotredapplicationspinchars}
\label{application_of_quotient-redistributing_functor_on_spin_chars}
Let \(s > r \geq 0\) and let \(a \geq \frac12 (r(r+1)+s(s+1))-1\). 
Let \(\eps\) be the residue of \(a+1\) modulo $2$.
Then
\begin{align*}
\quotred{\eps}{-(s-r-1)} \spn{ \btwoc{a} \sqcup 2(\twoc{r} + \twoc{s})}
    &= \pm \sqrt{2} \spn{ \btwoc{a} \sqcup 2(\twoc{r+1} + \twoc{s-1})}.
\end{align*}
\end{restatable}

\begin{eg}
\label{eg:qr_on_spin}
As in \Cref{eg:qr}, we take \(r=0\), \(s=2\), \(a=2\) and hence \(\eps = 1\).
Now \(\btwoc{2} \sqcup 2(\twoc{0}+\twoc{2})\) is the strict partition \((4,3,2)\).
We calculate
\begin{align*}
\quotred{1}{-1} \spn{4,3,2}
    &= e_0^{(1)} e_1^{(1)} \spn{4,3,2} - f_1^{(1)}f_0^{(1)} e_0^{(2)} e_1^{(2)} \spe{4,3,2} \\
    &= \sqrt{2} \spn{4,3}  - \sqrt{2} f_1^{(1)} f_0^{(1)} \spn{3,2} \\
    &= \sqrt{2} \spn{4,3} - 2\sqrt{2} \spn{4,3} \\
    &= - \sqrt{2} \spn{4,3}
\end{align*}
and indeed \((4,3) = \btwoc{2} \sqcup 2(\twoc{1} + \twoc{1})\).
This is illustrated with Young diagrams below; the removed nodes are coloured.
\[
\gyoung(0110,011,!\YpC01!\wht)
\quad
\xmapsto{\phantom{s}R_1^{(-1)}}
\quad
\gyoung(0110,011,:)
\]
\end{eg}

These allow us to complete the goal of this section, restated below.

\proportionality*

\begin{proof} 
By assumption there exist non-negative integers \(a, r, s\) with $s\gs r$ such that \(\al = \btwoc{a} \sqcup 2(\twoc{r}+\twoc{s})\) and \(\la = \corandquot{\twoc{a}}{\twoc{r}}{\twoc{s}}\).
Choose \(b\in\bbn\) such that $b\gs a$ and $b\gs|\twoc{r+s}|-1$. 
Applying \cref{prop:homogeneous_proportional} with $\al=\btwoc b\sqcup2\twoc{r+s}$, we obtain
\[
    \bspn{ \btwoc{b} \sqcup 2\twoc{r+s} } \propto \bspe{ \corandquot{\twoc{b}}{\emptyset}{\twoc{r+s}} }.
\]

Applying \qrf{}s to each side \(r\) times and using \Cref{application_of_quotient-redistributing_functor_on_chars,application_of_quotient-redistributing_functor_on_spin_chars} iteratively -- the choice of \(b\) having been made to satisfy the hypotheses of these \lcnamecref{application_of_quotient-redistributing_functor_on_chars}s -- we obtain
\[
    \bspn{ \btwoc{b} \sqcup 2(\twoc{r}+\twoc{s})} \propto \bspe{ \corandquot{\twoc{b}}{\twoc{r}}{\twoc{s}} }.
\]
Then applying \rsf{}s to each side \(b-a\) times and using \Cref{runner-swap_application_on_chars,runner-swap_application_on_spin_chars} iteratively yields
\[
    \bspn{ \btwoc{a} \sqcup 2(\twoc{r}+\twoc{s})} \propto \bspe{ \corandquot{\twoc{a}}{\twoc{r}}{\twoc{s}} }
\]
as required.
\end{proof}

\begin{rmk}
For characters in RoCK blocks, a possible alternative to our application of \qrf{}s is to use the first author's explicit formul\ae{} for the decomposition numbers of RoCK blocks in terms of symmetric functions \cite[Theorem~5.3]{fayers20spin2alt}.
Using these formul\ae{}, it can be shown that \cref{thm:proportionality} for RoCK blocks is equivalent to the following statement about symmetric functions: if $\si$ and $\tau$ are $2$-cores, then \[P_{\si+\tau}=s_\si s_\tau,\] where $P_\al$ denotes the Schur P-function labelled by a strict partition $\al$, and $s_\la$ is the Schur function corresponding to a partition $\la$ (see Macdonald's book \cite{macd} for basic definitions regarding symmetric functions).
We have not been able to find this result in the literature; our main theorem yields a proof of this identity.
\end{rmk}

\section{Runner-swapping functions}
\label{sec:runner-swapping}

The purpose of this section is to describe the action of the \rsf{}s $\runnerswap\ep c$ on certain characters and spin characters. In fact we will establish the action in much more generality than that required in \Cref{subsec:if_application_of_functors}: we identify, for all characters and a large class of spin characters, a choice of parameters for which the image is (up to sign) a specified single (spin) character.

We recall the definition of the function from \Cref{sec:fsas_implies_proportional-if}.

\runnerswapping*

\subsection{Action on characters}
\label{subsec:runner-swapping_functor_on_chars}

First we consider the effect of $\runnerswap\eps c$ on characters $\spe\la$. We prove the following result (recall for $r\gs0$ we write \(\twoc{r}\) for the \(2\)-core partition \((r, r-1,\dots,1)\), and given a bipartition \(\bipla\), we write \(\corandbipquot{\twoc{r}}{\bipla}\) for the partition with $2$-core $\twoc r$ and $2$-quotient $\bipla$).

\begin{thm}
\label{runner-swap_application_on_charsgeneral}
Let \(\la = \corandbipquot{\twoc{a}}{\bipla}\) be any partition.

\begin{enumerate}[(i)]
    \item
Let \(\eps\) be the residue of \(a\) modulo \(2\). Then
\begin{align*}
    \runnerswap{\eps}{a+1} \spe{ \corandbipquot{\twoc{a}}{\bipla} }
     = \pm \spe{ \corandbipquot{\twoc{a+1}}{\bipla} }. \\
\intertext{ \item
\csname cref@label\endcsname{item:runner-swap_application_on_charsgeneral_case_for_application} 
Suppose \(a \geq 1\), and let $\ep$ be the residue of $a+1$ modulo $2$. Then }
\runnerswap{\eps}{-a} \spe{ \corandbipquot{\twoc{a}}{\bipla} }
    = \pm \spe{ \corandbipquot{\twoc{a-1}}{\bipla} }.
\end{align*}
\end{enumerate}
\end{thm}

The special case of \cref{runner-swap_application_on_charsgeneral}\Cref{item:runner-swap_application_on_charsgeneral_case_for_application} in which \(\bipla = (\twoc{r}, \twoc{s})\) gives \cref{runner-swap_application_on_chars} required in the proof of our main theorem.

To prepare for the proof, we introduce some notation. Recall that $\cald_d$ is the $d$th diagonal in $\bbn^2$. For a partition $\la$ and $d\in\bbz$, we write
\[
\addsondiag{d}{\la}=
\begin{cases}
\phantom{-}1&\text{if $\la$ has an addable node in $\cald_d$;}
\\
-1&\text{if $\la$ has a removable node in $\cald_d$;}
\\
\phantom{-}0&\text{otherwise}.
\end{cases}
\]
If $\mu$ is another partition, we write $\diffondiag d\mu\la$ for the number of nodes of $\mu$ in $\cald_d$ minus the number of nodes of $\la$ in $\cald_d$.

\begin{lemma}\label{ldd}
Let $\la$ be a partition, $c\in\bbz$, and $\ep\in\{0,1\}$.
If \(\mu\) is a partition such that $\spe\mu$ appears with non-zero coefficient in $\runnerswap{\eps}{c}\spe\la$, then $\diffondiag d\mu\la \leq \addsondiag{d}{\la}$ for all \(d \in \bbz\) with \(d \equiv \eps \ppmod{2}\).
\end{lemma}

\begin{proof}
It is clear that if $\spe\mu$ appears with non-zero coefficient in $\runnerswap{\eps}{c}\spe\la$ then $\mu$ is obtained from $\la$ by removing and then adding $\ep$-nodes. 
In particular, $\la$ and $\mu$ agree on $\cald_d$ for every $d\nequiv\ep\ppmod2$.
This implies that $-1 \ls \diffondiag{d}{\mu}{\la} \ls 1$, and $\diffondiag{d}{\mu}{\la}=0$ if $\addsondiag{d}{\la}=0$.
This automatically gives the desired result for all $d$ for which $\addsondiag{d}{\la} \geq 0$.
So we take $d$ with $\addsondiag{d}{\la}=-1$, and assume for a contradiction that $\diffondiag{d}{\mu}{\la}=0$; that is, the removable node of $\la$ in $\cald_d$ is also present in $\mu$.

Let $\La$ be the set of partitions $\nu$ such that $\la$ and $\mu$ can each be obtained from $\nu$ by adding $\eps$-nodes.
Then $\La$ is partitioned into pairs $\{\nu^0,\nu^1\}$, where the removable node of $\la$ in $\cald_d$ is also a removable node of $\nu^0$, and $\nu^1$ is obtained from $\nu^0$ by removing this node.
Each such pair satisfies $|\nu^1|=|\nu^0|-1$, so the contributions to $\ip{\runnerswap{\eps}{c}\spe\la}{\spe\mu}$ from $\nu^0$ and $\nu^1$ cancel out. Summing over all such pairs, we find that $\ip{\runnerswap{\eps}{c}\spe\la}{\spe\mu}=0$, contrary to assumption.
\end{proof}

Given a partition $\la$, let $\swp\la\ep$ be the partition obtained from $\la$ by simultaneously adding all the addable $\eps$-nodes and removing all the removable $\eps$-nodes.
We can equivalently characterise $\swp\la\ep$ as the unique partition such that:
\begin{itemize}
\item
$\swp\la\ep$ is obtained from $\la$ by adding and/or removing $\eps$-nodes, and
\item
$\diffondiag{d}{\swp\la\ep}{\la} = \addsondiag{d}{\la}$ for every $d\equiv\eps\ppmod2$.
\end{itemize}
Let $\netaddables\ep\la$ be the number of addable $\ep$-nodes of $\la$ minus the number of removable $\ep$-nodes (that is, \(\netaddables{\ep}{\la} = \sum_{d \equiv \eps \ppmod{2}} \addsondiag{d}{\la}= |\swp\la\ep|-|\la|\)).

\begin{propn}
\label{characters:action_of_runner-swapping_functor}
Let \(\la\) be a partition and $\ep\in\{0,1\}$.
Then \(\runnerswap{\eps}{\netaddables{\ep}{\la}}\spe{\la} = \pm \spe{\swp\la\ep}\).
\end{propn}

\newcommand{\sumoverd}{\sum_{\substack{d \equiv \eps \\ \scriptscriptstyle{(\operatorname{mod}\,2)}}}}

\begin{pf}
Suppose $\spe\mu$ appears with non-zero coefficient in $\runnerswap{\eps}{\netaddables{\ep}{\la}}\spe\la$.
In particular, \(\mu\) can be obtained from \(\la\) by removing then adding \(\eps\)-nodes, with a net gain of \(\netaddables{\eps}{\la}\)-many nodes, and so \(|\mu| - |\la| = \netaddables{\eps}{\la}\).
By definition of \(\netaddables{\eps}{\la}\), we therefore have 
\begin{align*}
    |\mu|-|\la| &= \sumoverd \addsondiag{d}{\la}.
\intertext{
Meanwhile, by considering the net number of nodes added to each diagonal individually, we have
}
    |\mu| - |\la| &= \sumoverd \diffondiag{d}{\mu}{\la}.
\end{align*}
Equating these two expressions for \(|\mu|-|\la|\) gives
\begin{equation}
        \sumoverd \diffondiag{d}{\mu}{\la}
        \ \ = \
        \sumoverd \addsondiag{d}{\la}. \tag{\(\dagger\)} \label{eq:summed_equality}
\end{equation}
But by \Cref{ldd}, we have ${\diffondiag{d}{\mu}{\la}} \leq \addsondiag{d}{\la}$ for every $d\equiv\ep\ppmod2$.
Thus for \eqref{eq:summed_equality} to hold, the equality \({\diffondiag{d}{\mu}{\la}} = \addsondiag{d}{\la}\) must hold for every \(d \equiv \eps \ppmod{2}\).
Then \(\mu\) fulfils the characterisation of \(\swp\la\ep\) given directly above the statement of this proposition, so \(\mu = \swp\la\ep\).

It remains to determine the coefficient of $\spe{\swp\la\ep}$ in $\runnerswap{\eps}{\netaddables{\ep}{\la}}\spe\la$.
It is easy to see that the only partition $\nu$ such that $\la$ and $\swp\la\ep$ are both obtained from $\nu$ by adding $\eps$-nodes is $\la \cap \swp\la\ep$, the partition obtained from \(\la\) by removing all removable \(\eps\)-nodes.
Thus
\[
\ip{\runnerswap{\eps}{\netaddables{\ep}{\la}}\spe\la}{\spe{\swp\la\ep}}
    = (-1)^{|\swp\la\ep|-|\la\cap\swp\la\ep|}.\qedhere
\]
\end{pf}

In view of \Cref{characters:action_of_runner-swapping_functor}, we would like to understand the partition $\swp\la\ep$.
Another equivalent characterisation of \(\swp\la\ep\) is via the following standard abacus combinatorics, motivating the name ``\rsf'' for \(\runnerswap{\eps}{c}\).

\begin{propn}
\label{prop:swp_is_runner-swap}
Let \(\la\) be a partition and let \(\eps \in \set{0,1}\).
An abacus display for \(\swp{\la}{\eps}\) can be obtained by choosing an \(r\)-bead abacus display for \(\la\) where \(r \equiv \beps \ppmod{2}\) and swapping the runners. 
\end{propn}

\begin{proof}
It is well-known (and easy to see) how to identify addable and removable nodes in an abacus display: with the number of beads being congruent to $\bep$ modulo $2$, addable $\ep$-nodes correspond to beads on runner $0$ with no bead immediately to the right, and adding these nodes corresponds to moving these beads to the right; dually, removable $\ep$-nodes correspond to beads on runner $1$ with no bead immediately to the left, and removing these nodes corresponds to moving these beads to the left.
So adding all the addable $\ep$-nodes and removing all the removable $\ep$-nodes can be accomplished on the abacus by moving every bead on runner $0$ to the right and every bead on runner $1$ to the left, which is the same as swapping the two runners.
\end{proof}

The abacus interpretation in \cref{prop:swp_is_runner-swap} enables us to describe \(\swp\la\ep\) in terms of \(2\)-cores and $2$-quotients as follows.

\begin{lemma}
\label{lemma:computing_swp_partition_on_cores_and_quotients}
\begin{enumerate}[(i), beginthm]
    \item\label{item:swp_on_cores}
Let \(a \geq 0\) and let \(\eps \in \set{0,1}\).
Then
\[
    \netaddables{\eps}{\twoc{a}} = \begin{cases}
        a+1 & \text{ if \(\eps \equiv a \ppmod{2}\);} \\
        -a & \text{ if \(\eps \not\equiv a \ppmod{2}\),} 
    \end{cases}
\qquad
    \swp{\twoc{a}}{\eps} = \begin{cases}
        \twoc{a+1} & \text{ if \(\eps \equiv a \ppmod{2}\);} \\
        \twoc{a-1} & \text{ if \(\eps \not\equiv a \ppmod{2}\),} 
    \end{cases}
\]
unless \(a=0\) and \(\eps=1\), in which case \(\netaddables{1}{\twoc{0}} = 0\) still holds but \(\swp{(\twoc{0})}{1} = \twoc{0}\).
    \item\label{item:swp_is_swp_on_core}
Let \(\nu\) be a \(2\)-core partition, let \(\la^{(0)}, \la^{(1)}\) be partitions and let \(\eps \in \set{0,1}\).
Then
\[
    \swp{\corandquot{\nu}{\la^{(0)}}{\la^{(1)}}}{\eps} = \corandquot{\swp{\nu}{\eps}}{\la^{(0)}}{\la^{(1)}}
\]
unless \(\nu = \emptyset\) and \(\eps=1\), in which case \(\swp{\corandquot{\emptyset}{\la^{(0)}}{\la^{(1)}}}{1} = \corandquot{\emptyset}{\la^{(1)}}{\la^{(0)}}\).
In any case, we have \(\netaddables{\eps}{\corandquot{\nu}{\la^{(0)}}{\la^{(1)}}} = \netaddables{\eps}{\nu}\).
\end{enumerate}
\end{lemma}

\begin{pf}
Part (i) is routine (with or without \Cref{prop:swp_is_runner-swap}). For part (ii) we apply \Cref{prop:swp_is_runner-swap}: write \(\la = \corandquot{\nu}{\la^{(0)}}{\la^{(1)}}\) and choose an \(r\)-bead abacus display for \(\la\) where \(r \equiv \beps \ppmod{2}\), so that an \(r\)-bead abacus display for \(\swp{\la}{\ep}\) is obtained by swapping the runners. Then clearly the \(2\)-core of \(\swp{\la}{\ep}\) has \abd given by swapping the runners in an \(r\)-bead abacus display for \(\nu\), which by \Cref{prop:swp_is_runner-swap} again is precisely \(\swp{\nu}{\eps}\).
Meanwhile by our unusual convention for the \(2\)-quotient (\Cref{subsec:abacus_and_quotient}), swapping the runners has no effect on the \(2\)-quotient provided \(\nu \neq \emptyset\) or \(\eps =0\), for then one runner has strictly more beads than the other and we take \(\la^{(1)}\) to be given by the runner with more beads.
In the case \(\nu = \emptyset\) and \(\eps=1\), however, the runners have equal numbers of beads so \(\la^{(1)}\) is given by runner \(1\) before and after swapping, so the components of the \(2\)-quotient are swapped.
\end{pf}

\Cref{runner-swap_application_on_charsgeneral} now follows by combining \Cref{characters:action_of_runner-swapping_functor,lemma:computing_swp_partition_on_cores_and_quotients}.

\begin{rmk}
\begin{enumerate}[(i), beginthm]
    \item 
Our \rsf{}s are not universal in the sense that for a given pair $\la,\ep$, we need to choose a particular value $c$ to obtain $\runnerswap{\eps}c\spe{\la} = \pm \spe{\swp\la\ep}$.
In fact, one can define a universal runner-swapping function $S_\ep$ such that $S_\ep \spe{\la} = \pm \spe{\swp\la\ep}$ for all $\la$, by
\[
S_\ep=\exp(-f_\ep)\exp(e_\ep)\exp(-f_\ep)=\sum_{a,b,c\gs0}(-1)^{a+c}f^{(a)}_\ep e^{(b)}_\ep f^{(c)}_\ep
\]
(the expression $\exp(-f_\ep)\exp(e_\ep)\exp(-f_\ep)$ arises from the general theory of $\mathfrak{sl}_2$-representations, and is taken from \cite[\S4.2]{cr}). But we prefer our family of \rsf{}s $\runnerswap\ep c$, which afford simpler calculations.

    \item 
In the case where \(\la\) is RoCK, so that the runner with more beads has a bead on every row that the other runner does, all but one term in \(\runnerswap{\eps}{\netaddables{\eps}{\la}} \spe{\la}\) vanishes and our function becomes a single (divided power of a) induction or restriction functor.
In this case, Scopes showed that the functor in fact defines a Morita equivalence between the relevant blocks \cite{scopes91}.
Our application in \Cref{subsec:if_application_of_functors}, however, is to use \rsf{}s to reach non-RoCK partitions.
\end{enumerate}
\end{rmk}

\begin{rmk}
\label{remark:runner-swapping_functor_odd_p}
The proofs of \Cref{ldd,characters:action_of_runner-swapping_functor,prop:swp_is_runner-swap}
apply just as well with residues taken modulo \(p\) and abacus displays with \(p\) runners, for \(p\) an odd prime.
That is, we have constructed a function that, with suitable choice of parameters, adds all addable nodes and removes all removable nodes of a given residue, or equivalently swaps a pair of adjacent runners in an abacus display. With the appropriate convention for the $p$-quotient (called the ``ordered $p$-quotient'' in \cite{fayers2016}, where the reader can find more background on the appropriate combinatorics when $p$ is odd), this function preserves the $p$-quotient -- except (as in the exceptional case in \cref{lemma:computing_swp_partition_on_cores_and_quotients}(ii)) where the two runners in question have the same number of beads, in which case the corresponding components of the $p$-quotient are swapped.

For example, let \(p=5\) and consider the function \(S_2^{(1)} = \sum_{a \geq 0} (-1)^{(a+1)} f_2^{(a+1)} e_2^{(a)}\) applied to the character labelled by the partition \((9,8,5^2,2,1^3)\).
We have
\begin{align*}
S_2^{(1)} \spe{9,8,5,1^5}
    &= f_2^{(1)} \spe{9,8,5,1^5} - f_2^{(2)}e_2^{(1)} \spe{9,8,5,1^5} \\
    &= \spe{9,8,5,1^6} + \spe{9^2,5,1^5} - f_2^{(2)} \spe{9,8,4,1^5} \\
    &= - \spe{9^2,4,1^6}.
\end{align*}
The partitions $(9,8,5,1^5)$ and $(9^2,4,1^6)$ both have $5$-quotient $(\vn,\vn,(1^2),(2),(1))$ (with the appropriate convention), and their $5$-cores, namely $(2)$ and $(3)$, differ in the addition of a $2$-node.

This is illustrated with Young diagrams and a choice of abacus display below. The removed and added nodes, and the swapped runners, are coloured.
(Here we have chosen a number of beads congruent to \(0\) modulo \(5\), so that the parts with final nodes of residue \(2\) are given on runner \(2\), and it is runners \(1\) and \(2\) that are swapped.)
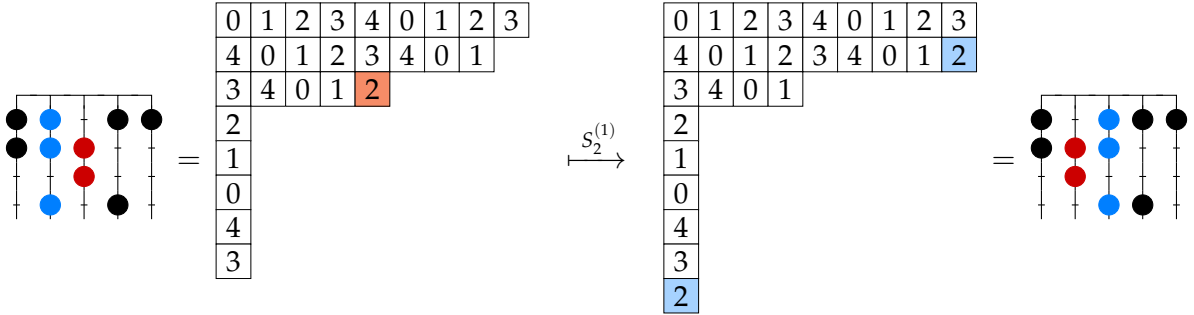
\[
\abacus(lmmmr,bAnbb,bACnn,nnCnn,nAnbn) 
\;=\;
\gyoung(012340123,40123401,3401!\YpC2!\wht,2,1,0,4,3,:)
\quad
\xmapsto{\phantom{s}S_2^{(1)}}
\quad
\gyoung(012340123,40123401!\YpA2!\wht,3401,2,1,0,4,3,!\YpA2!\wht)
\;=\;
\abacus(lmmmr,bnAbb,bCAnn,nCnnn,nnAbn) 
\]
\end{rmk}

\subsection{Action on spin characters}
\label{subsec:runner-swapping_functor_on_spin_chars}

We next prove a corresponding result for spin characters. This is not quite as general as \cref{runner-swap_application_on_charsgeneral} (it doesn't apply to all strict partitions), though still more general than is needed in \cref{subsec:if_application_of_functors}.

\begin{thm}
\label{runner-swap_application_on_spin_charsgeneral}
Let \(\al = \btwoc{a} \sqcup 2\eta\) be a strict partition whose odd parts form a \fbc.
\begin{enumerate}[(i)]
    \item
Let $\ep$ be the residue of $a$ modulo $2$. Then
\begin{align*}
    \runnerswap{\eps}{a+1} \spn{ \btwoc{a} \sqcup 2\eta }
     = \pm \spn{ \btwoc{a+1} \sqcup 2\eta }. \\
\intertext{ \item
\csname cref@label\endcsname{item:runner-swap_application_on_spin_charsgeneral_case_for_application} 
Suppose \(a \geq 1\), and let $\ep$ be the residue of $a+1$ modulo $2$. Then }
\runnerswap{\eps}{-a} \spn{ \btwoc{a} \sqcup 2\eta }
    = \pm \spn{ \btwoc{a-1} \sqcup 2\eta }.
\end{align*}
\end{enumerate}
\end{thm}

The special case of \Cref{runner-swap_application_on_spin_charsgeneral} in which $\eta = \twoc r+\twoc s$ gives \cref{runner-swap_application_on_spin_chars} (required in the proof of our main theorem).

Our proof is similar to that of \Cref{runner-swap_application_on_charsgeneral}, but (because we need to consider \spr{}s rather than residues) we consider the effect of the functor on pairs of adjacent columns
rather than individual diagonals.

For \(\al\in\scrd\) and an even non-negative integer $d$, let $\spinaddsondiag{d}{\al}$ denote the number of \spams in columns $d$ and $d+1$ minus the number of \sprms in columns $d$ and $d+1$. (When $d=0$, we ignore any reference to column $d$.) 
Note that any particular column can contain at most one \spam and at most one \sprm, and cannot contain both.
Column \(1\) must contain either a \spad or a \sprm; meanwhile, for any other column of residue \(\eps\), if there is a spin-addable or \sprm in that column, then the adjacent column of residue \(\eps\) contains either a \spad or \sprm (not necessarily of the same type).
Thus \(\spinaddsondiag{d}{\al}\in \{-2,0,2\}\) if \(d > 0\), and \(\spinaddsondiag{d}{\al} \in \{-1,1\}\) if \(d=0\).

If $\be$ is another strict partition, let $\bdiffondiag{d}{\be}{\al}$ be the number of nodes of $\be$ in columns $d$ and $d+1$ minus the number of nodes of $\al$ in these columns, i.e.\ \(\bdiffondiag{d}{\be}{\al} = \be'_d+\be'_{d+1}-\al'_d-\al'_{d+1}\) (ignoring $\al'_d$ and $\be'_d$ if $d=0$).

\begin{lemma}\label{lessdd}
Suppose $\al\in\scrd$, $c\in\bbz$, and $\ep\in\{0,1\}$. If \(\be\in\scrd\) and $\spn\be$ appears with non-zero coefficient in $\runnerswap{\eps}{c}\spn\al$, then $\bdiffondiag{d}{\be}{\al} \leq \spinaddsondiag{d}{\al}$ for all \(d \geq 0\) with \(d \equiv 2\eps \ppmod{4}\).
\end{lemma}

\begin{pf}
The fact that $\spn\be$ appears in $\runnerswap{\eps}{c}\spn\al$ means that $\be$ is obtained from $\al$ by removing and then adding \nds\ep.
This implies in particular that $-2 \ls \bdiffondiag{d}{\be}{\al} \ls 2$ for every $d$, and $-1 \ls \bdiffondiag{0}{\be}{\al} \ls 1$. 
Now it is immediate that $\bdiffondiag{d}{\be}{\al} \leq \spinaddsondiag{d}{\al}$ if $\spinaddsondiag{d}{\al} > 0$.
So we consider \(d\) such that $\spinaddsondiag{d}{\al} \leq 0$.
We assume for a contradiction that $\bdiffondiag{d}{\be}{\al} > \spinaddsondiag{d}{\al}$.

Let $\Ga$ be the set of strict partitions $\ga$ such that $\al$ and $\be$ can both be obtained from $\ga$ by adding \nds\ep. Say that two partitions $\ga^0,\ga^1\in\Ga$ are \emph{equivalent} if they agree outside columns $d,d+1$, i.e.\ $(\ga^0)'_k=(\ga^1)'_k$ for all $k\neq d,d+1$.

Now we consider the possible values of $\spinaddsondiag{d}{\al}$ and $\bdiffondiag{d}{\be}{\al}$ satisfying our assumptions. There are four cases.

\begin{description}

\item[$\spinaddsondiag{d}{\al} = -2$ and $\bdiffondiag{d}{\be}{\al}=0$:]
In this case $\al$ and $\be$ agree in columns $d,d+1$, both having \sprms in both columns.
Then each equivalence class in $\Ga$ has the form $\{\ga^0,\ga^1,\ga^2\}$, where $\bdiffondiag{d}{\ga^i}{\be} = -i$ for each $i$.
Take such a class, and let $a=|\al|-|\ga^0|$.
We claim that
\[
\sum_{i=0}^2
    (-1)^{a+i+c}
    \ip{e_{\eps}^{(a+i)}\spn\al}{\spn{\ga^i}}
    \ip{f_{\eps}^{(a+i+c)}\spn{\ga^i}}{\spn\be}
=0.
\]
Using \cref{spinbranch}, we find that there are positive real numbers \(s,t\) such that 
\begin{align*}
\ip{e_{\eps}^{(a)}\spn\al}{\spn{\ga^0}}&=s,
    &\qquad
    \ip{f_{\eps}^{(a+c)}\spn{\ga^0}}{\spn\be}&=t,
\\
\ip{e_{\eps}^{(a+1)}\spn\al}{\spn{\ga^1}}&=\qt s,
    &\qquad
    \ip{f_{\eps}^{(a+1+c)}\spn{\ga^1}}{\spn\be}&=\qt t,
\\
\ip{e_{\eps}^{(a+2)}\spn\al}{\spn{\ga^2}}&=s,
    &\qquad
    \ip{f_{\eps}^{(a+2+c)}\spn{\ga^2}}{\spn\be}&=t,
\end{align*}
which yields the claim. Now summing over all equivalence classes gives $\ip{\runnerswap{\eps}{c}\spn\al}{\spn\be}=0$, contrary to hypothesis.

\item[$\spinaddsondiag{d}{\al} = -2$ and $\bdiffondiag{d}{\be}{\al} = -1$:]
In this case \(\al\) has \sprms in both columns \(d\) and \(d+1\), one of which has been removed in \(\be\).
Then each equivalence class in $\Ga$ has the form $\{\ga^0,\ga^1\}$, where again $\bdiffondiag{d}{\ga^i}{\be} = -i$ for each $i$, and $\ga^1$ is obtained from $\ga^0$ by removing a node in column $d$ or $d+1$.
Now, letting $a=|\al|-|\ga^0|$, there are $s,t$ such that
\begin{align*}
\ip{e_{\eps}^{(a)}\spn\al}{\spn{\ga^0}} &= \qt s,
    &\qquad
    \ip{f_{\eps}^{(a+c)}\spn{\ga^0}}{\spn\be} &= t,
\\
\ip{e_{\eps}^{(a+1)}\spn\al}{\spn{\ga^1}} &=s, 
    &\qquad
    \ip{f_{\eps}^{(a+1+c)}\spn{\ga^1}}{\spn\be} &= \qt t,
\end{align*}
and again we get $\ip{\runnerswap{\eps}{c}\spn\al}{\spn\be}=0$ and a contradiction.

\item[$\spinaddsondiag{d}{\al} = -1$ and $\bdiffondiag{d}{\be}{\al} > -1$:]
In this case $d=\ep=0$ and \(\al\) has a \sprm in column \(1\).
This node cannot have been removed in \(\be\), or else then \(\bdiffondiag{d}{\be}{\al} = -1\); so \(\be\) also has a \sprm in column $1$, and $\bdiffondiag{d}{\be}{\al}=0$.
Then each equivalence classes in $\Ga$ has the form $\{\ga^0,\ga^1\}$, where $\ga^1$ is obtained from $\ga^0$ by removing a node in column $1$.
Now there are $s,t$ such that
\begin{align*}
\ip{e_{\eps}^{(a)}\spn\al}{\spn{\ga^0}} &= s,
    &\qquad
    \ip{f_{\eps}^{(a+c)}\spn{\ga^0}}{\spn\be}&=t,
\\
\ip{e_{\eps}^{(a+1)}\spn\al}{\spn{\ga^1}} &= s,
    &\qquad
    \ip{f_{\eps}^{(a+1+c)}\spn{\ga^1}}{\spn\be}&=t,
\end{align*}
and we reach a contradiction as in the previous case.

\item[$\spinaddsondiag{d}{\al}=0$ and $\bdiffondiag{d}{\be}{\al} > 0$:]
In this case, columns \(d\) and \(d+1\) of \(\al\) contain a \spam and a \sprm (if neither column contains a \spre or \spam, then necessarily \(\bdiffondiag{d}{\be}{\al} = 0\)), and the \spam has been added in \(\be\).
Then each equivalence class in $\Ga$ has the form $\{\ga^1,\ga^2\}$, where $\bdiffondiag{d}{\ga^i}{\be} = -i$ for each $i$, and $\ga^2$ is obtained from $\ga^1$ by removing a node in column $d$ or $d+1$. Now there are $s,t$ such that
\begin{align*}
\ip{e_{\eps}^{(a)}\spn\al}{\spn{\ga^1}} &= s,
    &\qquad
    \ip{f_{\eps}^{(a+c)}\spn{\ga^1}}{\spn\be}&=\sqrt{2}t,
\\
\ip{e_{\eps}^{(a+1)}\spn\al}{\spn{\ga^2}} &= \sqrt{2}s,
    &\qquad
    \ip{f_{\eps}^{(a+1+c)}\spn{\ga^2}}{\spn\be}&=t,
\end{align*}
and again the terms cancel.\qedhere
\end{description}
\end{pf}

Given \(\al\in\scrd\) and \(\eps \in \{0,1\}\), we construct a strict partition $\bswp\al\ep$ as follows. Examine each pair of adjacent columns with \spr $\ep$, indexed by \(d\) and \(d+1\) for some non-negative \(d\equiv2\ep\ppmod4\), in turn: if $\al$ has \spams but not \sprms in columns \(d\) and \(d+1\), 
then we add these \spams; if $\al$ has \sprms but not \spams in columns \(d\) and \(d+1\), 
then we remove the \sprms; otherwise we do nothing in columns \(d\) and \(d+1\). We can equivalently characterise $\bswp\al\ep$ as the unique strict partition  such that:
\begin{itemize}
\item
$\bswp\al\ep$ is obtained from $\al$ by adding and/or removing \nds\ep, and
\item
$\bdiffondiag{d}{\bswp\al\ep}{\al} = \spinaddsondiag{d}{\al}$ for every $d\equiv2\ep\ppmod4$.
\end{itemize}
Let \(\netspinaddables{\eps}{\al}\) be the number of \espams\ep of \(\al\) minus the number of \esprms\ep (that is, \(\netspinaddables{\eps}{\al} = \sum_{d \equiv 2\eps\ppmod{4}} \spinaddsondiag{d}{\al}\), and also \(\netspinaddables{\eps}{\al} = \abs{\bswp\al\eps} - \abs{\al}\)).

\begin{restatable}{propn}{runnerswaponspins}
\label{spin_characters:action_of_runner-swapping_functor}
Suppose $\al\in\scrd$ and $\ep\in\{0,1\}$.
Then $\runnerswap{\eps}{\netspinaddables{\ep}{\al}}\spn{\al} = \pm\spn{\bswp\al\ep}$.
\end{restatable}

\begin{pf}
Suppose $\spn\be$ appears with non-zero coefficient in $\runnerswap{\eps}{\netspinaddables{\ep}{\al}}\spn\al$. 
\Cref{lessdd} gives $\bdiffondiag{d}{\be}{\al} \leq \spinaddsondiag{d}{\al}$ for every $d\equiv2\ep\ppmod4$.
But the sum over all $d$ of $\bdiffondiag{d}{\be}{\al}$ is $|\be| - |\al|$; this is  the net number of nodes added by \(\runnerswap{\eps}{\netspinaddables{\ep}{\al}}\), which is \(\netspinaddables{\ep}{\al} = \sum_{d \equiv 2\eps \ppmod{4}} \spinaddsondiag{d}{\al}\).
Thus we have equality for every $d$, which forces $\be=\bswp\al\ep$.

It remains to determine the coefficient of $\spn{\bswp\al\ep}$ in $\runnerswap{\eps}{\netspinaddables{\ep}{\al}}\spn\al$.
To do this, let $\Ga$ be the set of $\ga\in\scrd$ such that $\al$ and $\bswp\al\ep$ can both be obtained by adding \nds\ep to $\ga$. The strict partition \(\al \cap \bswp\al\ep\), obtained from \(\al\) by removing the pairs of \esprms\ep in adjacent columns of \(\al\) (and the \esprm\ep in column 1, if it exists), is one such partition.
Setting \(a = \abs{\al} - \abs{\al \cap \bswp\al\ep}\), we have from \Cref{spinbranch}
\[
    \ip{ e_\eps^{(a)}\spn{\al} }{ \spn{\al\cap\bswp{\al}{\ep}} }
    = \ip{ f_\eps^{(a+\netspinaddables{\ep}{\al})} \spn{\al\cap\bswp{\al}{\ep}} }{ \spn{\bswp{\al}{\ep}} }
    = 1.
\]
Thus the contribution from \(\al \cap \bswp\al\ep\) to the coefficient of \(\spn{\bswp{\al}{\ep}}\) is \((-1)^{a+\netspinaddables{\ep}{\al}}\).

To describe \(\Ga\) fully, let $D$ be the set of all $d\equiv2\ep\ppmod4$ such that $\al$ has both a \sprm and a \spam in columns $d,d+1$.
For every $E\subseteq D$, define $\ga_E$ to be the unique strict partition such that
\begin{itemize}
\item
$\ga_E$ agrees with $\al\cap\bswp\al\ep$ except in columns $d,d+1$ for $d\in E$, and
\item
$\bdiffondiag{d}{\ga_E}{\al} = -1$  for all $d\in E$ (that is, the \sprm in column \(d\) or \(d+1\) has been removed in \(\ga_E\) for all $d\in E$).
\end{itemize}
Then $\Ga=\lset{\ga_E}{E\subseteq D}$, with \(\ga_\emptyset = \al \cap \bswp{\al}{\ep}\) and $|\ga_E|=|\al\cap\bswp\al\ep|-|E|$ for each $E$.
Furthermore, \Cref{spinbranch} gives
\[
\ip{e_{\eps}^{(a+|E|)}\spn\al}{\spn{\ga_E}}
    = \ip{f_{\eps}^{(a+|E|+c)}\spn{\ga_E}}{\spn{\bswp\al\ep}}
    =\qt^{|E|}.
\]
Summing over $E$, we obtain
\begin{align*}
\ip{\runnerswap{\eps}{\netspinaddables{\ep}{\al}}\spn\al}{\spn{\bswp\al\ep}}
    &= (-1)^{a+\netspinaddables{\ep}{\al}} \sum_{E\subseteq D}(-2)^{|E|} \\
    &= (-1)^{a+\netspinaddables{\ep}{\al}} \sum_{k=0}^{\abs{D}} \binom{\abs{D}}{k} (-2)^k \\
    &= (-1)^{a+\netspinaddables{\ep}{\al}} (1-2)^{\abs{D}} \\
    &= (-1)^{a+\netspinaddables{\ep}{\al}+|D|}.\qedhere
\end{align*}
\end{pf}

In view of \cref{spin_characters:action_of_runner-swapping_functor}, we would like to understand the partition \(\bswp{\al}{\eps}\).

\begin{lemma}\label{describebsw}
Suppose $\al\in\scrd$ and $\ep\in\{0,1\}$. Then $\bswp\al\ep$ is obtained from $\al$ by:
\begin{itemize}
\item
replacing each part $d\equiv2\ep-1\ppmod4$ with $d+2$,
\item
replacing each part $d>1$ such that $d\equiv2\ep+1\ppmod4$ with $d-2$,
\item
inserting the part $1$, if $\ep=0$ and $1\notin\al$,
\item
removing the part $1$, if $\ep=0$ and $1\in\al$
\end{itemize}
and then reordering parts into decreasing order.
\end{lemma}

\begin{pf}
For each $j>0$ with $j\equiv2\ep\ppmod4$ we can consider the two columns $j,j+1$ of \spre in isolation and verify the result; there are eight cases to check, depending on which of the integers $j-1,j,j+1$ are parts of $\al$.

When $\ep=0$, we additionally need to consider column $1$, and again the result is easy to check.
\end{pf}

\begin{rmk}
\cref{describebsw} can alternatively be described in terms of the $4$-bar-abacus introduced by Bessenrodt and Olsson \cite{bessenrodt97blocksofdoublecovers}: it says that the abacus display for $\bswp\al\ep$ can be obtained from the abacus display for $\al$ by swapping runners $1$ and $3$. This gives further justification to the name ``\rsf''. We refer the reader to \cite[\S3]{bessenrodt97blocksofdoublecovers} for more details on the $4$-bar-abacus.
\end{rmk}

Now we can give an analogue of \cref{lemma:computing_swp_partition_on_cores_and_quotients} for strict partitions.

\begin{lemma}
\label{lemma:computing_bswp_partition_on_abacus}
\begin{enumerate}[(i), beginthm]
    \item
Let \(a \geq 0\) and let \(\eps \in \set{0,1}\).
Then
\[
    \netspinaddables{\eps}{\btwoc{a}} = \begin{cases}
        a+1 & \text{ if \(\eps \equiv a \ppmod{2}\);} \\
        -a & \text{ if \(\eps \not\equiv a \ppmod{2}\),} 
    \end{cases}
\qquad
\bswp{\btwoc{a}}{\eps} = \begin{cases}
        \btwoc{a+1} & \text{ if \(\eps \equiv a \ppmod{2}\);} \\
        \btwoc{a-1} & \text{ if \(\eps \not\equiv a \ppmod{2}\),} 
    \end{cases}
\]
unless \(a=0\) and \(\eps=1\), in which case \(\netspinaddables{1}{\btwoc{0}} = 0\) still holds but \(\bswp{(\btwoc{0})}{1} = \btwoc{0}\).
    \item
Let \(\ga\in\scrd\) with all parts odd, let \(\eta\in\scrd\) and let \(\eps \in \set{0,1}\).
Then
\[
    \bswp{(\gamma \sqcup 2\eta)}{\eps} = \bswp{\gamma}{\eps} \sqcup 2\eta
\]
and hence \(\netspinaddables{\eps}{(\gamma \sqcup 2\eta)} = \netspinaddables{\eps}{\gamma}\).
\end{enumerate}
\end{lemma}

\begin{proof}
Part (i) is routine, and part (ii) follows immediately from \cref{describebsw}.
\end{proof}

\Cref{runner-swap_application_on_spin_charsgeneral} now follows by combining \Cref{spin_characters:action_of_runner-swapping_functor,lemma:computing_bswp_partition_on_abacus}.

\begin{rmk}
It is interesting to ask what we get when we apply the function $\runnerswap{\eps}{c}$ for values of $c$ other than \(\netaddables\ep\la\) or \(\netspinaddables{\ep}{\al}\).

First consider a partition $\la$.
\cref{ldd} (and the argument at the start of the proof of \cref{characters:action_of_runner-swapping_functor}) shows that if $c > \netaddables\ep\la$, then $\runnerswap{\eps}{c}\spe\la=0$.
On the other hand, if \(r_\eps\) denotes the number of removable $\ep$-nodes of $\la$, then clearly \(\runnerswap{\eps}{c} \spe\la = 0\) for \(c < -r_\eps\), whilst $\runnerswap{\ep}{-r_\eps}\spe\la = \pm\spe{\la^{-\ep}}$, where $\la^{-\ep}$ is the partition obtained by removing all the removable $\ep$-nodes from $\la$.
Meanwhile if $-r_\eps < c < \netaddables\ep\la$, it is not hard to show that $\runnerswap{\eps}{c}\spe\la$ is a linear combination of two or more different characters.

Now consider $\al\in\scrd$. If $c > \netspinaddables{\ep}{\al}$, then \cref{lessdd} shows that $\runnerswap{\eps}{c}\spn\al=0$. On the other hand, if \(r_\eps\) denotes the number of \esprms\ep of $\al$, then \(\runnerswap{\eps}{c} \spn{\al} = 0\) for \(c < -r_\eps\), whilst $\runnerswap{\ep}{-r_\eps}\spn\al$ is a non-zero multiple of $\al^{-\ep}$, where $\al^{-\ep}$ is the partition obtained by removing all the \esprms\ep from $\al$.
However, this multiple need not be $\pm1$; for example, $\runnerswap{1}{-1}\spn{(2)} = \sqrt2\spn{(1)}$.
Meanwhile if $- r_\eps < c < \netspinaddables{\ep}{\al}$, then $\runnerswap{\eps}{c}\spn\al$ is a linear combination of two or more characters, except in some cases where $c = \netspinaddables{\ep}{\al} - 1$. We leave the reader to work out the details.
\end{rmk}

\section{Quotient-redistributing functions}
\label{sec:quotient-redistributing}

The purpose of this section is to describe the action of the \qrf on certain RoCK characters and spin characters.
We recall the definition of the function from \Cref{sec:fsas_implies_proportional-if}.

\quotientredistributing*

\subsection{Action on characters}
\label{subsec:quotient-redistributing_functor_on_chars}

We will apply the \qrf to characters in RoCK blocks. We say that a partition $\la$ is RoCK if the character $\spe\la$ lies in a RoCK block; that is, if $w\ls c+1$, where $\la$ has $2$-weight $w$ and $2$-core $(c,c-1,\dots,1)$. In fact the structure of RoCK partitions is quite easy to understand; they were first analysed by James and Mathas in \cite{jamesmathas1996} (where they were called ``partitions with enormous $2$-core'').  Recall that if $\la$ is a partition, then we write $\bipla=(\la^{(0)},\la^{(1)})$ for its $2$-quotient. If $\la$ is RoCK, then $\la$ is obtained from its $2$-core by adding $\la^{(1)}_i$ horizontal dominoes (i.e.~pairs of orthogonally adjacent nodes) in row $i$, and adding $(\la^{(0)})'_i$ vertical dominoes in column $i$, for each $i$. The RoCK condition ensures that the horizontal and vertical dominoes do not meet.
This structure is illustrated in \Cref{fig:partition_correspondence}.

\Yfillopacity1

\Yboxdim{9pt}

\begin{figure}[ht]
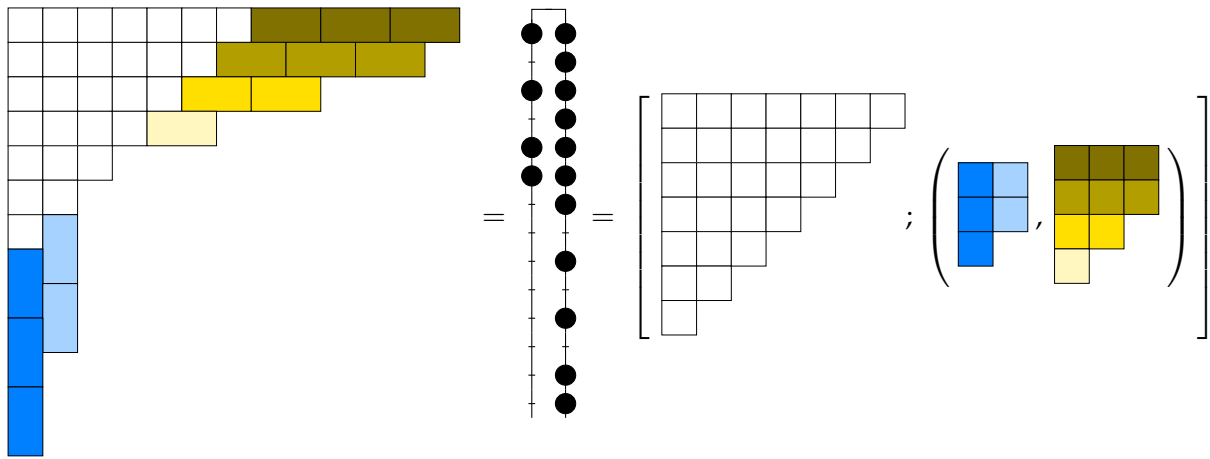

    \centering
    \[
    \gyoung(%
    ;;;;;;;;;;;;;;!\YvdB_2_2_2,!\wht%
    ;;;;;;;;;;;;;!\YdB_2_2_2,!\wht%
    ;;;;;;;;;;;;!\YcB_2_2,!\wht%
    ;;;;;;;;;;;!\YpB_2,!\wht%
    ;;;;;;;;;;,%
    ;;;;;;;;;,%
    ;;;;;;;;,%
    ;;;;;;;,%
    ;;;;;;,%
    ;;;;;,%
    ;;;;,%
    ;;;,%
    ;;,%
    ;!\YpA|2:,%
    !\YcA|2:,!\YpA%
    :|2:,%
    !\YcA|2,%
    |2,%
    )%
    \
    \;=\;
    \sabacus(1,lr,bb,nb,bb,nb,bb,bb,nb,nb,nb,nb,nb,nb,nb,nb,nn,nb,nn,nb,nn,nb,nb)
    \;=\;
    \left[\,
    \;\gyoung(%
    :,%
    ;;;;;;;;;;;;;;,%
    ;;;;;;;;;;;;;,%
    ;;;;;;;;;;;;,%
    ;;;;;;;;;;;,%
    ;;;;;;;;;;,%
    ;;;;;;;;;,%
    ;;;;;;;;,%
    ;;;;;;;,%
    ;;;;;;,%
    ;;;;;,%
    ;;;;,%
    ;;;,%
    ;;,%
    ;,%
    )%
    ;\;
    \left(
    \;
    \gyoung(!\YcA;!\YpA;,!\YcA;!\YpA;,!\YcA;)
    \;,\;
    \gyoung(!\YvdB;;;,!\YdB;;;,!\YcB;;,!\YpB;)
    \;
    \right)
    \,\right]
    \]
    \caption{The RoCK partition \((20,19,16,13,10,9,8,7,6,5,4,3,2^5,1^3)\), depicted as a Young diagram, on an abacus, and in terms of its \(2\)-core and \(2\)-quotient. Observe the correspondence between nodes in the components of the \(2\)-quotient and ``dominoes'' in the Young diagram indicated by the colouring.}
    \label{fig:partition_correspondence}
\end{figure}

We introduce some notation to help us describe the action of the \qrf on RoCK partitions.

\begin{defn}
Given bipartitions \(\bipla\) and \(\bipmu\), write \(\bipla \addtoget \bipmu\) if $\mu^{(0)}$ can be obtained from $\la^{(0)}$ by adding nodes in distinct columns, and $\mu^{(1)}$ can be obtained from $\la^{(1)}$ by adding nodes in distinct rows.
\end{defn}

Now the following result comes from the branching rule (after some translation of notation, part (ii) is effectively a special case of \cite[Lemma 3.1]{ct}).

\begin{lemma}\label{partitions:composed_functors_remove_dominoes}
Let \(\la\) be a partition with \(2\)-core \(\nu\) and \(2\)-weight \(w\).
Let \(r \geq 0\).
Let $\ep$ be the residue of $\len\nu+1$ modulo $2$. 
\begin{enumerate}[(i)]
\item
Suppose $w\ls\len\nu+1$ (that is, that \(\la\) is RoCK).
\[
\ip{e_{\bep}^{(r)}e_{\ep}^{(r)}\spe\la}{\spe\mu}
    = \begin{cases}
        1   &\parbox{318pt}{if \(\mu\) is a partition with \(2\)-core \(\nu\) and \(2\)-weight \(w-r\), and $\bipmu\addtoget\bipla$;}
        \\
        0   &\text{otherwise.}
\end{cases}
\]
\item
Suppose $w+r\ls\len\nu+1$ (that is, that any partition obtained from \(\la\) by adding \(r\) dominoes is RoCK).
\[
\ip{f_{\ep}^{(r)}f_{\bep}^{(r)}\spe\la}{\spe\mu}
    = \begin{cases}
    1   &\parbox{318pt}{if \(\mu\) is a partition with \(2\)-core \(\nu\) and \(2\)-weight \(w+r\), and $\bipmu\removetoget\bipla$;}
    \\
    0   &\text{otherwise.}
\end{cases}
\]
\end{enumerate}
\end{lemma}

Observe that the functors \(e_{\bep}^{(r)}e_{\ep}^{(r)}\) and \(f_{\ep}^{(r)}f_{\bep}^{(r)}\) modify only the \(2\)-quotient of the partition labelling the character they act on, and not its \(2\)-core (in fact this is true even if the partition is not RoCK).

We deduce a description of the action of the quotient-redistributing functor.

\begin{defn}
Given \bips $\bipla$ and $\bipmu$, let
\[
\interm{\bipla}{\bipmu} = \setbuild{\bipnu}{\bipla \removetoget \bipnu \addtoget \bipmu}.
\]
\end{defn}

The following proposition follows immediately from the definitions and \cref{partitions:composed_functors_remove_dominoes}.

\begin{propn}
\label{prop:quotient_redistributing_functor_in_terms_of_interm_sum}
Let \(\la\) be a partition with \(2\)-core \(\nu\) and \(2\)-weight \(w\).
Let \(d \in \bbz\).
Let $\ep$ be the residue of $\len\nu+1$ modulo $2$.
Suppose \(\max\{w,w+d\} \leq \len{\nu} + 1\) (that is, that \(\la\), and any partition obtained from \(\la\) by adding \(d\) dominoes if \(d > 0\), is RoCK).
Then
\[
    \quotred{\eps}{d} \spe{\la} = \sum_{\mu} \left( \sum_{\bipnu \in \interm{\bipla}{\bipmu}} (-1)^{\abs{\bipmu\ydsm\bipnu}} \right) \spe{\mu}
\]
where the sum is over all partitions \(\mu\) with \(2\)-core \(\nu\) and \(2\)-weight \(w-d\).
\end{propn}

It remains, then, to consider \(\interm{\bipla}{\bipmu}\) for the desired bipartitions. Given $r,s\gs0$, we write $\biptwoc rs$ for the bipartition $(\twoc r,\twoc s)$.

\begin{lemma}
\label{lemma:interm_sum_for_bicores}
Let \(s > r \geq 0\).
Let \(\bipmu\) be a bipartition of \(\abs{\biptwoc{r+1}{s-1}}\).
Then
\[
    \sum_{\bipnu \in \interm{\biptwoc{r}{s}}{\bipmu}} (-1)^{\abs{\bipmu\ydsm\bipnu}}
        = \begin{cases}
    (-1)^{r+1} & \text{ if \(\bipmu=\biptwoc{r+1}{s-1}\);} \\
    0 & \text{otherwise.}
    \end{cases}
\]
\end{lemma}

\begin{proof}
Let \(\bipla\) be any bipartition and let \(\bipnu  \in \interm{\bipla}{\bipmu}\).
Suppose there exists a removable node \(x\) in \(\bipla\) which is also present (though not necessarily removable) in \(\bipmu\), and furthermore that if \(x\) lies in \(\bipmu^{(0)}\) then there is no node directly below it in \(\bipmu^{(0)}\), while if \(x\) lies in \(\bipmu^{(1)}\) then there is no node directly to its right in \(\bipmu^{(1)}\).
Observe that if \(x\) is present in \(\bipnu\), then it is removable in \(\bipnu\), and furthermore the column or row (according to whether \(x\) lies in the \(0\)th or \(1\)st component) containing \(x\) is unchanged between \(\bipla\), \(\bipnu\) and \(\bipmu\), so that \(\bipnu \nodesm x\) also lies in \(\interm{\bipla}{\bipmu}\).
On the other hand, if \(x\) is not present in \(\bipnu\), then \(x\) is addable to \(\bipnu\), and \(\bipnu + x\) lies in \(\interm{\bipla}{\bipmu}\).
Thus we can define an involution on the set \(\interm{\bipla}{\bipmu}\) by adding or removing the node \(x\). This involution also changes the number of nodes added to reach \(\bipmu\) from \(\bipnu\) by exactly \(1\), and therefore reverses the sign \((-1)^{\abs{\bipmu\ydsm\bipnu}}\) in the sum \(\sum_{\bipnu \in \interm{\bipla}{\bipmu}} (-1)^{\abs{\bipmu\ydsm\bipnu}}\). Thus the contributions cancel out, and the sum vanishes.

Now consider \(\bipla = \biptwoc{r}{s}\), which contains a removable node in every nonempty row and column.
The previous paragraph shows that our sum vanishes unless \(\bipmu\) is obtained from \(\biptwoc{r}{s}\) in such a way that every nonempty column of \(\twoc{r}\) and every nonempty row of \(\twoc{s}\) has a node either added or removed.
In this case, the number of added and removed nodes amongst the nonempty rows and columns is equal to \(r+s\), and so the change in the number of nodes amongst these rows and columns is congruent to \(r+s\) modulo \(2\).
But the total change in the number of nodes going from $\biptwoc rs$ to $\mu$ is
\[
|\mu|-|\biptwoc rs|=|\biptwoc{r+1}{s-1}|-|\biptwoc rs|=-(s-r-1).
\]
So there are more nodes added or removed than in just the nonempty rows or columns.
The only possibilities are the addition of nodes in the first row (i.e.\ the empty columns) of \(\twoc{r}\) and the addition of nodes in the first column (i.e.\ the empty rows) of \(\twoc{s}\).

If nodes are added in the first column of \(\twoc{s}\), this means the final nonempty row in \(\twoc{s}\) cannot have a node removed, so a node must be added; then the row above that cannot have a node removed, and so on; thus every nonempty row of \(\twoc{s}\) has a node added.
Then the total change in number of nodes from \(\biptwoc{r}{s}\) to \(\bipmu\) is at least \((s+1)-r>0\), a net gain, contradicting the assumed size of \(\bipmu\).
Thus nodes must be added in the first row of \(\twoc{r}\), and hence a node must be added in every nonempty column of \(\twoc{r}\); then for the total change in number of nodes to be \(-(s-r-1)\) there must be exactly \(1\) node added to the first row of \(\twoc{r}\) and a node removed from every nonempty row of \(\twoc{s}\).
This precisely yields \(\biptwoc{r+1}{s-1}\).

It remains to see that \(\sum_{\bipnu \in \interm{\biptwoc{r}{s}}{\biptwoc{r+1}{s-1}}} (-1)^{\abs{\biptwoc{r+1}{s-1}\ydsm\bipnu}}\) equals \((-1)^{r+1}\).
Indeed, \(\biptwoc{r}{s-1}\) is the unique element of \(\interm{\biptwoc{r}{s}}{\biptwoc{r+1}{s-1}}\), and \(\abs{\biptwoc{r+1}{s-1} \ydsm \biptwoc{r}{s-1}} = r+1\).
\end{proof}

Combining \Cref{prop:quotient_redistributing_functor_in_terms_of_interm_sum,lemma:interm_sum_for_bicores} allows us to deduce the action on the characters required in \Cref{sec:fsas_implies_proportional-if}, restated below.

\quotredapplicationchars*

\subsection{Action on spin characters}
\label{subsec:quotient-redistributing_functor_on_spin_chars}

Now we come to spin characters, where we also focus on RoCK blocks. We say that $\al\in\scrd$ is RoCK if the character $\spn\al$ lies in a RoCK block.

We begin by looking at the structure of \fsemic RoCK strict partitions. Part (i) of the following \lcnamecref{basicspinrock} is immediate; part (ii) is given in \cite[Corollaries 5.3 \& 5.5]{fayers18spin2}.

\begin{lemma}\label{basicspinrock}
Let \(\al\in\scrd\) with \fbc \(\ga\).
\begin{enumerate}[(i)]
    \item
\(\al\) is \fsemic \iff \(\al = (\ga + 4\si) \sqcup 2\eta\) for some \(\si \in \scrp\), \(\eta \in \scrd\).
    \item 
\(\al\) is RoCK \iff \(\al = (\ga + 4\si) \sqcup 2\eta\) for some \(\si \in \scrp\), \(\eta \in \scrd\) with \(2\abs{\si} + \abs{\eta} \leq \len{\dbl{\ga}} +1\).
\end{enumerate}
\end{lemma}


It turns out that the action of the \qrf can easily be described in terms of the decomposition given in \cref{basicspinrock}. 

We begin with a spin analogue of \cref{partitions:composed_functors_remove_dominoes}. 
We continue to use the notation $\addtoget$ from \cref{subsec:quotient-redistributing_functor_on_chars}, and define the following parameter.

\begin{defn}
Given $\eta,\theta\in\scrd$ with $\eta\supseteq\theta$, define \(\kom\eta\theta\) to be the number of values $c\gs1$ such that exactly one of columns $c$ and $c+1$ contains a node of \(\eta\ydsm\theta\).
\end{defn}

We note that $\kom\eta\theta$ can also be characterised as the number of integers which are parts of $\eta$ or $\theta$ but not both.

\begin{propn}[{\cite[Proposition 5.6 and its Frobenius reciprocal]{fayers20spin2alt}}]
\label{rockbranch}
Suppose \(\al\in\scrd\) with \fbc \(\ga\) and $4$-bar-weight \(w\).
Write \(\al = (\ga+4\si)\sqcup2\eta\), where \(\si \in \scrp\), \(\eta \in \scrd\) are such that $2\abs\si+\abs\eta=w$.
Let \(r \geq 0\).
Let $\ep$ be the residue of $\len{\dbl\ga}$ modulo $2$.
\begin{enumerate}[(i)]
\item
Suppose $w\ls\len{\dbl{\ga}}+1$.
\[
\ip{e_{\bep}^{(r)}e_{\ep}^{(r)}\spn\al}{\spn\be}
    =
    \begin{cases}
    \sqrt2^{\,\kom{\eta}{\theta}}
    &\parbox{254pt}{if \(\be\in\scrd\) with \fbc \(\ga\) and $4$-bar-weight \(w-r\), expressed as \(\be = (\ga + 4\tau) \sqcup 2\theta\) where \((\theta,\tau)\addtoget(\eta,\si)\);}
    \\
    0&\text{otherwise.}
    \end{cases}
\]
\item
Suppose $w+r\ls\len{\dbl{\ga}}+1$.
\[
\ip{f_{\ep}^{(r)}f_{\bep}^{(r)}\spn\al}{\spn\be}
    =
    \begin{cases}
    \sqrt2^{\,\kom{\eta}{\theta}}
    &\parbox{254pt}{if \(\be\in\scrd\) with \fbc \(\ga\) and $4$-bar-weight \(w+r\), expressed as \(\be = (\ga + 4\tau) \sqcup 2\theta\) where \((\theta,\tau)\removetoget(\eta,\si)\);}
    \\
    0&\text{otherwise.}
    \end{cases}
\]
\end{enumerate}
\end{propn}

In view of this, it is natural to make the following definitions.

\begin{defn}
Given $\eta,\theta\in\scrd$, let
\begin{align*}
\intermzero{\eta}{\theta} &= \setbuild*{\ze \in \scrd \,}{\, \parbox{173pt}{both \(\eta\) and \(\theta\) can be obtained from \(\ze\) by adding nodes in distinct columns}}
\intertext{%
and given $\si,\tau\in\scrp$, let %
}
\intermone{\si}{\tau} &= \setbuild*{\phi \in \scrp \,}{\, \parbox{173pt}{both \(\si\) and \(\tau\) can be obtained from \(\phi\) by adding nodes in distinct rows}}.
\end{align*}
Let \(\intermsize{\si}{\tau} = \abs{\intermone{\si}{\tau}}\).
\end{defn}

\begin{propn}
\label{prop:quotient_redistributing_functor_in_terms_of_interm_sum_spin_case}
Suppose \(\al\in\scrd\) with \fbc \(\ga\) and $4$-bar-weight \(w\), and \(d \in \bbz\). Let $\ep$ be the residue of $\len{\dbl\ga}+1$ modulo $2$.
Suppose \(\max\{w,w+d\} \leq \len{\dbl{\ga}} + 1\).
Write \(\al =(\ga+4\si) \sqcup 2\eta\), where \(\eta \in \scrd\) and $\si\in\scrp$ with $2\abs\si+\abs\eta=w$.
Then
\[
\quotred{\eps}{d} \spn{\al}
    = \sum_{\be}
        \left(
        \intermsize{\si}{\tau}
        \sum_{\ze \in \intermzero{\eta}{\theta}}
            (-1)^{\abs{\theta\ydsm\ze}}
            \sqrt2^{\,\kom{\eta}{\ze}+\kom{\theta}{\ze}}
        \right)
    \spn{\be}
\]
where the sum is over all \(\be\in\scrd\) with \fbc \(\ga\) and $4$-bar-weight \(w+d\) expressed as \(\beta = (\ga+4\tau)\sqcup 2\theta\).
\end{propn}

\begin{proof}
It follows from the definitions and \Cref{rockbranch} that
\begin{align*}
\quotred{\eps}{d} \spn{\al}
    &= \sum_{\be}
        \left(
        \sum_{(\ze,\phi) \in \intermzero{\eta}{\theta} \times \intermone{\si}{\tau}}
            (-1)^{\abs{\theta\ydsm\ze}+2\abs{\tau\ydsm\phi}}
            \sqrt2^{\,\kom{\eta}{\ze}+\kom{\theta}{\ze}}
        \right)
    \spn{\be}
\end{align*}
where the sum is over all \(\be\in\scrd\) with \fbc \(\ga\) and $4$-bar-weight \(w+d\) expressed as \(\beta = (\ga+4\tau)\sqcup 2\theta\).
This simplifies to the expression in the \lcnamecref{prop:quotient_redistributing_functor_in_terms_of_interm_sum_spin_case} by noting that the summands of the inner sum do not depend on \(\phi \in \intermone{\si}{\tau}\).
\end{proof}

Our next task is to evaluate the sum over \(\intermzero{\eta}{\theta}\) appearing in \Cref{prop:quotient_redistributing_functor_in_terms_of_interm_sum_spin_case}, which we denote
\[
\isum\eta\theta=\sum_{\ze \in \intermzero{\eta}{\theta}}(-1)^{\abs{\theta\ydsm\ze}}\sqrt2^{\,\kom{\eta}{\ze}+\kom{\theta}{\ze}}.
\]

\begin{propn}\label{prop:tabrim_evaluation}
Suppose $\eta,\theta\in\scrd$. Let $d=\abs\theta-\abs\eta$, and let $r$ be the number of parts of $\eta$ greater than $|d|$. Then
\[
\isum\eta\theta=
\begin{cases}
(-1)^r\sqrt2&\text{if 
$\theta=\eta\partsm(-d)$;}
\\
(-1)^r&\text{if $\theta=\eta$;}
\\
(-1)^{r+d}\sqrt2&\text{if 
$\theta=\eta\sqcup(d)$;}
\\
0&\text{otherwise.}
\end{cases}
\]
\end{propn}

We prove \cref{prop:tabrim_evaluation} by iteratively removing columns from \(\eta\) and \(\theta\).
Fix \(\eta\) and \(\theta\), and for \(c \geq 1\), let \(\coldiff{c} = \eta'_c - \theta'_c\) be the difference between the length of the \(c\)th columns in \(\eta\) and \(\theta\).
For any $\ze\in\scrd$, let $\colrem\ze$ be the strict partition obtained by deleting the first column. 

The iterative step hinges on the following \lcnamecref{intermind}.

\begin{lemma}\label{intermind}
Suppose $\eta,\theta\in\scrd$. 
Then
\[
\isum\eta\theta=
\begin{cases}
\phantom{-\sqrt{2}}\isum\ceta\ctheta
&\text{if \(\coldiff{1} = \coldiff{2} = 1\), or if \(\coldiff{1} = \coldiff{2} = 0\) and \(\eta'_1 - \eta'_2 = 0\);}
\\
\phantom{\sqrt{2}}{-}\isum\ceta\ctheta
&\text{if \(\coldiff{1} = \coldiff{2} = -1\), or if \(\coldiff{1} = \coldiff{2} = 0\) and \(\eta'_1 - \eta'_2 = 1\);}
\\
\phantom{-}\sqrt2\isum\ceta\ctheta
&\text{if \(\coldiff{1} = 1\) and \(\coldiff{2} = 0\);}
\\
-\sqrt2\isum\ceta\ctheta
&\text{if \(\coldiff{1} = -1\) and \(\coldiff{2} = 0\);}
\\
\phantom{-\sqrt{2}B}0
&\text{if \(\coldiff{1} = 0\) and \(\coldiff{2} \neq 0\), or if \(\abs{\coldiff{c}}>1\) for any \(c\).}
\end{cases}
\]
\end{lemma}

\begin{pf}
Observe that for any \(\zeta \in \iet\), we must have \(\eta'_c, \theta'_c \in \set{\zeta'_c, \zeta'_c + 1}\) for all \(c\).
In particular, if \(\abs{\coldiff{c}} > 1\) for any \(c\) then \(\iet = \emptyset\) and \(\isum{\eta}{\theta} = 0\) as claimed. For the remainder of the proof, assume \(\abs{\coldiff{c}} \leq 1\) for all \(c\).

For \(\zeta \in \iet\), let $\isummand\ze\eta\theta=(-1)^{\abs{\theta\ydsm\ze}}\sqrt2^{\kom{\eta}{\ze}+\kom{\theta}{\ze}}$ denote the contribution to \(\isum\eta\theta\) from \(\zeta\).
Our strategy is to compare $\isummand\zeta\eta\theta$ and $\isummand\czeta\ceta\ctheta$. 
We claim that $\ze\mapsto\czeta$ defines a surjective function from $\iet$ to $\ciet$, so that our strategy does consider all contributions to \(\isum\ceta\ctheta\).
Indeed, it is clear that if \(\zeta \in \iet\) then $\czeta\in\ciet$.
For surjectivity, consider a strict partition \(\delta \in \ciet\), and observe that \(\delta'_1 \leq \eta'_2 \leq \eta'_1\) and also that \(\eta'_1 \leq \eta'_2 + 1 \leq \delta'_1+2\), and likewise for \(\theta\), so that
\[
    \delta'_1 \leq \eta'_1, \theta'_1 \leq \delta'_1 +2.
\]
Since we assume \(\abs{\coldiff{1}} \leq 1\), we then have that either \(\{\eta'_1, \theta'_1\} \subseteq \{\delta'_1, \delta'_1+1\}\) and hence adjoining a column of length \(\delta'_1\) to \(\delta\) yields an element of \(\iet\), or \(\{\eta'_1, \theta'_1\} \subseteq \{\delta'_1+1, \delta'_1+2\}\) and hence adjoining a column of length \(\delta'_1+1\) to \(\delta\) yields an element of \(\iet\).
(Note also that these are the only possible column lengths that can be adjoined to \(\delta\) to obtain a strict partition, and thus every element of \(\ciet\) has a preimage of size \(1\) or \(2\).)

To compare the contributions \(\isummand\zeta\eta\theta\) and \(\isummand\czeta\ceta\ctheta\), we break down into several cases.

\begin{description}
\item[\(\coldiff{1} = 1\):]
In this case $\ze'_1 = \theta'_1 = \eta'_1 - 1$ for every $\ze\in\iet$, and it follows that $\ze\mapsto\czeta$ is a bijection.
Clearly $\abs{\theta\ydsm\ze}=\abs{\ctheta\ydsm\czeta}$ for every $\ze\in\iet$.
Note that column \(1\) contributes to \(\kom\eta\zeta\) if and only if \(\eta'_2 = \zeta'_2\), while column \(1\) contributes to \(\kom\theta\zeta\) if and only if \(\theta'_2 = \zeta'_2+1\), and thus
\[
\kom\eta\ze=
\begin{cases}
\kom\ceta\czeta+1 &\text{if }\eta'_2=\ze'_2
\\
\kom\ceta\czeta & \text{if }\eta'_2=\ze'_2+1,
\end{cases}
\qquad\kom\theta\ze=
\begin{cases}
\kom\ctheta\czeta&\text{if }\theta'_2=\ze'_2
\\
\kom\ctheta\czeta+1&\text{if }\theta'_2=\ze'_2+1,
\end{cases}
\]
and hence
\[
\kom\eta\ze+\kom\theta\ze=
\begin{cases}
\kom\ceta\czeta+\kom\ctheta\czeta+1&\text{if \(\coldiff{2}=0\);}
\\
\kom\ceta\czeta+\kom\ctheta\czeta&\text{if \(\coldiff{2}=1\).}
\end{cases}
\]
Therefore
\[
\isummand\ze\eta\theta=
\begin{cases}
\sqrt2\isummand\czeta\ceta\ctheta&\text{if \(\coldiff{2} = 0\);}
\\[5pt]
\phantom{\sqrt{2}}\isummand\czeta\ceta\ctheta&\text{if \(\coldiff{2} = 1\).}
\end{cases}
\]
for every $\ze\in\iet$. Summing over $\ze$ gives the claimed values when \(\coldiff{1} = 1\).

\item[\(\coldiff{1} = -1\):]
This is very similar to the previous case; now $\abs{\theta\ydsm\ze}=\abs{\ctheta\ydsm\czeta}+1$ for every $\ze$, and the formul\ae{} from the last case apply with $\eta$ and $\theta$ interchanged and a change of sign in the last formula.

\item[\(\coldiff{1} = 0\) and \(\coldiff{2}\neq0\):]
We consider the case \(\coldiff{2} = 1\); the case \(\coldiff{2} = -1\) is obtained by interchanging \(\eta\) and \(\theta\) in what follows.
Having \(\coldiff{1} = 0\) and \(\coldiff{2} = 1\) implies $\ze'_2 = \theta'_2= \eta'_2 -1$ for every $\ze\in\iet$, while $\ze'_1$ can be either $\eta'_1 = \theta'_1$ or $\eta'_1 - 1 = \theta'_1-1$.
So for each $\de\in\ciet$ there are two partitions $\ze^+,\ze^-\in\iet$ satisfying $\colrem{\ze^+} = \colrem{\ze^-} = \delta$: we obtain $\ze^+$ by adjoining a column of length $\eta'_1$ to $\de$, and $\ze^-$ by adjoining a column of length $\eta'_1-1$.
But now $\abs{\theta\ydsm\ze^-} - \abs{\theta\ydsm\ze^+} = 1$ (and also \(\abs{\eta\ydsm\ze^-} - \abs{\eta\ydsm\ze^+} = 1\)), while
\[
\kom\eta{\ze^-} = \kom\ceta\de+1 = \kom\eta{\ze^+}+1,
\qquad
\kom\theta{\ze^+} = \kom\ctheta\de+1 = \kom\theta{\ze^-}+1,
\]
so that $\isummand{\ze^+}\eta\theta + \isummand{\ze^-}\eta\theta=0$.
This applies for all $\de\in\ciet$, giving $\isum\eta\theta=0$, as required.

\item[\(\coldiff{1} = \coldiff{2} = 0\) and \(\eta'_1 - \eta'_2 = 0\):]
For each $\de\in\ciet$ with $\de'_1=\eta'_2$, there is a unique $\ze\in\iet$ satisfying $\colrem\ze=\de$, obtained by adding a column of length $\eta'_1$ to $\de$. This $\ze$ satisfies
\[
\abs{\theta\ydsm\ze} = \abs{\ctheta\ydsm\de},
\qquad
\kom\eta\ze = \kom\ceta\de,
\qquad
\kom\theta\ze = \kom\ctheta\de
\]
so that $\isummand\ze\eta\theta=\isummand\de\ceta\ctheta$.

For each $\de\in\ciet$ with $\de'_1=\eta'_2-1$, there are two partitions $\ze^+,\ze^-\in\iet$ satisfying $\colrem{\ze^+}=\colrem{\ze^-}=\de$, obtained from $\de$ by adding columns of length $\eta'_1$ and $\eta'_1-1$ respectively. These partitions satisfy
\[
\abs{\theta\ydsm\ze^-} = \abs{\ctheta\ydsm\de}+1 = \abs{\theta\ydsm\ze^+}+1,
\]
\[
\kom\eta{\ze^+} = \kom\ceta\de+1 = \kom\eta{\ze^-}+1,
\qquad
\kom\theta{\ze^+} = \kom\ctheta\de+1 = \kom\theta{\ze^-}+1,
\]
and therefore \(\isummand{\ze^+}{\eta}{\theta} = 2\isummand{\de}{\ceta}{\ctheta}\) and \(\isummand{\ze^-}{\eta}{\theta} = -\isummand{\de}{\ceta}{\ctheta}\).
Hence
\[
\isummand{\de^+}\eta\theta+\isummand{\de^-}\eta\theta
    = \isummand\de\ceta\ctheta.
\]
Now summing over $\iet$ gives $\isum\eta\theta=\isum\ceta\ctheta$, as required.

\item[\(\coldiff{1} = \coldiff{2} =0\) and \(\eta'_1 - \eta'_2 = 1\):]
This case is similar to the previous one.
Now for each $\de\in\ciet$ with $\de'_1=\eta'_2-1$ there is a unique $\ze\in\iet$ such that $\czeta=\de$, and this $\ze$ satisfies $\isummand\ze\eta\theta=-\isummand\de\ceta\ctheta$.
For each $\de$ with $\de'_1=\eta'_2$, there are two partitions $\ze^+,\ze^-\in\iet$ such that $\colrem{\ze^+}=\colrem{\ze^-}=\de$, and these partitions satisfy $\isummand{\ze^+}\eta\theta=\isummand\de\eta\theta$ and $\isummand{\ze^-}\eta\theta=-2\isummand\de\ceta\ctheta$. Summing over $\iet$ then gives $\isum\eta\theta=-\isum\ceta\ctheta$.
This case is illustrated \Cref{fig:Bcalc_example}.
\qedhere
\end{description}
\end{pf}


\begin{figure}[htb]
    \centering
\Yboxdim{5pt}
\renewcommand{\arraystretch}{1.6}
\renewcommand{\arraycolsep}{4pt}
\[
\begin{array}{lccccr}
\toprule
\multicolumn{2}{c}{\zeta} & \abs{\theta\ydsm\ze} & \kom\eta\ze & \kom\theta\ze & \isummand\ze\eta\theta\vphantom{\isummand{\czeta}{\ceta}{\ctheta}}
\\\midrule
(6,5,2,1) & \gyoung(;;;;;;!\YcA;!\wht,;;;;;!\YcA;!\wht,;;!\YcB;!\wht,;)
&1&2&2&-4
\\
(6,5,2) & \gyoung(;;;;;;!\YcA;!\wht,;;;;;!\YcA;!\wht,;;!\YcB;!\wht,!\YcG;)
&2&3&3&8
\\\midrule
(6,5,1) & \gyoung(;;;;;;!\YcA;!\wht,;;;;;!\YcA;!\wht,;!\YcG;!\YcB;!\wht,!\YcG;)
&3&3&1&-4
\\\midrule
(6,4,2,1) & \gyoung(;;;;;;!\YcA;!\wht,;;;;!\YcG;!\YcA;!\wht,;;!\YcB;!\wht,;)
&2&2&4&8
\\
(6,4,2) & \gyoung(;;;;;;!\YcA;!\wht,;;;;!\YcG;!\YcA;!\wht,;;!\YcB;!\wht,!\YcG;!\wht)
&3&3&5&-16
\\\midrule
(6,4,1) & \gyoung(;;;;;;!\YcA;!\wht,;;;;!\YcG;!\YcA;!\wht,;!\YcG;!\YcB;!\wht,!\YcG;)
&4&3&3&8
\\\midrule
(6,3,2,1) & \gyoung(;;;;;;!\YcA;!\wht,;;;!\YcG;;!\YcA;!\wht,;;!\YcB;!\wht,;) 
&3&2&2&-4
\\
(6,3,2) & \gyoung(;;;;;;!\YcA;!\wht,;;;!\YcG;;!\YcA;!\wht,;;!\YcB;!\wht,!\YcG;) 
&4&3&3&8
\\\midrule
(6,3,1) & \gyoung(;;;;;;!\YcA;!\wht,;;;!\YcG;;!\YcA;!\wht,;!\YcG;!\YcB;!\wht,!\YcG;)
&5&3&1&-4
\\
\bottomrule
\end{array}
\quad\ 
\begin{array}{lccccr}
\toprule
\multicolumn{2}{c}{\czeta} & \abs{\ctheta\ydsm\czeta} & \kom\ceta\czeta &  \kom\ctheta\czeta & \isummand\czeta\ceta\ctheta
\\\midrule
\multirow2*{(5,4,1)} & \multirow2*{ \gyoung(;;;;;!\YcA;!\wht,;;;;!\YcA;!\wht,;!\YcB;) }
&\multirow2*1&\multirow2*2&\multirow2*2&\multirow2*{\ensuremath{-4}}
\\
&&&&
\\\midrule
(5,4) & \gyoung(;;;;;!\YcA;!\wht,;;;;!\YcA;,!\YcG;!\YcB;)
&2&3&1&4
\\\midrule
\multirow2*{(5,3,1)}&\multirow2*{ \gyoung(;;;;;!\YcA;!\wht,;;;!\YcG;!\YcA;!\wht,;!\YcB;) }
&\multirow2*2&\multirow2*2&\multirow2*4&\multirow2*8
\\
&&&&
\\\midrule
(5,3) & \gyoung(;;;;;!\YcA;!\wht,;;;!\YcG;!\YcA;,!\YcG;!\YcB;)
&3&3&3&-8
\\\midrule
\multirow2*{(5,2,1)}
& \multirow2*{ \gyoung(;;;;;!\YcA;!\wht,;;!\YcG;;!\YcA;!\wht,;!\YcB;) }
&\multirow2*3&\multirow2*2&\multirow2*2&\multirow2*{\ensuremath{-4}}
\\
&&&&
\\\midrule
(5,2) & \gyoung(;;;;;!\YcA;!\wht,;;!\YcG;;!\YcA;,!\YcG;!\YcB;)
&4&3&1&4
\\
\bottomrule
\end{array}
\]
    \caption{%
A comparison of the calculations of $\isum\eta\theta$ and $\isum\ceta\ctheta$ when $\eta=(7,6,2,1)$ and $\theta=(6,5,3,1)$ (and hence $\ceta=(6,5,1)$ and $\ctheta=(5,4,2)$).
This example falls into the final case (\(\coldiff{1} = \coldiff{2} = 0\) and \(\eta'_1 - \eta'_2 = 1\)) in the proof of \Cref{intermind}.
On the left are the partitions in $\iet$, on the right the partitions in $\ciet$, matched up according to the function $\ze\mapsto\czeta$.
The Young diagrams of the partitions are drawn, with the nodes which must be added to form \(\eta\) or \(\ceta\) shaded blue and the nodes which must be added to form \(\theta\) or \(\ctheta\) shaded yellow (and the nodes which must be added in either case shaded green).}
    \label{fig:Bcalc_example}
\end{figure}

\vspace{3pt}

\begin{proof}[Proof of \cref{prop:tabrim_evaluation}]
Let \(\cdifseq = (\coldiff{1}, \coldiff{2}, \ldots)\) be the sequence of column differences between \(\eta\) and \(\theta\).
Observe that \(\eta = \theta \sqcup (-d)\) if and only if \(\cdifseq = (1, \ldots, 1, 0, \ldots, 0)\) where the number of \(1\)s is \(-d\), that \(\theta = \eta \sqcup (d)\) if and only if \(\cdifseq = (-1, \ldots, -1, 0, \ldots, 0)\), where the number of \(-1\)s is \(d\), and trivially \(\eta = \theta\) if and only if \(\cdifseq = (0, \ldots 0)\).
So we are required to show \(\isum{\eta}{\theta}\) takes the claimed values when \(\cdifseq\) is of one of these forms, and is zero otherwise.

We do so using \Cref{intermind} iteratively.
The fifth case of \Cref{intermind} tells us that if \(\cdifseq\) has any terms strictly greater than \(1\) in absolute value, or if \(\cdifseq\) has any non-zero term following a \(0\), then \(\isum{\eta}{\theta} = 0\).
Meanwhile \(\cdifseq\) cannot contain a \(1\) followed by a \(-1\), for if \(\coldiff{c} = 1\) then \(\eta'_{c+1} \geq \eta'_{c} - 1 = \theta'_c \geq \theta'_{c+1}\) and hence \(\coldiff{c+1} \geq 0\), and similarly \(\cdifseq\) cannot contain a \(-1\) followed by a \(1\).
Thus if \(\isum{\eta}{\theta} \neq 0\), then \(\cdifseq\) is of the form \((1, \ldots, 1, 0, \ldots, 0)\), \((-1, \ldots, -1, 0, \ldots, 0)\) or \((0, \ldots 0)\), where the number of non-zero entries is \(\abs{d}\).

If \(\cdifseq = (0, \ldots, 0)\), the first and second cases of \Cref{intermind} tell us that, as we remove columns from \(\eta\) and \(\theta\), we pick up a sign every time a column length decreases (by \(1\), the only possibly amount by which a column length can decrease in a strict partition). Thus \(\isum{\eta}{\theta} = (-1)^{\len{\eta}} = (-1)^{r}\) in this case, as required.

If \(\cdifseq = (1, \ldots, 1, 0, \ldots, 0)\), the first case of \Cref{intermind} tells us that removing the columns which differ by \(1\) do not alter the value of \(\isum{\eta}{\theta}\) until the last, when the third case tells us that we pick up a factor of \(\sqrt{2}\). Then the remaining columns are those meeting parts greater than \(\abs{d}\), so arguing as in the previous paragraph gives \(\isum{\eta}{\theta} = (-1)^r \sqrt{2}\) as required.
If \(\cdifseq = (-1, \ldots, -1, 0, \ldots, 0)\), a similar argument applies but using the second and fourth cases of \Cref{intermind}, from which we pick up extra signs for each of the \(d\) differing columns, yielding \(\isum{\eta}{\theta} = (-1)^{r+d} \sqrt{2}\) as required.
\end{proof}

Combining \Cref{prop:quotient_redistributing_functor_in_terms_of_interm_sum_spin_case,prop:tabrim_evaluation} then yields the following. (Recall that given partitions \(\si\) and \(\tau\), we write \(\intermsize{\si}{\tau} = \abs{\intermone{\si}{\tau}}\).)

\begin{thm}
\label{thm:action_of_quotient_redistributing_functor_on_arbitrary_RoCK_character}
Suppose \(\al\in\scrd\) has \fbc \(\ga\) and $4$-bar-weight \(w\). Let \(d \in \bbz\). Let $\ep$ be the residue of $\len{\dbl\ga}+1$ modulo $2$. Suppose \(\max\{w,w+d\} \leq \len{\dbl{\ga}} + 1\).
Write \(\al =(\ga+4\si) \sqcup 2\eta\), where \(\eta \in \scrd\) and $\si\in\scrp$ with $2\abs\si+\abs\eta=w$. Let \(r\) be the number of parts of \(\eta\) greater than \(\abs{d}\).
Then
\[
\quotred{\eps}{d} \spn{\al}
    = \begin{cases}
        \hphantom{(-1)^{\lenabove{d}{\eta}+d} \sqrt{2}}\mathllap{(-1)^{\lenabove{-d}{\eta}} \sqrt{2}}
            \ \displaystyle\sum_{\tau} \intermsize{\si}{\tau} \spn{(\ga + 4\tau) \sqcup2(\eta\partsm(-d))}
            & \text{if \(d<0\) and \(-d \in\eta\);} \\[12pt]
        \hphantom{(-1)^{\lenabove{d}{\eta}+d} \sqrt{2}} \mathllap{(-1)^{r}}
            \ \displaystyle\sum_{\tau} \intermsize{\si}{\tau} \spn{(\ga+4\tau) \sqcup 2\eta}
            & \text{if \(d=0\);} \\[12pt]
        (-1)^{\lenabove{d}{\eta}+d} \sqrt{2}
            \ \displaystyle\sum_{\tau} \intermsize{\si}{\tau} \spn{(\ga + 4\tau) \sqcup2(\eta \sqcup (d))}
            & \text{if \(d>0\) and \(d \not\in \eta\);} \\[10pt]
        \hphantom{(-1)^{\lenabove{d}{\eta}+d} \sqrt{2}}\ \ \; 0
            & \text{otherwise;}
    \end{cases}
\]
where in each case the sum is over partitions \(\tau\) with \(\abs\tau=\abs{\si}\).
\end{thm}

Unfortunately it does not seem to be possible to give a simpler expression for $\intermsize\si\tau$. Nevertheless, we can deduce the result we need.

\begin{cory}
\label{cory:action_of_quotient_redistributing_functor_on_4stepped_odd_parts_RoCK_spin_chars}
Suppose \(\al\in\scrd\) with \fbc \(\ga\) and suppose \(\al = \ga \sqcup 2\eta\) for some \(\eta\in\scrd\).
Let \(d \in \bbz\).
Let $\ep$ be the residue of $\len{\dbl\ga}+1$ modulo $2$.
Suppose \(\max\{\abs{\eta}, \abs{\eta}+d\} \leq \len{\dbl{\ga}}+1\).
Then
\[
\quotred{\eps}{d} \spn{\al}
    = \begin{cases}
        \hphantom{(-1)^{\lenabove{d}{\eta}+d} \sqrt{2}}\mathllap{(-1)^{\lenabove{-d}{\eta}} \sqrt{2}}
            \spn{\ga \sqcup2(\eta\partsm(-d))}
            & \text{if \(d<0\) and \(-d \in\eta\);} \\[5pt]
        \hphantom{(-1)^{\lenabove{d}{\eta}+d} \sqrt{2}} \mathllap{(-1)^{r}}
            \spn{\ga \sqcup2\eta}

            & \text{if \(d=0\);} \\[5pt]
        (-1)^{\lenabove{d}{\eta}+d} \sqrt{2} \spn{\ga \sqcup2(\eta\sqcup(d))}
            & \text{if \(d>0\) and \(d \not\in \eta\);} \\[3pt]
        \hphantom{(-1)^{\lenabove{d}{\eta}+d} \sqrt{2}}
        \; 0
            & \text{otherwise.}
    \end{cases}
\]
\end{cory}
\begin{proof}
Set \(\si = \emptyset\) in \Cref{thm:action_of_quotient_redistributing_functor_on_arbitrary_RoCK_character}, which forces \(\tau = \emptyset\) and hence \(\intermsize{\si}{\tau} = 1\).
\end{proof}

We can now deduce the case required in our application, restated below.

\quotredapplicationspinchars*

\begin{proof}
This follows from \Cref{cory:action_of_quotient_redistributing_functor_on_4stepped_odd_parts_RoCK_spin_chars} by observing that \((\twoc{r}+\twoc{s}) \partsm (-(s-r-1)) = \twoc{r+1}+\twoc{s-1}\).
\end{proof}

The results required in \Cref{sec:fsas_implies_proportional-if} are now all proved, and thus our main theorem is established.

\end{document}